\documentclass[12pt]{amsart}
\usepackage{amsmath,amsfonts,amssymb,amsthm,a4wide,url,accents,mathscinet,mathdots,tensor,enumitem}
\usepackage[all]{xy}

\newtheorem{theorem}{Theorem}[section]
\newtheorem{lemma}[theorem]{Lemma}
\newtheorem{definition}[theorem]{Definition}
\newtheorem{proposition}[theorem]{Proposition}
\newtheorem{corollary}[theorem]{Corollary}
\newtheorem{conjecture}[theorem]{Conjecture}
\theoremstyle{remark}
\newtheorem{remark}[theorem]{Remark}
\newtheorem{example}[theorem]{Example}

\begin{document}

\newcommand{\C}{\mathbb{C}}
\newcommand{\R}{\mathbb{R}}
\newcommand{\Z}{\mathbb{Z}}
\newcommand{\cell}{C}
\newcommand{\Comp}{\mathfrak{C}}
\newcommand{\xvec}{\mathsf{x}}
\newcommand{\xbasis}{x}
\newcommand{\OO}{\mathcal{O}}
\newcommand{\bs}{\backslash}
\newcommand{\yy}{\mathfrak{g}}
\newcommand{\xx}{\mathfrak{h}}
\newcommand{\tr}{\operatorname{tr}}
\newcommand{\pr}{\operatorname{pr}}
\newcommand{\gen}{\operatorname{gen}}
\newcommand{\sll}{\mathfrak{sl}}
\newcommand{\desl}{D}
\newcommand{\rmv}{\Omega}                                      
\newcommand{\CO}{\mathcal{O}}                                
\newcommand{\Matk}{\mathcal{M}_k}                              
\newcommand{\Graph}{\mathcal{G}}
\newcommand{\IrrS}{\Irr^\square}
\newcommand{\spltng}{splitting}
\newcommand{\wtness}{double socle}
\newcommand{\id}{\operatorname{id}}                            
\newcommand{\JH}{\operatorname{JH}}                            
\newcommand{\GH}{\mathbb{H}}                                   
\newcommand{\IH}{\mathcal{H}}                                  
\newcommand{\End}{\operatorname{End}}
\newcommand{\smth}{\operatorname{sm}}
\newcommand{\Ker}{\operatorname{Ker}}
\newcommand{\commvar}{\mathfrak{X}}
\newcommand{\Lieg}{\mathfrak{g}}
\newcommand{\GLS}{(GLS)}
\newcommand{\zz}{\mathfrak{z}}
\newcommand{\Cusp}{\Irr_c}
\newcommand{\somele}{\unlhd}
\newcommand{\cspline}{\mathcal{L}}
\newcommand{\Hom}{\operatorname{Hom}}
\newcommand{\rk}{r}
\newcommand{\X}{X}
\newcommand{\Y}{\tilde X}
\newcommand{\Reps}{\mathcal{C}}
\newcommand{\Gr}{\mathcal{R}}
\newcommand{\rflx}{\mathcal{T}}
\newcommand{\rsig}{\mathcal{I}}
\newcommand{\asig}{\mathcal{J}}
\newcommand{\rltn}{\leadsto}                                   
\newcommand{\lrltn}[1]{\overset{#1}\rltn}
\newcommand{\word}{\mathfrak{w}}
\newcommand{\tseq}{\mathcal{A}}                                
\newcommand{\biseq}{bi-sequence}                               
\newcommand{\bitmplt}{\begin{pmatrix}a_1&\dots&a_k\\
b_1&\dots&b_k\end{pmatrix}}
\newcommand{\permat}{{\mathbf P}_}                                       
\newcommand{\LC}{LC}                                           
\newcommand{\clsf}{\mathfrak{c}}                               
\newcommand{\std}{\zeta}                                       
\newcommand{\bss}{\mathcal{B}}
\newcommand{\lderiv}{\mathcal{D}^{\operatorname{l}}}
\newcommand{\rderiv}{\mathcal{D}^{\operatorname{r}}}
\newcommand{\jac}{J}
\newcommand{\Mult}{\mathfrak{M}}
\newcommand{\lmlt}{\mu^l}
\newcommand{\rmlt}{\mu^r}
\newcommand{\sgn}{\operatorname{sgn}}
\newcommand{\Nrd}{\operatorname{Nrd}}
\newcommand{\abs}[1]{\left|{#1}\right|}
\newcommand{\supp}{\operatorname{supp}}                        
\newcommand{\rest}{\big|}
\newcommand{\GL}{\operatorname{GL}}
\newcommand{\Irr}{\operatorname{Irr}}
\newcommand{\m}{\mathfrak{m}}
\newcommand{\n}{\mathfrak{n}}
\newcommand{\soc}{\operatorname{soc}}                          
\newcommand{\coss}{\operatorname{cos}}                          
\newcommand{\SI}{SI}                                           
\newcommand{\zele}[1]{\operatorname{Z}( #1 )}                  
\newcommand{\lshft}[1]{\overset{\leftarrow}{#1}}               
\newcommand{\rshft}[1]{\overset{\rightarrow}{#1}}              
\newcommand{\LI}{LI}                                          
\newcommand{\RI}{RI}                                          
\newcommand{\LM}{$\square$-irreducible}                      
\newcommand{\cmplx}{\mathfrak{c}}                             
\newcommand{\depth}{\mathfrak{d}}                             
\newcommand{\APU}{APU}                                        
\newcommand{\adj}{\,\vdash\,}                                 
\newcommand{\obt}{\,\models\,}                                
\newcommand{\grimg}[1]{\left\langle#1\right\rangle}                                  
\newcommand{\del}{\Delta}
\newcommand{\lnrset}{\mathfrak{S}^l}
\newcommand{\rnrset}{\mathfrak{S}^r}
\newcommand{\speh}[2]{{#1}^{(#2)}}
\newcommand{\cstar}{(*)}

\author{Erez Lapid}
\thanks{E.L. was partially supported by Grant \# 711733 from the Minerva Foundation.}
\address{Department of Mathematics, Weizmann Institute of Science, Rehovot 7610001, Israel}
\email{erez.m.lapid@gmail.com}
\author{Alberto M{\'{\i}}nguez}
\thanks{A.M. was partially supported by P12-FQM-2696}
\address{Institut de Math\'ematiques de Jussieu, Universit\'e Paris VI, Paris, France}
\email{alberto.minguez@imj-prg.fr}
\title[Geometric conditions for $\square$-irreducibility]{Geometric conditions for $\square$-irreducibility of certain representations
of the general linear group over a non-archimedean local field}

\begin{abstract}
Let $\pi$ be an irreducible, complex, smooth representation of $\GL_n$ over a local non-archimedean (skew) field.
Assuming $\pi$ has regular Zelevinsky parameters, we give a geometric necessary and sufficient criterion for the irreducibility
of the parabolic induction of $\pi\otimes\pi$ to $\GL_{2n}$.
The latter irreducibility property is the $p$-adic analogue of a special case of the notion of ``real representations'' introduced by Leclerc and studied recently by
Kang--Kashiwara--Kim--Oh (in the context of KLR or quantum affine algebras).
Our criterion is in terms of singularities of Schubert varieties of type $A$ and admits a simple combinatorial description.
It is also equivalent to a condition studied by Geiss--Leclerc--Schr\"oer.
\end{abstract}

\maketitle

\setcounter{tocdepth}{1}
\tableofcontents

\section{Introduction}

In this paper, which is a sequel to \cite{MR3573961}, we study several questions arising from the problem of characterizing reducibility of parabolic induction
for smooth, complex representations of the general linear group over a non-archimedean local field $F$.\footnote{In the body of the paper, we also consider
skew fields but for the introduction we stick to the commutative case.}
We connect this representation-theoretic question to combinatorics and geometry.

As customary, we consider all groups $\GL_n(F)$, $n\ge0$ at once and denote simply by $\Irr$ the set of equivalence classes of irreducible representations of $\GL_n(F)$, $n\ge0$.
By the Zelevinsky classification \cite{MR584084}, $\Irr$ is in one-to-one correspondence with the monoid of multisegments, which are certain essentially combinatorial objects.
We write $\zele{\m}$ for the irreducible representation corresponding to the multisegment $\m$ and denote by $\times$ normalized parabolic induction.
Then $\zele{\m+\n}$ occurs with multiplicity one in the Jordan--H\"older sequence of $\zele{\m}\times\zele{\n}$. Consequently,
\[
\zele{\m}\times\zele{\n}\text{ is irreducible }\iff\LI(\m,\n)\text{ and }\LI(\n,\m)
\]
where $\LI(\m,\n)$ is the condition $\soc(\zele{\m}\times\zele{\n})=\zele{\m+\n}$ and $\soc$ denotes the socle.
This was the point of departure of \cite{MR3573961} which led us to study the property $\LI(\m,\n)$ and characterize it purely combinatorially
in special cases.

In general, $\soc(\pi\times\sigma)$ is not necessarily irreducible for $\pi,\sigma\in\Irr$.
In fact, as was shown by Leclerc \cite{MR1959765}, there are examples of $\pi\in\Irr$ for which $\pi\times\pi$ is semisimple of length $2$.
However, it turns out that if $\pi\times\pi$ is irreducible (in which case we say that $\pi$ is \LM)\footnote{In general, an object $M$ in a ring category
is called ``real'' if $M\otimes M$ is simple. In the $p$-adic case at hand we opted for a different terminology, for obvious reasons.}
then for any $\sigma\in\Irr$, $\soc(\pi\times\sigma)$ is irreducible and
occurs with multiplicity one in the Jordan--H\"older sequence of $\pi\times\sigma$.
This is an analogue of a recent result of Kang--Kashiwara--Kim--Oh, originally proved in the context of finite-dimensional modules
of either quiver Hecke (a.k.a. KLR) algebras or quantum affine algebras \cite{MR3314831}.
The argument can be adapted to the $p$-adic setting without much difficulty -- see \S\ref{sec: KKKO}.

Granted this result, two natural interrelated problems arise. The first is to characterize (combinatorially or otherwise) the
$\square$-irreducibility of $\pi$. The second is to characterize the condition $\LI(\m,\n)$ or more generally determine
$\soc(\pi\times\sigma)$, at least when $\pi$ is \LM.
We focus on the first question in this paper, leaving the second one for a future work.

Let us briefly recall the geometry pertaining to the Zelevinsky classification \cite{MR617466, MR783619, MR863522}.
Consider pairs $(V,A)$ where $V=\oplus_{n\in\Z}V_n$ is a finite-dimensional $\Z$-graded $\C$-vector space
and $A$ is in the space $E_+(V)$ of $\C$-linear (nilpotent) endomorphisms of $V$ such that $A(V_n)\subset V_{n+1}$ for all $n$.
The isomorphism types of such pairs are parameterized by (certain) multisegments in a simple way.
Similarly if $E_+(V)$ is replaced by $E_-(V)$, with the obvious meaning.
Given $V$ as before, the group $\GL(V)$ of grading preserving linear automorphisms of $V$ acts with finitely many orbits on each of the spaces $E_\pm(V)$,
which are in duality with respect to the $\GL(V)$-invariant pairing $A,B\mapsto\tr AB=\tr BA$.
Consider the algebraic set
\[
\commvar(V)=\{(A,B)\in E_+(V)\times E_-(V):AB=BA\}.
\]
By a well-known result of Pyasetskii \cite{MR0390138}, the set of $\GL(V)$-orbits in $E_+(V)$ (or $E_-(V)$) is in canonical bijection with the
set of irreducible components of $\commvar(V)$.

The work of Geiss--Leclerc--Schr\"oer (in a more general context) highlighted the property that an irreducible component of $\commvar(V)$
admits an open $\GL(V)$-orbit. The following is a stronger variant of a special case of their beautiful conjecture. \nocite{MR2242628}
\begin{conjecture}(cf. \cite[Conjecture 18.1]{MR2822235}, \cite{LecChev}) \label{conj: GLSi}
Let $\Comp_\m$ be the irreducible component in $\commvar(V_\m)$ (for suitable $V_\m$) corresponding to a multisegment $\m$.
Then $Z(\m)$ is \LM\ if and only if $\Comp_\m$ admits an open $\GL(V_\m)$-orbit.
\end{conjecture}
The pertinent geometric condition admits an even more down-to-earth interpretation (see \S\ref{sec: GLSconj}).

Our main result is a proof of this conjecture in the so-called regular case, where we link the above condition to another geometric criterion.
Before stating it, let us introduce some more notation.
\begin{itemize}
\item For any integers $a\le b+1$ let $Z([a,b])$ be the character $\abs{\det\cdot}^{(a+b)/2}$ of $\GL_{b-a+1}(F)$.

\item For any permutation $\sigma\in S_k$, $k>0$ we denote by $\cell_\sigma$ (resp., $X_\sigma$) the corresponding Schubert cell (resp., variety)
in the flag variety of type $A_{k-1}$.
Thus, $\cell_\sigma$ is Zariski open in its closure $X_\sigma$ and $X_\sigma\supset\cell_{\sigma'}$ if and only if $\sigma'\le\sigma$ in the Bruhat order.

\item For $N>1$ let $U_q(\widehat{\sll}_N)$ be the quantum affine algebra pertaining to the affine Lie algebra $\widehat{\sll}_N$ where $q\in\C^*$ is not a root of unity.
The finite-dimensional simple modules of $U_q(\widehat{\sll}_N)$ are parameterized by Drinfeld polynomials,
or what amounts to the same, by monomials in the formal variables $Y_{i,a}$, $i=1,\dots,N-1$, $a\in\C^*$ (e.g., \cite{MR2642561}).
\end{itemize}

\begin{theorem} \label{thm: maini}
Let $\m=[a_1,b_1]+\dots+[a_k,b_k]$ where $a_1,\dots,a_k,b_1,\dots,b_k$ are integers such that $a_i\le b_i$ for all $i$.
Assume that $b_1>\dots>b_k$ and that $a_1,\dots,a_k$ are distinct.
Then Conjecture \ref{conj: GLSi} holds for
\[
\pi=Z(\m)=\soc(Z([a_1,b_1])\times\dots\times Z([a_k,b_k])).
\]
Moreover, let $\sigma,\sigma_0\in S_k$ be the permutations such that $a_{\sigma(1)}<\dots<a_{\sigma(k)}$ and for all $i$
\[
\sigma_0^{-1}(i)=\max\{j\le x_i:j\notin\sigma_0^{-1}(\{i+1,\dots,k\})\}\text{ where }x_i=\#\{j:a_j\le b_i+1\}.
\]
For $N>1+\max_i(b_i-a_i)$ let $L_N$ be the finite-dimensional simple module of $U_q(\widehat{\sll}_N)$ corresponding to the monomial $\prod_{i=1}^kY_{b_i-a_i+1,q^{a_i+b_i}}$.
Then $\sigma_0\le\sigma$ and the following conditions are equivalent.
\begin{enumerate}
\item \label{cond: pisqrirred} $\pi$ is \LM.
\item \label{cond: Lreal} $L_N$ is real, i.e., $L_N\otimes L_N$ is irreducible, for $N\gg1$.
\item \label{cond: GLS} $\Comp_\m$ admits an open $\GL(V_\m)$-orbit. (See Conjecture \ref{conj: GLSi}.)
\addtocounter{enumi}{1}
\begin{enumerate}[leftmargin=0pt, label*=\alph*]
\item \label{cond: smlocus} The smooth locus of $X_\sigma$ contains $\cell_{\sigma_0}$.
\item \label{cond: rsmlocus} $X_\sigma$ is rationally smooth at any point of $\cell_{\sigma_0}$.
\item \label{cond: ntrans} The number of transpositions $\tau\in S_k$ such that $\sigma_0\tau\le\sigma$ is equal to the length of $\sigma$.
\item \label{cond: KL} The Kazhdan--Lusztig polynomial $P_{\sigma_0,\sigma}$ with respect to $S_k$ is $1$.
\item \label{cond: KLall} $P_{\sigma',\sigma}\equiv1$ for every $\sigma'\in S_k$ such that $\sigma_0\le\sigma'\le\sigma$.
\end{enumerate}
\item \label{cond: detform} In the Grothendieck group we have
\[
\pi=\sum_{\sigma'\in S_k:\sigma_0\le\sigma'\le\sigma}\sgn\sigma\sigma'
\ Z([a_{\sigma(1)},b_{\sigma'(1)}])\times\dots\times Z([a_{\sigma(k)},b_{\sigma'(k)}]).
\]
\item \label{cond: sseq} There does not exist a sequence $1\le n_1<\dots<n_r\le k$, $r\ge4$ such that if $a'_i=a_{n_i}$ and $b'_i=b_{n_i}$ then either
\[
a'_{i+1}<a'_i\le b'_{i+1}+1,\ i=3,\dots,r-1,\ a'_3<a'_1\le b'_3+1\text{ and }a'_r<a'_2<a'_{r-1}
\]
or
\[
a'_{i+1}<a'_i\le b'_{i+1}+1,\ i=4,\dots,r-1,\ a'_4<a'_2\le b'_4+1\text{ and }a'_3<a'_r<a'_1<a'_l
\]
where $l=2$ if $r=4$ and $l=r-1$ otherwise.
\end{enumerate}
\end{theorem}

The equivalence of conditions \ref{cond: pisqrirred} and \ref{cond: Lreal} follows from the quantum Schur--Weyl duality \cite{MR1405590}.
The equivalence of conditions \ref{cond: smlocus}, \ref{cond: rsmlocus}, \ref{cond: ntrans}, \ref{cond: KL} and \ref{cond: KLall}
(for any $\sigma,\sigma_0\in S_k$) is well known (\cite{MR788771}).
The equivalence of conditions \ref{cond: KLall} and \ref{cond: detform} follows from the properties of the Arakawa--Suzuki functor \cite{MR1652134}
and the Kazhdan--Lusztig conjecture \cite{MR610137, MR632980} -- see \S\ref{sec: KLid}.
If $\max_ia_i\le b_k+1$ then $\sigma_0$ is the identity and condition \ref{cond: smlocus} simply becomes the smoothness of $X_\sigma$.
In this case condition \ref{cond: sseq} is tantamount to the well-known smoothness criterion of Lakshmibai--Sandhya \cite{MR1051089} that $\sigma$ avoids the patterns $3412$ and $4231$.
In the general case, the equivalence of conditions \ref{cond: smlocus} and \ref{cond: sseq} follows from the description of the maximal singular loci of $X_\sigma$ due (independently) to
Billey--Warrington, Cortez, Kassel--Lascoux--Reutenauer and Manivel \cite{MR1990570, MR1994224, MR2015302, MR1853139}, as explained in \S\ref{sec: smth pairs}
and \S\ref{sec: combi}. Incidentally, $\sigma_0$ is a stack-sortable permutation in the sense of Knuth. (Roughly speaking, it encodes how the sets
$\{a_1,\dots,a_k\}$ and $\{b_1,\dots,b_k\}$ are interleaved.)

Thus, the main innovative part of the paper is the equivalence of the conditions \ref{cond: pisqrirred}, \ref{cond: GLS} and \ref{cond: sseq}.
(See Theorem \ref{thm: main}.)
The ensuing equivalence of conditions \ref{cond: GLS} and \ref{cond: smlocus} is striking since at first glance, the two geometric conditions are seemingly of a different nature.
Indeed, at present we do not have a good geometric insight for this equivalence. Instead, we prove it combinatorially.

The case where $a_1>\dots>a_k$ (i.e., where $\sigma$ is the longest element of $S_k$) is especially important.
It was considered in \cite{MR3573961} under the name ``ladder representations''. In other contexts it has been known under different names.

Let us say a few words about the proof.
The implication \eqref{cond: sseq}$\implies$\eqref{cond: pisqrirred} is proved in \S\ref{sec: main} by induction on $k$.
For the induction step we use the simple observation
that if $\pi\hookrightarrow\pi_1\times\pi_2$ and $\pi\times\pi_1$ is irreducible then $\pi$ is \LM\ provided that $\pi_2$ is \LM.
(See Lemma \ref{lem: albertoidea}.)
In the case at hand we take $\pi_1$ to be a ladder representation and use the results of \cite{MR3573961}
to check the required properties combinatorially.
A parallel argument yields the implication \eqref{cond: sseq}$\implies$\eqref{cond: GLS}.

For the inverse direction $\lnot$\eqref{cond: sseq}$\implies\lnot$\eqref{cond: pisqrirred}, i.e., to prove non-$\square$-irreducibility,
we make several reductions to certain basic cases for which we use the following ``double socle'' strategy.
Given $\pi=\zele{\m}$, we construct $\square$-irreducible representations $\pi_1,\pi_2$ such that $\pi\hookrightarrow\pi_1\times\pi_2$.
Then $\Pi:=\soc(\pi_1\times\soc(\pi_2\times\pi))\hookrightarrow\pi_1\times\pi_2\times\pi$ is irreducible and hence
$\Pi\hookrightarrow\omega\times\pi$ for some irreducible subquotient $\omega$ of $\pi_1\times\pi_2$.
We then show that this is not possible unless $\omega=\pi$. This ensures that $\Pi\hookrightarrow\pi\times\pi$ and hence that $\pi$ is not \LM\
provided that $\Pi\ne\zele{\m+\m}$.
The proof is rather technical and once again, uses the results of \cite{MR3573961}.
The same reductions apply to the implication $\lnot$\eqref{cond: sseq}$\implies\lnot$\eqref{cond: GLS}, for which the basic cases are easy to verify.

As far as we know, Theorem \ref{thm: maini} is the first instance where a non-trivial infinite family of non-\LM\ representations is exhibited.
We remark that Theorem \ref{thm: maini} is proved more generally for Zelevinsky's segment representations
$Z([a,b])=\soc(\rho\abs{\det\cdot}^a\times\rho\abs{\det\cdot}^{a+1}\times\dots\times\rho\abs{\det\cdot}^b)$
for any fixed supercuspidal $\rho\in\Irr$.
Theorem \ref{thm: maini} implies the following curious identity of Kazhdan--Lusztig polynomials with respect to $S_{2k}$.
\begin{theorem}[Corollary \ref{cor: KLidnt}]
Let $\sigma,\sigma_0\in S_k$ be such that the equivalent conditions \ref{cond: smlocus}--\ref{cond: KLall} of Theorem \ref{thm: maini} are satisfied and
$\sigma_0$ is $213$-avoiding.
Let $\tilde\sigma\in S_{2k}$ be given by $\tilde\sigma(2i-j)=2\sigma(i)-j$, $i=1,\dots,k$, $j=0,1$ and
let $H\simeq S_k\times S_k$ be the parabolic subgroup of $S_{2k}$ of type $(k,k)$. Then
\[
\sum_{w\in H}\sgn w\ P_{\widetilde{\sigma'}w,\widetilde{\sigma}}(1)=1
\]
for any $\sigma'$ such that $\sigma_0\le\sigma'\le\sigma$.
\end{theorem}
See also Theorem \ref{thm: higherKL} for a generalization and \cite{1705.06517} for a follow-up conjecture.
It would be interesting to have a geometric interpretation of these identities.

At the moment, it is not clear what would replace the smoothness condition \ref{cond: smlocus} of Theorem \ref{thm: maini} in the non-regular case.
At any rate, it seems that in order to attack Conjecture \ref{conj: GLSi} in the general case with our approach, it is imperative to generalize
the results of \cite{MR3573961} to a broader class of representations. We hope to pursue this in a forthcoming work.

The determination of $\soc(\pi\times\sigma)$ is a very useful tool in representation theory.
Already in the case where $\pi$ is supercuspidal, partial results in this direction were used by M\oe glin--Waldspurger
to explicate the Zelevinsky involution \cite{MR863522}. The analysis in this case was completed independently in \cite{MR2306606} and \cite{MR2527415}
and was a key ingredient for explicating the theta correspondence for dual pairs of type II \cite{MR2504432}.
These results were extended in \cite{MR3573961} to ladder representations and yielded
a new and simplified proof of the classification of the unitary dual of $\GL_m(F)$ and its inner forms.

The problem makes sense for classical groups as well.
This approach was used by Gan--Takeda \cite{MR3454380} in their new proof of the theta correspondence for dual pairs of type I.
A better understanding in this case is a prerequisite for determining the (still unknown) unitary dual of classical groups.
(See \cite{1703.09475} for some preliminary results in this direction.)

As alluded to above, the real modules of quantum affine algebras and KLR algebras play a crucial role in the monoidal categorification of
certain (quantum) cluster algebras \cite{MR3077685, MR2682185, 1412.8106, 1502.06714}.
In particular, they are expected to represent the cluster monomials.
It remains to be seen whether our results shed any light on this procedure, or whether they can be extended beyond type $A$.
We caution however that although the notions of quantum affine algebras and KLR algebras make sense for any Cartan datum,
the link to representation theory of $p$-adic groups works well only for type $A$.
Thus, the study of reducibility questions of parabolic induction (say, for classical groups) may lead to different questions.

The contents of this paper is as follows.
In \S\ref{sec: KKKO} we translate some of the results and proofs of Kang--Kashiwara--Kim--Oh into the language of $p$-adic groups (mostly $\GL(n,F))$.
Essentially, the role of $R$-matrices is played by the usual intertwining operators.
We then recall in \S\ref{sec: classification} the Zelevinsky classification (extended to division algebras) and the irreducibility criteria of \cite{MR3573961},
which are the principal tools for the proof of the main result.
In \S\ref{sec: GLSconj} we explicate the openness criterion of Geiss--Leclerc--Schr\"oer in the case at hand and state equivalent forms
of Conjecture \ref{conj: GLSi}.
We also give some consistency checks which will be used in the proof of the main result.
This concludes the first part of the paper.

In the second part we focus on the case of irreducible representations with regular parameters, for which our main result applies.
In \S\ref{sec: smth pairs} we recall some well-known facts about singularities of Schubert varieties of type $A$.
In \S\ref{sec: combi} we introduce the main combinatorial criterion for multisegments and reinterpret it using the results of
\cite{MR1990570, MR1994224, MR2015302, MR1853139}.
The recent thesis of Deng Taiwang \cite{1603.06387} sheds more light on some of the material of this section as well as on \S\ref{sec: KLid}.
The main result is stated in \S\ref{sec: main} where the $\square$-irreducibility part is proved.
Exemplars of non-\LM\ representations are constructed in \S\ref{sec: basicases}.
The reduction to these special cases is accomplished in \S\ref{sec: comproof} where the proof of the main result is completed.
Finally, in \S\ref{sec: KLid} we interpret the main result in terms of an identity of Kazhdan--Lusztig polynomials via the Arakawa--Suzuki functor.

\subsection*{Acknowledgement}
We are grateful to a number of mathematicians and institutes for their help in preparing this manuscript.
First and foremost we are grateful to Guy Henniart for suggesting the formulation of Lemma \ref{lem: Henniart} and for several other suggestions.
We thank Max Gurevich for his input on Lemma \ref{lem: LM}.
We are thankful to David Hernandez and Bernard Leclerc for very useful correspondence and to Pascal Boyer for communicating to us the thesis of Taiwang Deng.
We are indebted to Greg Warrington for clarifying some points in \cite{MR1990570} and most importantly for
his C code for computing the Kazhdan--Lusztig polynomials for the symmetric group which was pivotal for coming up with
the formulation of our main result.
Finally, we thank Karim Adiprasito, Arkady Berenstein, Joseph Bernstein, Michel Brion, Tobias Finis, Masaki Kashiwara,
David Kazhdan, Eric Opdam, Gordan Savin and Toshiyuki Tanisaki for useful correspondence and discussions at various stages of this project.

Part of this work was done while the first-named author was visiting the Institute for Mathematical Sciences, National University of Singapore in spring 2016.
He thanks the IMS for their kind support.
The second-named author would like to thank the CNRS for the one year of ``d\'el\'egation'' where part of this work was completed and the
University of Seville for its hospitality during the academic year 2017-2018.
Both authors would like to thank Instituto de Matem\'aticas Universidad de Sevilla for providing good working conditions during
a visit of the first-named author.

\subsection*{Notation}
Throughout we fix a non-archimedean local division algebra $D$ with center $F$.
We denote by $\abs{\cdot}$ the normalized absolute value on $F$ and by $\# A$ the cardinality of a finite set $A$.
For any reductive group $G$ over $F$ we denote by $\Reps(G)$ the category of complex, smooth representations
of $G(F)$ of finite length (hence admissible) and by $\Irr G$ the set of irreducible objects of $\Reps(G)$ up to equivalence.
We have a well-known decomposition $\Reps(G)=\oplus_{\mathcal{D}}\Reps_{\mathcal{D}}(G)$ according to supercuspidal data
(i.e., pairs $(M,\sigma)$ where $M$ is a Levi subgroup of $G$ and $\sigma$ is a supercuspidal representation of $M(F)$, up to conjugation by $G(F)$).
We will mostly consider the groups $G_n=\GL_n(D)$, $n=0,1,2,\dots$, the multiplicative group of the ring of $n\times n$ matrices over $D$.
If $\pi_i\in\Reps(G_{n_i})$ $i=1,2$, we denote by $\pi_1\times\pi_2\in\Reps(G_{n_1+n_2})$ the representation
parabolically induced from $\pi_1\otimes\pi_2$ (normalized induction).
This functor (and the isomorphism of induction by stages) endow $\Reps=\oplus_{n\ge0}\Reps(G_n)$ with the structure of a ring category\footnote{i.e.,
a locally finite $\C$-linear abelian monoidal category in which $\End(1)=\C$ and the tensor product bifunctor is bilinear and biexact -- cf. \cite[Ch. 4]{MR3242743}.}
where the identity (which we denote by $1$) is the one-dimensional representation of $G_0$.
Let $\Gr_n$ (resp., $\Gr$) be the Grothendieck group of $\Reps(G_n)$ (resp., $\Reps$).
Even though $\times$ is not commutative in $\Reps$, $\Gr=\oplus_{n\ge0}\Gr_n$ is nevertheless a commutative graded ring under $\times$.
Set $\Irr=\cup_{n\ge0}\Irr G_n$ and let $\Cusp\subset\Irr$ be the subset of supercuspidal representations of $G_n$, $n>0$.
(Note that by convention we exclude $1\in\Irr G_0$ from $\Cusp$.)

Let $\pi,\pi'\in\Reps(G_n)$ and $\chi$ a character of $F^*$. We use the following notation and terminology.
\begin{itemize}
\item $\deg(\pi)=n$.
\item $\pi^\vee\in\Reps(G_n)$ is the contragredient of $\pi$.
\item $\pi\chi\in\Reps(G_n)$ is the representation obtained from $\pi$ by twisting by the character $\chi\circ\Nrd$ where $\Nrd$ is the reduced norm on $G_n$.
\item $\JH(\pi)$ is the Jordan--H\"older sequence of $\pi$ (i.e., the image of $\pi$ in $\Gr_n$), viewed as a finite multiset of $\Irr G_n$.
\item We write $\pi'\le\pi$ if $\JH(\pi')\subset\JH(\pi)$ (including multiplicities).
\item $\soc(\pi)$ (resp., $\coss(\pi)$) is the socle (resp., cosocle) of $\pi$, i.e., the largest semisimple subrepresentation (resp., quotient) of $\pi$.
\item We say that $\pi$ is \SI\ if $\soc(\pi)$ is irreducible and occurs with multiplicity one in $\JH(\pi)$.
\item $\supp\pi$ is the supercuspidal support of $\pi$ considered as a finite subset of $\Cusp$ (without multiplicity).
\item $\jac_{(m,n-m)}(\pi)\in\Reps(G_m\times G_{n-m})$ is the (normalized) Jacquet module of $\pi$ with respect to the standard (upper triangular)
parabolic subgroup of $G_n$ of type $(m,n-m)$, $0\le m\le n$. Often we simply write $\jac(\pi)$ if $m$ and $n$ are clear from the context.
\item For $\rho_1,\dots,\rho_k\in\Cusp$ with $m=\deg\rho_1+\dots+\deg\rho_k\le n$, $\jac_{(m,n-m)}(\pi)_{\rho_1+\dots+\rho_k;*}$
denotes the maximal subrepresentation $\sigma$ of $\jac_{(m,n-m)}(\pi)$
with the property that any supercuspidal irreducible subquotient of a Jacquet module of $\sigma$ is of the form
$\rho_{\tau(1)}\otimes\dots\otimes\rho_{\tau(k)}\otimes\rho'_1\otimes\dots\otimes\rho'_l$ for some permutation $\tau$ of $\{1,\dots,k\}$
and $\rho'_i\in\Cusp$. Similarly for $\jac_{(n-m,m)}(\pi)_{*;\rho_1+\dots+\rho_k}$.
\end{itemize}

\part{}
\section{Some results of Kang--Kashiwara--Kim--Oh} \label{sec: KKKO}
The purpose of this section is to translate some of the results of \cite{MR3314831} and \cite{1412.8106} to the language of
representations of reductive groups (mostly, $G_n$) over $F$.

\subsection{}
For the next lemma let $G$ be a (connected) reductive group defined over $F$ and let $P,Q$ be parabolic subgroups of $G$
(defined over $F$) such that $R=P\cap Q$ is also a parabolic subgroup. We denote (normalized) parabolic induction by $I_P^G$.
It is a functor from $\Reps(M_P)$ to $\Reps(G)$ where $M_P$ denotes the Levi part of $P$. Similarly for $I_R^P$ and $I_R^Q$.

We are very grateful to Guy Henniart for suggesting to us the following neat formulation.

\begin{lemma} \label{lem: Henniart}
Let $\pi\in\Reps(M_R)$ and let $\sigma$ (resp., $\tau$) be a subrepresentation of $I_R^P(\pi)$ (resp., $I_R^Q(\pi)$).
Assume that $I_P^G(\sigma)\subset I_Q^G(\tau)$ (as subrepresentations of $I_R^G(\pi)$).
Then there exists a subrepresentation $\kappa$ of $\pi$ such that $\sigma\subset I_R^P(\kappa)$ and $I_R^Q(\kappa)\subset\tau$.
\end{lemma}

\begin{proof}
Let $\kappa$ be the space $\{f(e):f\in\sigma\}\subset\pi$ where we view $\sigma$ as a subrepresentation of $I_R^P(\pi)$.
It is clear that $\kappa$ is a subrepresentation of $\pi$ and $\sigma\subset I_R^P(\kappa)$
(and in fact, $\kappa$ is the smallest subrepresentation with this property).
It remains to show that $I_R^Q(\kappa)\subset\tau$.

We first recall that there exists a compact open subgroup $K_0$ of $G$ and basis $\bss$ of neighborhoods of $1$ in $G$ consisting of
normal subgroups of $K_0$ such that $PK\cap Q=R(K\cap Q)$ for every $K\in\bss$.
Indeed, fix an opposite parabolic $\bar Q$ to $Q$ with unipotent radical $\bar V$.
By \cite[p. 16]{MR771671} $G$ admits a compact open subgroup $K_0$ and a basis of neighborhoods of $1$
consisting of normal subgroups of $K_0$ satisfying $K=(K\cap \bar V)(K\cap Q)$.
On the other hand, we have $P\bar Q\cap Q=R$.
Hence, $PK\cap Q=(P(K\cap\bar V)\cap Q)(K\cap Q)\subset (P\bar V\cap Q)(K\cap Q)=R(K\cap Q)$ as required.

For any $v\in\pi$ and a compact open subgroup $K$ of $G$ such that $v\in\pi^{K\cap R}$ denote by $\varphi_{v;K}$
the element of $I_R^Q(\pi)$ which is supported in $R(K\cap Q)$ and takes the value $v$ on $K\cap Q$.
We claim that for any $K\in\bss$, $\tau$ contains $\{\varphi_{f(e);K}:f\in\sigma^{P\cap K}\}$.
Indeed, let $f\in\sigma^{P\cap K}$ and consider the element $\varphi'_f\in I_P^G(\sigma)$
which is supported in $PK$ and has constant value $f$ on $K$. We can view $\varphi'_f$ as an element of
$I_R^G(\pi)\simeq I_Q^G(I_R^Q(\pi))$. Let $\phi_f$ (resp., $\psi_f$) be the image of $\varphi'_f$ in $I_R^G(\pi)$ (resp., $I_Q^G(I_R^Q(\pi))$).
Then for any $g\in G$, $\phi_f(g)=\varphi'_f(g)(1)$ and $\psi_f(g)\in I_R^Q(\pi)$ is given by $q\mapsto\varphi'_f(qg)(1)$.
Since $I_P^G(\sigma)\subset I_Q^G(\tau)$ we infer that $\psi_f(g)\in\tau$ for all $g\in G$ and in particular, $\psi_f(e)\in\tau$.
On the other hand, since $\varphi'_f\rest_Q$ is supported in $PK\cap Q = R(K\cap Q)$, $\psi_f(e)$ coincides with $\varphi_{f(e);K}$. Our claim follows.

Fix a compact set $\Omega\subset Q$ such that $R\Omega=Q$ and let $\psi\in I_R^Q(\kappa)$. We need to show that $\psi\in\tau$.
Since $\psi$ takes only finitely many values on $\Omega$, there exists $K\in\bss$ and for any $g\in\Omega$ there exists
$f_g\in\sigma^{P\cap K}$ such that $f_g(e)=\psi(g)$ (and in particular $\psi(g)\in\pi^{K\cap R}$). By the discussion above we therefore have
\begin{equation} \label{eq: intau}
\varphi_{\psi(g);K_1}\in\tau\text{ for any $g\in\Omega$ and $K_1\in\bss$ contained in $K$}.
\end{equation}
Let $K_2=\cap_{g\in\Omega}g^{-1}Kg$. Then
\[
\psi=\sum_{\eta}I_R^Q(\pi,\eta^{-1})\varphi_{\psi(\eta);K_2^\eta}
\]
where $\eta$ ranges over a set of representatives of $R\bs Q/(K'\cap Q)$ contained in $\Omega$ and $K_2^\eta=\eta K_2\eta^{-1}$.
It remains to show that $\varphi_{\psi(\eta);K_2^\eta}\in\tau$ for all $\eta\in\Omega$.
However, if $K_1\in\bss$ is any subgroup of $K_2^\eta$ then
\[
\varphi_{\psi(\eta);K_2^\eta}=\sum_{\gamma\in(K_2^\eta\cap R)(K_1\cap Q)\bs K_2^\eta\cap Q}I_R^Q(\pi,\gamma^{-1})\varphi_{\psi(\eta);K_1}
\]
and the claim therefore follows from \eqref{eq: intau}.
\end{proof}

Specializing to the general linear groups we obtain

\begin{corollary}(cf. \cite[Lemma 3.1]{MR3314831}) \label{cor: threetensor}
Let $\pi_i\in\Reps(G_{n_i})$, $i=1,2,3$.
Let $\sigma$ (resp., $\tau$) be a subrepresentation of $\pi_1\times\pi_2$ (resp., $\pi_2\times\pi_3$).
Assume that $\sigma\times\pi_3\subset\pi_1\times\tau$. Then there exists a subrepresentation $\omega$ of $\pi_2$ such that
$\sigma\subset\pi_1\times\omega$ and $\omega\times\pi_3\subset\tau$.
In particular, if $\pi_2$ is irreducible and $\sigma\ne0$ then $\tau=\pi_2\times\pi_3$.
\end{corollary}

\begin{proof}
Indeed, let $P$ (resp., $Q$, $R$) be the standard parabolic subgroup of $G_{n_1+n_2+n_3}$ of type $(n_1+n_2,n_3)$ (resp., $(n_1,n_2+n_3)$, $(n_1,n_2,n_3)$)
so that $R=P\cap Q$. For brevity we denote by $I^{1,2}$ the functor of parabolic induction from $G_{n_1+n_2}\cap R$ to $G_{n_1+n_2}$;
similarly for $I^{2,3}$.
By Lemma \ref{lem: Henniart}, there exists a subrepresentation $\kappa$ of $\pi_1\otimes\pi_2\otimes\pi_3$ such that
$\sigma\otimes\pi_3\subset I_R^P(\kappa)$ and $I_R^Q(\kappa)\subset\pi_1\otimes\tau$.
Let $\alpha$ be the smallest subrepresentation of $\pi_2\otimes\pi_3$ such that $\kappa\subset\pi_1\otimes\alpha$.
Namely, $\alpha$ is the sum of the subrepresentations $\alpha_\lambda:=\lambda^*(\kappa)$
where $\lambda$ varies in $\pi_1^*$ and $\lambda^*$ is the map $\lambda\otimes\id_{\pi_2\otimes\pi_3}:\pi_1\otimes\pi_2\otimes\pi_3\rightarrow\pi_2\otimes\pi_3$.
Note that $I^{2,3}(\alpha)\subset\tau$ since for any $\lambda\in\pi_1^*$ we have
\[
I^{2,3}(\alpha_\lambda)=(\lambda\otimes\id_{\pi_2\times\pi_3})(I_R^Q(\kappa))\subset\tau.
\]
Now for any $w\in\pi_3$ let $\omega_w$ be the subrepresentation
\[
\omega_w=\{v\in\pi_2:v\otimes w\in\alpha\}
\]
of $\pi_2$ and let $\omega=\cap_{w\in\pi_3}\omega_w$.
Since $I^{2,3}(\alpha)\subset\tau$ we have $\omega\times\pi_3\subset\tau$. It remains to show that $\sigma\subset\pi_1\times\omega$.
By assumption $\sigma\otimes\pi_3\subset I_R^P(\kappa)\subset I_R^P(\pi_1\otimes\alpha)$.
Thus, for any $w\in\pi_3$ we have
\[
\sigma\otimes\C w\subset I_R^P(\pi_1\otimes\alpha)\cap (\pi_1\times\pi_2\otimes\C w).
\]
As a representation of $G_{n_1+n_2}$ the latter is
\begin{multline*}
I^{1,2}(\pi_1\otimes\alpha)\cap I^{1,2}(\pi_1\otimes\pi_2\otimes\C w)
=\\I^{1,2}(\pi_1\otimes(\alpha\cap\pi_2\otimes\C w))=
I^{1,2}(\pi_1\otimes\omega_w\otimes\C w)=\pi_1\times\omega_w\otimes\C w.
\end{multline*}
It follows that $\sigma\subset\pi_1\times\omega_w$ for all $w$ and hence $\sigma\subset\pi_1\times\omega$ as required.
\end{proof}

\subsection{}
Let $\pi_i\in\Reps(G_{n_i})$, $i=1,2$.
We write $M_{\pi_1,\pi_2}(s)$ for the standard intertwining operator
\[
M_{\pi_1,\pi_2}(s):\pi_1\abs{\cdot}^s\times\pi_2\rightarrow\pi_2\times\pi_1\abs{\cdot}^s
\]
(see e.g. \cite[\S IV]{MR1989693}). (It depends on a choice of a Haar measure, but this will be immaterial for us.)
If $\pi_1,\pi_2\ne0$ let $r_{\pi_1,\pi_2}\ge0$ be the order of the pole of $M_{\pi_1,\pi_2}(s)$ at $s=0$ and let
\[
R_{\pi_1,\pi_2}=\lim_{s\rightarrow0}s^{r_{\pi_1,\pi_2}}M_{\pi_1,\pi_2}(s).
\]
Thus, $R_{\pi_1,\pi_2}$ is a non-zero intertwining operator from $\pi_1\times\pi_2$ to $\pi_2\times\pi_1$.
The following result is standard.
\begin{lemma} \label{lem: elemM}
Let $0\ne\pi_i\in\Reps(G_{n_i})$, $i=1,2,3$ and
\[
0\rightarrow\tau\rightarrow\pi_1\rightarrow\tau'\rightarrow0
\]
a short exact sequence. Then
\begin{enumerate}
\setcounter{enumi}{-1}
\item $r_{1,\pi_1}=r_{\pi_1,1}=0$ and $R_{1,\pi_1}=R_{\pi_1,1}=\id_{\pi_1}$.
\item We have a commutative diagram of short exact sequences
\[
\xymatrix{
0\ar[r] & \tau\abs{\cdot}^s\times\pi_2 \ar[r] \ar[d]^{M_{\tau,\pi_2}(s)} & \pi_1\abs{\cdot}^s\times\pi_2\ar[r] \ar[d]^{M_{\pi_1,\pi_2}(s)} &
\tau'\abs{\cdot}^s\times\pi_2\ar[r] \ar[d]^{M_{\tau',\pi_2}(s)} & 0\\
0\ar[r] & \pi_2\times\tau\abs{\cdot}^s \ar[r] & \pi_2\times\pi_1\abs{\cdot}^s \ar[r] & \pi_2\times\tau'\abs{\cdot}^s\ar[r] & 0}
\]
\item \label{part: sqle} $r_{\tau,\pi_2}\le r_{\pi_1,\pi_2}$ if $\tau\ne0$; $r_{\tau',\pi_2}\le r_{\pi_1,\pi_2}$ if $\tau'\ne0$.
\item \label{part: nonzeroR} $R_{\pi_1,\pi_2}$ restricts to an intertwining operator $\tau\times\pi_2\rightarrow\pi_2\times\tau$. More precisely,
\[
R_{\pi_1,\pi_2}\rest_{\tau\times\pi_2}=\begin{cases}R_{\tau,\pi_2}&\text{if }\tau\ne0\text{ and }r_{\pi_1,\pi_2}=r_{\tau,\pi_2},\\
0&\text{otherwise.}\end{cases}
\]
\item $M_{\pi_1\times\pi_2,\pi_3}(s)=(M_{\pi_1,\pi_3}(s)\times\id_{\pi_2\abs{\cdot}^s})\circ(\id_{\pi_1\abs{\cdot}^s}\times M_{\pi_2,\pi_3}(s))$.
\item \label{part: ifisom} $r_{\pi_1\times\pi_2,\pi_3}\le r_{\pi_1,\pi_3}+r_{\pi_2,\pi_3}$ and
\[
(R_{\pi_1,\pi_3}\times\id_{\pi_2})\circ(\id_{\pi_1}\times R_{\pi_2,\pi_3})=
\begin{cases}R_{\pi_1\times\pi_2,\pi_3}&\text{if }r_{\pi_1\times\pi_2,\pi_3}=r_{\pi_1,\pi_3}+r_{\pi_2,\pi_3}\\0&\text{otherwise.}\end{cases}
\]
Moreover, $r_{\pi_1\times\pi_2,\pi_3}=r_{\pi_1,\pi_3}+r_{\pi_2,\pi_3}$ if at least one of $R_{\pi_i,\pi_3}$, $i=1,2$ is an isomorphism or if $\pi_3$ is irreducible
\cite[Lemma 2.8]{1412.8106}.
\item $M_{\pi_2\abs{\cdot}^s,\pi_1\abs{\cdot}^s}(-s)\circ M_{\pi_1,\pi_2}(s)=c(s)\id_{\pi_1\abs{\cdot}^s\times\pi_1}$ for some meromorphic function $c(s)$.
\item \label{part: compR} Suppose that $R_{\pi_1,\pi_2}$ is an isomorphism. Then $R_{\pi_2,\pi_1}\circ R_{\pi_1,\pi_2}$ is a non-zero scalar.
\end{enumerate}
\end{lemma}

\begin{proof}
The only non-evident part is that if $\pi_3$ is irreducible then $r_{\pi_1\times\pi_2,\pi_3}=r_{\pi_1,\pi_3}+r_{\pi_2,\pi_3}$.
Suppose that this is not the case. Then $(R_{\pi_1,\pi_3}\times\id_{\pi_2})\circ(\id_{\pi_1}\times R_{\pi_2,\pi_3})=0$
and therefore $\pi_1\times\operatorname{Im} R_{\pi_2,\pi_3}\subset\Ker R_{\pi_1,\pi_3}\times\pi_2$.
This contradicts an obvious analogue of Corollary \ref{cor: threetensor} since both $R_{\pi_2,\pi_3}$ and $R_{\pi_1,\pi_3}$ are non-zero.
\end{proof}

\begin{corollary}(cf. \cite[Theorem 3.2]{MR3314831}) \label{cor: maincor}
Let $0\ne\pi\in\Reps(G_n)$ be such that $R_{\pi,\pi}$ is a (non-zero) scalar.
Then for any $\sigma\in\Irr G_m$, $\soc(\pi\times\sigma)$ is irreducible and is equal to the image of $R_{\sigma,\pi}$.
In particular, $\pi$ and $\pi\times\pi$ are irreducible.
Similarly, $\soc(\sigma\times\pi)$ is irreducible and is equal to the image of $R_{\pi,\sigma}$.
Finally, $\soc(\pi\times\sigma)\simeq\coss(\sigma\times\pi)$ and $\soc(\sigma\times\pi)\simeq\coss(\sigma\times\pi)$.
\end{corollary}

\begin{proof}
By assumption, $R_{\pi,\pi}=\lambda^{-1}\id_{\pi\times\pi}$ for some non-zero scalar $\lambda$.
Therefore, by Lemma \ref{lem: elemM} part \ref{part: ifisom}
\begin{equation} \label{eq: radd}
r_{\pi\times\sigma,\pi}=r_{\pi,\pi}+r_{\sigma,\pi}\text{ and }\lambda R_{\pi\times\sigma,\pi}=\id_{\pi}\times R_{\sigma,\pi}.
\end{equation}
Let $\tau$ be a non-zero subrepresentation of $\pi\times\sigma$.
Thus, we have a commutative diagram
\[
\xymatrix{
\tau\times\pi \ar[rr]^{\lambda R_{\pi\times\sigma,\pi}} \ar@{^{(}->}[d] &&\pi\times\tau \ar@{^{(}->}[d]\\
\pi\times\sigma\times\pi \ar[rr]^{\id_{\pi}\times R_{\sigma,\pi}} && \pi\times\pi\times\sigma}
\]
Hence, $\tau\times\pi\subset\pi\times R_{\sigma,\pi}^{-1}(\tau)$.
It follows from Corollary \ref{cor: threetensor} that $R_{\sigma,\pi}^{-1}(\tau)=\sigma\times\pi$, i.e., the image of $R_{\sigma,\pi}$ is contained in $\tau$.
Since $\tau$ was arbitrary, we conclude that the image of $R_{\sigma,\pi}$ is irreducible and is equal to $\soc(\pi\times\sigma)$.

Applying this with $\sigma=1$ we obtain that $\pi$ is irreducible. Taking $\sigma=\pi$ we conclude that $\pi\times\pi$ is irreducible as well.

The other part is proved in a similar way. Analogously for the irreducibility of $\coss(\pi\times\sigma)$ and $\coss(\sigma\times\pi)$.
Finally, since both $\coss(\sigma\times\pi)$ and the image of $R_{\sigma,\pi}$ are irreducible, they coincide.
\end{proof}

\begin{corollary} (Cf. \cite[Corollary 3.3]{MR3314831}) \label{cor: LMconds}
The following conditions are equivalent for $0\ne\pi\in\Reps(G_n)$.
\begin{enumerate}
\item $\pi\times\pi$ is \SI.
\item $\pi\times\pi$ is irreducible.
\item $\End_{G_{2n}}(\pi\times\pi)=\C$.
\item $R_{\pi,\pi}$ is a scalar.
\end{enumerate}
\end{corollary}

\begin{proof}
Trivially, 2$\implies$1 and 3$\implies$4. By Schur's lemma 2$\implies$3. By Corollary \ref{cor: maincor} 4$\implies$2. It remains to show that 1$\implies$3.
Suppose that $\pi\times\pi$ is \SI\ and let $\pi_0=\soc(\pi\times\pi)$.
Let $A\in\End(\pi\times\pi)$. Since $\pi_0$ is irreducible, $A$ acts as a scalar $\lambda\in\C$ on $\pi_0$. Let $A'=A-\lambda\id_{\pi\times\pi}$.
Then $\Ker A'\supset\pi_0$. On the other hand $\operatorname{Im}(A')$, if non-zero, must contain $\pi_0$.
This would contradict the assumption that $\pi_0$ occurs with multiplicity one in $\JH(\pi\times\pi)$.
Hence $A\equiv\lambda\id_{\pi\times\pi}$ as required.
\end{proof}

We say that $\pi$ is \LM\ if it satisfies the conditions of Corollary \ref{cor: LMconds}.
We denote by $\IrrS\subset\Irr$ the set of \LM\ representations.
Note that $\pi\in\IrrS$ if and only if $\pi^\vee\in\IrrS$.

\begin{remark}
We do not know whether $\pi\times\pi$ is semisimple for every $\pi\in\Irr$, or even whether
$R_{\pi,\pi}$ is always an isomorphism (or equivalently, whether $R_{\pi,\pi}\circ R_{\pi,\pi}$ is a non-zero scalar).
\end{remark}

We write $\pi^{\times n}=\pi\times\dots\times\pi$ ($n$ times).

\begin{corollary} (cf. \cite[Corollary 3.4 and p. 391]{MR3314831}) \label{cor: pi1pi2LM}
Suppose that $\pi_1,\pi_2\in\IrrS$ and $R_{\pi_1,\pi_2}$ is an isomorphism. Then $\pi_1\times\pi_2\in\IrrS$ (and in particular $\pi_1\times\pi_2\in\Irr$).
If $\pi\in\IrrS$ then $\pi^{\times n}\in\IrrS$ for all $n\ge1$.
\end{corollary}

\begin{proof}
Let $\pi_3=\pi_1\times\pi_2$. We have
\[
M_{\pi_3,\pi_3}(s)=(\id_{\pi_1}\times M_{\pi_1,\pi_2}(s)\times\id_{\pi_2\abs{\cdot}^s})\circ(M_{\pi_1,\pi_1}(s)\times M_{\pi_2,\pi_2}(s))
\circ(\id_{\pi_1\abs{\cdot}^s}\times M_{\pi_2,\pi_1}(s)\times\id_{\pi_2}).
\]
By assumption, $R_{\pi_1,\pi_1}$, $R_{\pi_2,\pi_2}$ and $R_{\pi_1,\pi_2}$ are isomorphisms and therefore
(cf. Lemma \ref{lem: elemM} part \ref{part: ifisom}) $r_{\pi_3,\pi_3}=\sum_{i,j=1,2}r_{\pi_i,\pi_j}$ and
\[
R_{\pi_3,\pi_3}=(\id_{\pi_1}\times R_{\pi_1,\pi_2}\times\id_{\pi_2})\circ(R_{\pi_1,\pi_1}\times R_{\pi_2,\pi_2})
\circ(\id_{\pi_1}\times R_{\pi_2,\pi_1}\times\id_{\pi_2}).
\]
By assumption, this is proportional to $\id_{\pi_1}\times (R_{\pi_1,\pi_2}\circ R_{\pi_2,\pi_1})\times\id_{\pi_2}$
which is a scalar by Lemma \ref{lem: elemM} part \ref{part: compR}. The first part follows.

Similarly, the second part follows from Corollary \ref{cor: LMconds} and the fact that $R_{\pi^{\times n},\pi^{\times n}}$ is a non-zero scalar if $R_{\pi,\pi}$ is.
\end{proof}

We can slightly strengthen Corollary \ref{cor: maincor}. We are grateful to Max Gurevich for this observation.
\begin{lemma} (cf. \cite[Theorem 3.1]{1412.8106}) \label{lem: LM}
Suppose that $\pi\in\IrrS$.
Then for any $\sigma\in\Irr$, $\pi\times\sigma$ and $\sigma\times\pi$ are \SI.
\end{lemma}

\begin{proof}
We prove that $\pi\times\sigma$ is \SI. The other assertion is proved similarly.
Let $\tau=\soc(\pi\times\sigma)$ and $\tau'=\Ker R_{\sigma,\pi}\subset\sigma\times\pi$. We already know that $\tau$ is the image of $R_{\sigma,\pi}$
and $\tau$ is irreducible. Since $\JH(\sigma\times\pi)=\JH(\pi\times\sigma)$, it remains to show that $\tau$ does not occur in $\JH(\tau')$.

First note that $R_{\pi\times\sigma,\pi}\rest_{\tau\times\pi}\ne 0$, that is (Lemma \ref{lem: elemM} part \ref{part: nonzeroR}),
$r_{\pi\times\sigma,\pi}=r_{\tau,\pi}$. For otherwise, we would have (by \eqref{eq: radd}) $\tau\times\pi\subset\Ker(R_{\pi\times\sigma,\pi})=\pi\times\tau'$
which contradicts Corollary \ref{cor: threetensor} since $\tau'\subsetneq\sigma\times\pi$.

Since $\lambda R_{\sigma\times\pi,\pi}=R_{\sigma,\pi}\times\id_{\pi}$
(where as before $R_{\pi,\pi}=\lambda^{-1}\id_{\pi\times\pi}$),
the restriction of $R_{\sigma\times\pi,\pi}$ to $\tau'\times\pi$ vanishes.
Hence, by Lemma \ref{lem: elemM} parts \ref{part: nonzeroR} and \ref{part: ifisom},
\[
r_{\tau',\pi}<r_{\sigma\times\pi,\pi}=r_{\sigma,\pi}+r_{\pi,\pi}=r_{\pi\times\sigma,\pi}=r_{\tau,\pi}
\]
by the above. Thus, by Lemma \ref{lem: elemM} part \ref{part: sqle} $\tau$ cannot occur in $\JH(\tau')$ as required.
\end{proof}

\begin{remark}
In \cite{MR3573961} we defined a ``left multiplier'' to be an irreducible representation such that $\pi\times\sigma$ is \SI\ for any irreducible $\sigma$.
In view of Lemma \ref{lem: LM} this is equivalent to the conditions of Corollary \ref{cor: LMconds}.
\end{remark}

The following result gives a recursive way to deduce $\square$-irreducibility.
\begin{lemma} \label{lem: albertoidea}
Suppose that $\pi\hookrightarrow\pi_1\times\pi_2$ and $\pi\times\pi_1$ is irreducible.
Then $\pi\in\IrrS$ provided that $\pi_2\in\IrrS$.
\end{lemma}

\begin{proof}
Since $\pi\times\pi_1\in\Irr$ it follow from Lemma \ref{lem: LM} that $\Pi:=\pi\times\pi_1\times\pi_2$ is \SI\ provided that $\pi_2\in\IrrS$.
Since $\pi\times\pi\hookrightarrow\Pi$ we infer that $\pi\times\pi$ is \SI.
Hence $\pi\in\IrrS$ by Corollary \ref{cor: LMconds}.
\end{proof}

\subsection{}
One of the most fundamental facts in the representation theory of the groups $G_n$ is that $\Cusp\subset\IrrS$ \cite{MR0499010}.
For any $\pi\in\Reps(G_n)$ and $\rho\in\Cusp$ with $d=\deg(\rho)$ define
\begin{gather*}
\lmlt_\rho(\pi):=\max\{i\ge0:\jac_{di,n-di}(\pi)_{i\rho;*}\ne0\}=
\max\{i\ge0:\rho^{\times i}\otimes\sigma\le\jac_{di,n-di}(\pi)\text{ for some }\sigma\ne0\},\\
\rmlt_\rho(\pi):=\max\{i\ge0:\jac_{n-di,di}(\pi)_{*;i\rho}\ne0\}=\max\{i\ge0:\sigma\otimes\rho^{\times i}\le\jac_{n-di,di}(\pi)\text{ for some }\sigma\ne0\},
\end{gather*}
and let
\[
\lnrset(\pi)=\{\rho\in\Cusp:\lmlt_\rho(\pi)>0\},
\rnrset(\pi)=\{\rho\in\Cusp:\rmlt_\rho(\pi)>0\}
\subset\supp\pi.
\]
If $\pi\in\Irr$ then $\rho\in\lnrset(\pi)$ (resp., $\rho\in\rnrset(\pi)$) if and only if there exists $\pi'\in\Irr$, necessarily unique,
such that $\pi\hookrightarrow\rho\times\pi'$ (resp., $\pi\hookrightarrow\pi'\times\rho$).
More generally, for  $d=\deg\rho$ and $m=\lmlt_\rho(\pi)$ (resp., $m=\rmlt_\rho(\pi)$) the representation
$\jac_{(md,n-md)}(\pi)_{m\cdot\rho;*}$ (resp., $\jac_{(n-md,md)}(\pi)_{*;m\cdot\rho}$) is irreducible, i.e.
\[
\jac_{(md,n-md)}(\pi)_{m\cdot\rho;*}=\rho^{\times m}\otimes\pi'
\text{  (resp., $\jac_{(n-md,md)}(\pi)_{*;m\cdot\rho}=\pi'\otimes\rho^{\times m}$)}
\]
where $\pi'\in\Irr G_{n-md}$ (\cite{MR2306606}). In particular,
\[
\pi\hookrightarrow\rho^{\times m}\times\pi'
\text{  (resp., $\pi\hookrightarrow\pi'\times\rho^{\times m}$)}.
\]
Moreover, $\lmlt_\rho(\pi')=0$ (resp., $\rmlt_\rho(\pi')=0$). We write $\lderiv_\rho(\pi)=\pi'$ (resp., $\rderiv_\rho(\pi)=\pi'$).

The following result easily follows from the geometric lemma of Bernstein--Zelevinsky \cite{MR0579172} and Frobenius reciprocity.
\begin{lemma} \label{lem: lnrprop}
Let $\pi,\pi'\in\Reps(G_n)$, $\pi_i\in\Irr$, $i=1,2$ and $\rho\in\Cusp$. Then
\begin{enumerate}
\item If $\pi'\le\pi$ then $\lnrset(\pi')\subset\lnrset(\pi)$ and $\rnrset(\pi')\subset\rnrset(\pi)$.
\item $\lnrset(\pi_1\times\pi_2)=\lnrset(\pi_1)\cup\lnrset(\pi_2)$ and $\rnrset(\pi_1\times\pi_2)=\rnrset(\pi_1)\cup\rnrset(\pi_2)$.
\item $\lmlt_\rho(\pi_1\times\pi_2)=\lmlt_\rho(\pi_1)+\lmlt_\rho(\pi_2)$ and similarly for $\rmlt_\rho$.
\item $\rho^{\times \lmlt_\rho(\pi_1\times\pi_2)}\otimes\lderiv_\rho(\pi_1)\times\lderiv_\rho(\pi_2)\le\jac(\pi_1\times\pi_2)$ and
$\rderiv_\rho(\pi_1)\times\rderiv_\rho(\pi_2)\otimes\rho^{\times \rmlt_\rho(\pi_1\times\pi_2)}\le\jac(\pi_1\times\pi_2)$.
\item If $\pi\hookrightarrow\pi_1\times\pi_2$ then $\lnrset(\pi)\supset\lnrset(\pi_1)$ and $\rnrset(\pi)\supset\rnrset(\pi_2)$.
\end{enumerate}
\end{lemma}

\begin{corollary} (cf. \cite[Proposition 4.20]{1502.06714})
Suppose that $\pi_1\times\pi_2$ is irreducible. Then for any $\rho\in\Cusp$ we have
\[
\lderiv_\rho(\pi_1\times\pi_2)=\lderiv_\rho(\pi_1)\times\lderiv_\rho(\pi_2).
\]
In particular, $\lderiv_\rho(\pi_1)\times\lderiv_\rho(\pi_2)$ is irreducible. Similarly for $\rderiv_\rho$.
\end{corollary}

\begin{corollary} (cf. \cite[Corollary 4.21]{1502.06714})
If $\pi\in\IrrS$ then so are $\lderiv_{\rho}(\pi)$ and $\rderiv_\rho(\pi)$.
\end{corollary}

We say that $\pi'\in\Irr$ is a descendant of $\pi\in\Irr$ if there exists a sequence $\pi_0,\pi_1,\dots,\pi_n\in\Irr$, $n>0$
such that $\pi_0=\pi$, $\pi_n=\pi'$ and for all $i=1,\dots,n$, $\pi_i=\lderiv_{\rho_i}(\pi_{i-1})$ or $\pi_i=\rderiv_{\rho_i}(\pi_{i-1})$
for some $\rho_i\in\lnrset(\pi)$ (resp. $\rho_i\in\rnrset(\pi)$).

\begin{corollary} \label{cor: derisLM}
If $\pi\in\IrrS$ then so is any descendant of $\pi$.
\end{corollary}

By \cite[Lemma 2.5]{MR3573961} we also conclude

\begin{corollary} \label{cor: extractrho}
Suppose that $\pi_1\in\IrrS$, $\pi_2\in\Irr$ and $\rho\in\Cusp$. Let $m_i=\lmlt_\rho(\pi_i)$, $\pi'_i=\lderiv_\rho(\pi_i)$, $i=1,2$ and $m=m_1+m_2$.
Then $\rho^{\times m}\times\pi'_1\times\pi'_2$ is \SI. Hence, if moreover $m_2=0$ or $\rho\times\pi_1$ is irreducible, so that
\[
\pi_1\times\pi_2\hookrightarrow\rho^{\times m_2}\times\pi_1\times\pi_2'\hookrightarrow\rho^{\times m}\times\pi_1'\times\pi'_2,
\]
then
\[
\soc(\pi_1\times\pi_2)=\soc(\rho^{\times m}\times\soc(\pi'_1\times\pi'_2)).
\]
\end{corollary}

\section{Classification} \label{sec: classification}
We recall the classification of $\Irr$ which goes back to Bernstein--Zelevinsky and Zelevinsky in the case where $D=F$ \cite{MR0579172, MR584084}.
We refer the reader to \cite{MR3573961} and the references therein for more details and the history. Here we just record the facts and set the notation.

\subsection{}
For any $\rho\in\Cusp$ there exists a unique positive real number\footnote{In fact, $s_\rho$ is an integer, but this will not play any role here.} $s_\rho$ such that
$\rho\abs{\cdot}^{s_\rho}\times\rho$ is reducible.
(If $D=F$ then $s_\rho=1$.)
We write $\nu_\rho=\abs{\cdot}^{s_\rho}$, $\rshft{\rho}=\rho\nu_\rho$, $\lshft{\rho}=\rho\nu_\rho^{-1}$.
Note that $\nu_{\rho^\vee}=\nu_\rho$ and $\nu_{\rho\chi}=\nu_\rho$ for any character $\chi$ of $F^*$.

Moreover, if $\rho_1, \rho_2\in\Cusp$ then $\rho_1\times\rho_2$ is reducible if and only if $\rho_2$ is equal
to either $\rshft\rho_1$ or $\lshft\rho_1$.

A \emph{segment} is a finite non-empty subset of $\Cusp$ of the form $\Delta=\{\rho_1,\dots,\rho_k\}$ where $\rho_{i+1}=\rshft{\rho}_i$, $i=1,\dots,k-1$.
We write $b(\Delta)=\rho_1$, $e(\Delta)=\rho_k$ and $\deg\Delta=\sum_{i=1}^k\deg\rho_i=k\cdot\deg\rho_1$.
Since $\Delta$ is determined by $b(\Delta)$ and $e(\Delta)$ we often write $\Delta$ as $[b(\Delta),e(\Delta)]$.

Let $\Delta=\{\rho_1,\dots,\rho_k\}$ be a segment as before.
Then the representation $\rho_1\times\dots\times\rho_k\in\Reps(G_{\deg\Delta})$ is \SI.
We denote $\zele{\Delta}=\soc(\rho_1\times\dots\times\rho_k)\in\Irr G_{\deg\Delta}$.
For convenience, we also set $\zele{\emptyset}=1$.
We have
\[
\jac_{(i\deg\rho_1,(k-i)\deg\rho_1)}(\zele{\Delta})=\zele{\{\rho_1,\dots,\rho_i\}}\otimes\zele{\{\rho_{i+1},\dots,\rho_k\}},\ \ 0\le i\le k.
\]
Also, $\zele{\Delta}^\vee=\zele{\Delta^\vee}$ where $\Delta^\vee=\{\rho_k^\vee,\dots,\rho_1^\vee\}$.
We set
\begin{gather*}
\lshft{\Delta}=\{\lshft{\rho}_1,\dots,\lshft{\rho}_k\},\ \ \rshft{\Delta}=\{\rshft{\rho}_1,\dots,\rshft{\rho}_k\},\\
\Delta^+=[b(\Delta),e(\rshft{\Delta})],\ ^+\Delta=[b(\lshft\Delta),e(\Delta)],\ \Delta^-=[b(\Delta),e(\lshft{\Delta})],\ ^-\Delta=[b(\rshft\Delta),e(\Delta)].
\end{gather*}
Note that $\Delta^-$ or $^-\Delta$ can be empty.

Given two segments $\Delta_1$, $\Delta_2$ we write $\Delta_1\prec\Delta_2$ if $b(\Delta_1)\notin\Delta_2$, $b(\lshft{\Delta}_2)\in\Delta_1$ and
$e(\Delta_2)\notin\Delta_1$. In this case $\soc(\zele{\Delta_1}\times\zele{\Delta_2})=\zele{\Delta_1'}\times\zele{\Delta'_2}$ where
$\Delta_1'=\Delta_1\cup\Delta_2$, $\Delta_2'=\Delta_1\cap\Delta_2$ (the latter is possibly empty).
If either $\Delta_1\prec\Delta_2$ or $\Delta_2\prec\Delta_1$ then we say that $\Delta_1$ and $\Delta_2$ are linked.
In this case we say that $(\Delta'_1,\Delta'_2)$ as above is the offspring of $(\Delta_1,\Delta_2)$.
Note that $\{b(\Delta_1'),b(\Delta'_2)\}=\{b(\Delta_1),b(\Delta_2)\}$ and $\{e(\Delta_1'),e(\Delta'_2)\}=\{e(\Delta_1),e(\Delta_2)\}$.
(By convention, if $\Delta'_2=\emptyset$ in the case at hand, we write $b(\Delta_2')=b(\Delta_{3-j})$ and $e(\Delta_2')=e(\Delta_j)$ if $\Delta_j\prec\Delta_{3-j}$.)

A multisegment is a formal sum $\m=\Delta_1+\dots+\Delta_k$ of segments.
(We omit empty segments from this sum.)
In other words, the set $\Mult$ of multisegments is the free commutative monoid generated by all segments.
Write $\supp\m=\cup_{i=1}^k\Delta_i$ and $\deg\m=\sum_{i=1}^k\deg\Delta_i$.
Assume that $\Delta_1,\dots,\Delta_k$ is a sequence of segments such that $\Delta_i\not\prec\Delta_j$ for all $i<j$.
(Any multisegment can be ordered this way.)
Then the representation
\[
\std(\m):=\zele{\Delta_1}\times\dots\times\zele{\Delta_k}\in\Reps(G_{\deg\m})
\]
is \SI\ and depends only on $\m=\Delta_1+\dots+\Delta_k$.
The main result of the classification is that the map
\[
\m\in\Mult\mapsto\zele{\m}:=\soc(\std(\m))\in\Irr G_{\deg\m}
\]
defines a bijection between $\Mult$ and $\Irr$.
We write the inverse bijection as $\pi\mapsto\m(\pi)$.

Following Zelevinsky, we write $\m\adj\n$ if $\m$ is obtained from $\n$ by replacing a pair of linked segments in $\n$ by its offspring.
The transitive closure of this relation is denoted by $\obt$. (In particular, $\m\obt\m$.)

We recall some basic properties of the Zelevinsky classification. Let $\m,\n\in\Mult$. Then
\begin{enumerate}
\setcounter{enumi}{-1}
\item $\zele{0}=1$.
\item $\supp\zele{\m}=\supp\m$.
\item $\zele{\m}\le\std(\n)$ if and only if $\m\obt\n$.
\item $\std(\m)$ is irreducible, i.e. $\zele{\m}=\std(\m)$, if and only if $\m$ is \emph{pairwise unlinked}, that is,
no two segments in $\m$ are linked.
\item $\zele{\m+\n}$ occurs with multiplicity one in $\JH(\zele{\m}\times\zele{\n})$.
\item In particular, if $\zele{\m}\times\zele{\n}$ is irreducible then $\zele{\m}\times\zele{\n}=\zele{\m+\n}$.
\item We write $\LI(\zele{\m},\zele{\n})$ (resp., $\RI(\zele{\m},\zele{\n})$) for the condition
$\soc(\zele{\m}\times\zele{\n})=\zele{\m+\n}$ (resp., $\coss(\zele{\m}\times\zele{\n})=\zele{\m+\n}$).
Thus, $\zele{\m}\times\zele{\n}$ is irreducible if and only if both $\LI(\zele{\m},\zele{\n})$ and $\RI(\zele{\m},\zele{\n})$.
\item The condition $\LI(\zele{\m},\zele{\n})$ is satisfied if $\Delta\not\prec\Delta'$ for any segment $\Delta$ of $\m$ and $\Delta'$ of $\n$.
\end{enumerate}

As a ring, $\Gr$ is freely generated by $\zele{\Delta}$ as $\Delta$ ranges over all segments.
Equivalently, $\Gr$ is freely generated as an abelian group by $\std(\m)$, $\m\in\Mult$ (as well as by $\zele{\m}$, $\m\in\Mult$).
The change of basis matrix is unitriangular with respect to $\obt$ and its coefficients are given by values of Kazhdan--Luzstig polynomials
with respect to the symmetric group -- see \S\ref{sec: KLid}.

\subsection{Auxiliary results}
\begin{definition} \label{def: detachable}
Let $\m=\Delta_1+\dots+\Delta_k\in\Mult$.
We say that $\Delta_i$ is a detachable segment of $\m$ if at least one of the following conditions holds:
\begin{subequations}
\begin{equation} \label{eq: deltafirst}
\Delta_i\not\prec\Delta_j\text{ and }\lshft{\Delta}_i\not\prec\Delta_j\text{ for all }j\ne i
\end{equation}
or,
\begin{equation} \label{eq: deltalast}
\Delta_j\not\prec\Delta_i\text{ and }\lshft{\Delta}_j\not\prec\Delta_i\text{ for all }j\ne i.
\end{equation}
\end{subequations}
\end{definition}

\begin{lemma} \label{lem: 1stred}
Suppose that $\Delta$ is a detachable segment of $\m\in\Mult$ and let $\m'=\m-\Delta$.
Assume that $\zele{\m}\in\IrrS$. Then $\zele{\m'}\in\IrrS$.
\end{lemma}

\begin{proof}
Suppose that \eqref{eq: deltafirst} holds.
Let $\pi=\zele{\m}$ and $\pi'=\zele{\m'}$. Then $\pi\hookrightarrow\zele{\Delta}\times\pi'$ by the first condition on $\Delta$.
Thus, by Frobenius reciprocity $\jac(\pi)\twoheadrightarrow\zele{\Delta}\otimes\pi'$.
Hence, by the geometric lemma
\begin{equation} \label{eq: occurs}
\zele{\Delta+\Delta}\otimes\pi'\times\pi'=\zele{\Delta}\times\zele{\Delta}\otimes\pi'\times\pi'\le\jac(\pi\times\pi).
\end{equation}
Assume that $\pi$ is \LM. Then $\pi\times\pi=\zele{\m+\m}\hookrightarrow\zele{\Delta+\Delta}\times\zele{\m'+\m'}$.
On the other hand, it is easy to see using the geometric lemma that the condition $\lshft{\Delta}\not\prec\Delta'$ for any segment $\Delta'$ of $\m'$ guarantees that
\[
\jac(\zele{\Delta+\Delta}\times\zele{\m'+\m'})_{\Delta+\Delta;*}=\zele{\Delta+\Delta}\otimes\zele{\m'+\m'}.
\]
It follows from \eqref{eq: occurs} that $\pi'\times\pi'=\zele{\m'+\m'}$, i.e., that $\pi'$ is \LM\ as required.

The argument with the condition \eqref{eq: deltalast} is similar.
\end{proof}

For $\rho\in\Cusp$ let $f_\rho:\Cusp\rightarrow\Cusp$ be the function given by
\[
f_\rho(\rho')=\begin{cases}\lshft{\rho'}&\text{if }\rho'=\rho\nu_\rho^l\text{ for some }l\in\Z_{>0},\\\rho'&\text{otherwise.}\end{cases}
\]
Thus, for any segment $\Delta$, $f_\rho(\Delta)$ is either $\Delta$, $\Delta^-$ or $\lshft{\Delta}$.

\begin{definition} \label{def: contractible}
Let $\m=\Delta_1+\dots+\Delta_k\in\Mult$ and $\rho\in\Cusp$. We say that $\m$ is $\rho$-contractible if for every $i$,
$\#(\Delta_i\cap\{\rho,\rshft\rho\})\ne1$, i.e., either $\{\rho,\rshft{\rho}\}\subset\Delta_i$ or $\Delta_i\cap\{\rho,\rshft{\rho}\}=\emptyset$.
In this case, we say that the $\rho$-contraction of $\m$ is $f_\rho(\Delta_1)+\dots+f_\rho(\Delta_k)$.

We call $\m$ contractible if it is $\rho$-contractible for some $\rho\in\supp\m$.
\end{definition}

The following assertion follows from Corollary \ref{cor: fzele} of \S\ref{sec: KLid}.

\begin{proposition} \label{prop: contract}
Suppose that $\m$ is $\rho$-contractible for some $\rho\in\Cusp$ and let $\m'$ be the $\rho$-contraction of $\m$.
Then $\zele{\m}\in\IrrS$ if and only if $\zele{\m'}\in\IrrS$.
\end{proposition}

For any $\m\in\Mult$ we write $\lnrset(\m)=\lnrset(\zele{\m})\subset\supp\m$ and for any $\rho\in\Cusp$ we write
$\lderiv_\rho(\m)=\m(\lderiv_\rho(\zele{\m}))$ and similarly for $\rderiv_\rho(\m)$.
We recall the following combinatorial recipe for $\lnrset(\m)$, $\lderiv_\rho(\m)$ and $\m(\soc(\rho\times\zele{\m}))$.

\begin{lemma} (\cite{MR2306606, MR2527415}) \label{lem: desclder}
For $\rho'\in\Cusp$ let $I_{\rho'}=\{i:b(\Delta_i)=\rho'\}$. Then there exists a subset $I\subset I_\rho$ and an injective function
$f:I\rightarrow I_{\rshft\rho}$ such that if $J=I_\rho\setminus I$ then we have the following properties.
\begin{enumerate}
\item $\Delta_i\prec\Delta_{f(i)}$ for all $i\in I$.
\item If $\Delta_i\prec\Delta_j$ with $i\in I$ and $j\notin f(I)$ then $^+\Delta_j\not\prec\Delta_{f(i)}$.
\item If $\Delta_j\prec\Delta_{j'}$ with $j\in J$ and $j'\in I_{\rshft\rho}$ then $j'\in f(I)$ and $\Delta_{f^{-1}(j')}\not\prec\,^-\Delta_j$.
\end{enumerate}
Moreover, we have the following.
\begin{enumerate}
\item $\rho\in\lnrset(\m)$ if and only if $J\ne\emptyset$.
\item $\lderiv_\rho(\m)=\m+\sum_{j\in J}(^-\Delta_j-\Delta_j)$. In particular, $\sum_{i\in I}\Delta_i$
is independent of $I$ and $f$.
\item $\soc(\rho\times\zele{\m})=\zele{\m+\{\rho\}}$ if $f(I)=I_{\rshft\rho}$ and otherwise,
$\soc(\rho\times\zele{\m})=\zele{\m-\Delta_j+\,^+\Delta_j}$ where $j\in I_{\rshft\rho}\setminus f(I)$ is such that
$\Delta_j\not\prec\,^+\Delta_r$ for all $r\in I_{\rshft\rho}\setminus f(I)$.
\end{enumerate}
\end{lemma}

For convenience we record the following special cases.
\begin{lemma} \label{lem: spclcaseextrho}
Let $\m=\Delta_1+\dots+\Delta_k$ and $\rho\in\Cusp$. Let $n_\rho=\#\{i:b(\Delta_i)=\rho\}$ and similarly for $n_{\rshft\rho}$.
\begin{enumerate}
\item If $n_\rho=0$ then $\rho\notin\lnrset(\zele{\m})$.
\item If $n_\rho>n_{\rshft\rho}$ then $\rho\in\lnrset(\zele{\m})$.
\item Suppose that $n_\rho=1$ and let $\Delta$ be the segment of $\m$ such that $b(\Delta)=\rho$.
Then $\rho\in\lnrset(\zele{\m})$ if and only if there does not exist $\Delta'$ in $\m$ such that $b(\Delta')=\rshft\rho$ and $\Delta\prec\Delta'$.
In this case $\lderiv_\rho(\m)=\m-\Delta+\,^-\Delta$.
\item Suppose that $n_\rho=2$ and $n_{\rshft\rho}=1$. Let $s$ and $l$ be the indices such that $b(\Delta_s)=b(\Delta_l)=\rho$ with
$\Delta_s\subset\Delta_l$ and let $j$ be such $b(\Delta_j)=\rshft\rho$. Assume that $\Delta_l\prec\Delta_j$.
Then $\rho\in\lnrset(\m)$ and $\lderiv_\rho(\m)=\m-\Delta_s+\,^-\Delta_s$.
\end{enumerate}
\end{lemma}

\begin{lemma} \label{lem: soctimesrho}
Let $\m=\Delta_1+\dots+\Delta_k$ and $\rho\in\Cusp$. Let $n_\rho$ and $n_{\rshft\rho}$ be as before.
\begin{enumerate}
\item Suppose that $n_{\rshft\rho}=1$ and let $\Delta$ be the segment of $\m$ such that $b(\Delta)=\rshft\rho$.
Then
\[
\soc(\rho\times\zele{\m})=\begin{cases}\zele{\m+\{\rho\}}&\text{if $\exists\Delta'$ in $\m$ such that $b(\Delta')=\rho$ and }\Delta'\prec\Delta,\\
\zele{\m-\Delta+\,^+\Delta}&\text{otherwise.}\end{cases}
\]
\item Suppose that $n_\rho\le1$ and $n_{\rshft\rho}=2$. Let $s$ and $l$ be the indices such that $b(\Delta_s)=b(\Delta_l)=\rshft\rho$ with
$\Delta_s\subset\Delta_l$. If there exists $j$ such $b(\Delta_j)=\rho$ then assume that $\Delta_j\prec\Delta_s$.
Then $\soc(\rho\times\zele{\m})=\zele{\m-\Delta_l+\,^+\Delta_l}$.
\end{enumerate}
\end{lemma}

\subsection{Reduction to cuspidal lines} \label{sec: fixrho}
An equivalence class for the equivalence relation on $\Cusp$ generated by $\rho\sim\rshft\rho$ is called a \emph{cuspidal line}.
Thus, the cuspidal line containing $\rho\in\Cusp$ is $\Z_\rho:=\{\rho\nu_\rho^n:n\in\Z\}$.
For any cuspidal line $\cspline$ consider the Serre ring subcategory $\Reps_\cspline$ of $\Reps$ consisting of the representations whose supercuspidal support
is contained in $\cspline$. Let $\Gr_\cspline$ be the Grothendieck ring of $\Reps_\cspline$.
The following assertions are consequences of Zelevinsky classification:
\begin{enumerate}
\item As a commutative ring, $\Gr$ (resp., $\Gr_\cspline$) is freely generated (over $\Z$) by the images of $\zele{\Delta}$,
where $\Delta$ varies over all segments (resp., the segments contained in $\cspline$).
\item If $\pi_i\in\Irr\Reps_{\cspline_i}$ with $\cspline_1,\dots,\cspline_r$ distinct then $\pi_1\times\dots\times\pi_r$ is irreducible.
\item Conversely, any $\pi\in\Irr$ can be written uniquely (up to permutation) as $\pi=\pi_1\times\dots\times\pi_r$ where $\pi_i\in\Irr\Reps_{\cspline_i}$ and
$\cspline_1,\dots,\cspline_r$ distinct.
\item $\Gr$ is the coproduct (in the category of commutative rings) over all cuspidal lines of $\Gr_\cspline$, i.e.,
$\Gr$ is the inductive limit over finite sets $\{\cspline_1,\dots,\cspline_r\}$ of $\otimes_{i=1}^r\Gr_{\cspline_i}$.
\end{enumerate}
In practice, this enables us to reduce questions about $\Irr$ to $\Irr\Reps_\cspline$.
For instance, if $\pi=\pi_1\times\dots\times\pi_r$ as above then $\pi\in\IrrS$ if and only if $\pi_i\in\IrrS$ for all $i$.

Let $\rho\in\Cusp$ and denote by $\Mult_\rho$ the submonoid of multisegments supported in $\Z_\rho$.
Let $D'$ be another local non-archimiedean division algebra (not necessarily with center $F$) and let $\rho'$ be an irreducible supercuspidal
representation of some $\GL_m(D')$, $m>0$.
Define $\phi_{\rho,\rho'}:\Z_\rho\rightarrow\Z_{\rho'}$ by $\phi_{\rho,\rho'}(\rho\nu_\rho^n)=\rho'\nu_{\rho'}^n$.
(Thus, $\phi_{\rho,\rho'}$ is the unique bijection between $\Z_\rho$ and $\Z_{\rho'}$ which commutes with $\rshft{}$ and which maps $\rho$ to $\rho'$.)
It induces a bijection $\phi_{\rho,\rho'}:\Mult_\rho\rightarrow\Mult_{\rho'}$.
Sometimes it will be convenient to use the following fact.
\begin{theorem} \label{thm: indepcspline}
There is an equivalence of ring categories between $\Reps_\rho$ and $\Reps_{\rho'}$ taking
$\rho$ to $\rho'$ and $\rshft\rho$ to $\rshft{\rho'}$, hence taking
$\zele{\m}$ to $\zele{\phi_{\rho,\rho'}(\m)}$ and $\std(\m)$ to $\std(\phi_{\rho,\rho'}(\m))$ for any $\m\in\Mult_\rho$.
\end{theorem}
This follows from the explication of the Bernstein components of $\Reps$ as categories of finite-dimensional representations of Hecke algebras which in turn follows
either by the results of \cite{MR2827179} or by type theory \cite{MR1204652, MR2081220, MR2188448, MR2216835, MR2427423, MR2889743, MR2946230}.

In principle, one can circumvent the use of Theorem \ref{thm: indepcspline} for the purpose of the paper.
However, we will use it sporadically in \S\ref{sec: basicases} in order to simplify some inessential aspects of the argument.

\begin{remark}
Let $I$ be a finite set of cuspidal lines and let $\Reps_I$ be the Serre ring subcategory of $\Reps$ consisting of the representations whose supercuspidal support
is contained in $\cup I$. Clearly, $\Reps$ is the inductive limit of the $\Reps_I$'s as $I$ varies over the directed set of finite sets of cuspidal lines
(with respect to inclusion). Once can show that if $I=\{\cspline_1,\dots,\cspline_r\}$ then parabolic induction
gives rise to an equivalence of categories of the tensor product of $\Reps_{\cspline_1},\dots,\Reps_{\cspline_r}$ in the sense of \cite[\S5]{MR1106898}
with $\Reps_I$. We will not use this fact here.
\end{remark}

From now on we fix $\rho\in\Cusp$ and for simplicity write $\Reps_\rho=\Reps_{\Z_\rho}$, $\Gr_\rho=\Gr_{\Z_{\rho}}$, $\Irr_\rho=\Irr\Reps_\rho$.
We will only consider multisegments in $\Mult_\rho$.
We identify segments supported in $\Z_\rho$ with sets of integers of the form
$[a,b]=\{n\in\Z:a\le n\le b\}$ (with $a,b\in\Z$) by $[a,b]_\rho=\{\rho\nu_\rho^n:n\in [a,b]\}$.
We will also write $[a]=[a,a]$ for brevity.
If $\rho$ is clear from the context (which will often be the case) then we suppress it from the notation.

It will be convenient to use the convention that
\begin{equation} \label{eq: convention}
\zele{[a_1,b_1]+\dots+[a_k,b_k]}=\std([a_1,b_1]+\dots+[a_k,b_k])=0\text{ if $a_i>b_i+1$ for some $i$}.
\end{equation}



We order the segments supported in $\Z_\rho$ right-lexicographically, namely we write $[a_1,b_1]<_e[a_2,b_2]$ if either
$b_1<b_2$ or $b_1=b_1$ and $a_1<a_2$. Similarly for the left-lexicographic relation $<_b$.

Given $\m,\n\in\Mult_\rho$ we write $\n<_b\m$ if $\Delta'<_b\Delta$ for any segment $\Delta$ of $\m$ and $\Delta'$ of $\n$.
This implies that $\soc(\zele{\m}\times\zele{\n})=\zele{\m+\n}$.

For later use, we mention the following result which follows from \cite[Lemma 4.11]{MR3573961}.
\begin{lemma} \label{lem: extractsegment}
Let $\m_1,\m_2\in\Mult_\rho$ and $\pi_i=\zele{\m_i}$, $i=1,2$. Assume that the maximal segment $\Delta$ of $\m_1$
with respect to $<_b$ occurs with multiplicity one in $\m_1$ and $\Delta'<_b\Delta$ for any segment $\Delta'$ of $\m_2$.
Assume that $\pi'_1\times\pi_2$ is \SI\ where $\pi'_1=\zele{\m_1-\Delta}$. Then $\zele{\Delta}\times\pi_1'\times\pi_2$ is \SI\  and hence
\[
\m(\soc(\pi_1\times\pi_2))=\Delta+\m(\soc(\pi'_1\times\pi_2)).
\]
The same holds if $<_b$ is replaced by $<_e$. Dually, suppose that the minimal segment $\Delta$ of $\m_2$
with respect to $<_b$ occurs with multiplicity one in $\m_2$ and that $\Delta<_b\Delta'$ for any segment $\Delta'$ of $\m_1$.
Assume that $\pi_1\times\pi'_2$ is \SI\ where $\pi'_2=\zele{\m_2-\Delta}$. Then $\pi_1\times\pi_2'\times\zele{\Delta}$ is \SI\ and hence
\[
\m(\soc(\pi_1\times\pi_2))=\Delta+\m(\soc(\pi_1\times\pi'_2)).
\]
Similarly if $<_b$ is replaced by $<_e$.
\end{lemma}

Recall that a \emph{ladder} is a multisegment of the form $\m=[a_1,b_1]+\dots+[a_k,b_k]$ where
$a_1>\dots>a_k$ and $b_1>\dots>b_k$. The corresponding irreducible representation $\zele{\m}$ is called a ladder representation.
It is known that a ladder representation is \LM\ \cite{MR3573961}.

We will also need the following result which follows from Frobenius reciprocity and the description of the Jacquet modules
of a ladder representation \cite{MR2996769}.
\begin{lemma} \label{lem: chopladders}
Let $\m$ be a ladder as above and let $c_1,\dots,c_k\in\Z$ be such that $a_i\le c_i\le b_i+1$ for all $i$ and $c_1>\dots>c_k$. Then
\[
\zele{\m}=\soc(\zele{\sum_i[a_i,c_i-1]}\times\zele{\sum_i[c_i,b_i]}).
\]
\end{lemma}

One of the main results of \cite{MR3573961} is the description of $\soc(\pi\times\sigma)$ when $\pi$ is a ladder representation and $\sigma$ is irreducible.
We will recall an important consequence of this description but we first make a definition which makes sense
for any pair of multisegments and which we will revisit in the next section.

\begin{definition} \label{def: XandY}
Let $\m=\Delta_1+\dots+\Delta_k$ and $\n=\Delta'_1+\dots+\Delta'_l$ be two multisegments.
Let $\X_{\m;\n}=\{(i,j):\Delta_i\prec\Delta'_j\}$, $\Y_{\m;\n}=\{(i,j):\lshft{\Delta}_i\prec\Delta'_j\}$
and let $\rltn$ be the relation between $\X_{\m;\n}$ and $\Y_{\m;\n}$ given by
\[
(i_1,j_1)\rltn(i_2,j_2)\text{ if either }\begin{cases}i_1=i_2\text{ and }\Delta'_{j_2}\prec\Delta'_{j_1},\text{ or }\\
j_1=j_2\text{ and }\Delta_{i_1}\prec\Delta_{i_2}.\end{cases}
\]
A $\rltn$-matching is an injective function $f:\X_{\m;\n}\rightarrow \Y_{\m;\n}$ such that $x\rltn f(x)$ for all $x\in \X_{\m;\n}$.
We write $\LC(\m,\n)$ for the condition that there exists a $\rltn$-matching from $\X_{\m;\n}$ to $\Y_{\m;\n}$.
\end{definition}

\begin{theorem}\cite{MR3573961} \label{thm: laddercomb}
Suppose that $\pi=\zele{\m}$ is a ladder and $\sigma=\zele{\n}\in\Irr$.
Then $\LI(\pi,\sigma)$ if and only if $\LC(\m,\n)$. Similarly, $\RI(\pi,\sigma)$ if and only if $\LC(\n,\m)$.
Thus $\pi\times\sigma$ is irreducible if and only if both $\LC(\m,\n)$ and $\LC(\n,\m)$.
\end{theorem}

\subsection{The Zelevinsky involution} \label{sec: zeleinvo}

The combinatorial analogue $\m\mapsto\m^\#$ of the Zelevinsky involution was defined by M\oe glin--Waldspurger \cite{MR863522}.
(See also \cite{MR1371654} for an alternative description.)
For $0\ne\m=\Delta_1+\dots+\Delta_k$ with $\Delta_1\ge_e\dots\ge_e\Delta_k$ define $l>0$ and indices $1=i_1<\dots<i_l$ recursively by
\[
i_{j+1}=\min\{i:\Delta_i\prec\Delta_{i_j}\text{ and }e(\Delta_i)=e(\lshft{\Delta}_{i_j})\}
\text{ if such an index exists, otherwise $l=j$.}
\]
Set $\del(\m)=[e(\Delta_{i_l}),e(\Delta_1)]$ and
\[
\m^-=\m+\sum_{j=1}^l(\Delta_{i_j}^--\Delta_{i_j}).
\]
We also write $\zele{\m}^-=\zele{\m^-}$ and $\del(\zele{\m})=\del(\m)$.

\begin{remark}
The multisegment $\m$ is uniquely determined by $\m^-$ and $\del(\m)$. Indeed, writing $\m^-=\Delta'_1+\dots+\Delta'_l$
with $\Delta'_1\ge_e\dots\ge_e\Delta'_l$ and $\del(\m)=(\rho_1,\dots,\rho_s)$, define $0\le r\le s$ and $j_1>\dots>j_r$ by
\begin{gather*}
j_1=\max\{j:e(\Delta'_j)=\lshft{\rho}_1\}\text{ if defined, otherwise }r=0,\\
j_{i+1}=\max\{j:\Delta'_{j_i}\prec\Delta'_j\text{ and }e(\Delta'_j)=\lshft{\rho}_{i+1}\}\text{ if defined, otherwise }r=i.
\end{gather*}
Then
\[
\m=\m^-+\sum_{i=1}^r({\Delta'_{j_i}}^+-\Delta'_{j_i})+\sum_{i=r+1}^s\{\rho_i\}.
\]
\end{remark}

The map $\m\mapsto\m^\#$ is defined recursively by $0^\#=0$ and
\[
\m^\#=(\m^-)^\#+\del(\m),\ \ \m\ne0.
\]
We may then define $\zele{\m}^t=\zele{\m^\#}$.
This definition extends by linearity to $\Gr$ and determines an involution of graded rings \cite{MR863522, MR1285969, MR1390967, MR2349436}.
In particular,
\begin{proposition} \label{prop: ZIred}
Suppose that $\pi\in\IrrS$. Then $\pi^t\in\IrrS$.
\end{proposition}

We refer the reader to \cite{1701.07329} for a recent, more categorical point of view of Zelevinsky involution.


\begin{lemma} (\cite[Lemma 4.13]{MR3573961}) \label{lem: suppn>suppm}
Let $\m,\n\in\Mult_\rho$. Assume that $\max\supp\m<\max\supp\n$ and that $\zele{\m}\in\IrrS$.
Then $\soc(\zele{\m}\times\zele{\n})$ is the irreducible representation $\pi$ satisfying
\[
\pi^-=\soc(\zele{\m}\times\zele{\n^-})\text{ and }\del(\pi)=\del(\n).
\]
\end{lemma}

\subsection{Regular multisegments}
In the second part of the paper we will specialize to a certain class of multisegments.
Namely, we say that a multisegment $\m=\Delta_1+\dots+\Delta_k$ is \emph{regular} if $b(\Delta_1),\dots,b(\Delta_k)$ are distinct and
$e(\Delta_1),\dots,e(\Delta_k)$ are distinct. Note that if $\m$ is regular and $\n\obt\m$ then $\n$ is also regular.

A sub-multisegment of a multisegment $\m$ is a multisegment $\m_1$ for which there exists a multisegment $\m_2$ such that $\m=\m_1+\m_2$.
Clearly, a sub-multisegment of a regular multisegment is also regular.
The same is true for the $\rho$-contraction of a $\rho$-contractible regular multisegment. However, the Zelevinsky involution
does not preserve regularity.

\section{A variant of a conjecture of Geiss--Leclerc--Schr\"oer} \label{sec: GLSconj}
\subsection{}
There is a more geometric way, also due to Zelevinsky, to think about the Zelevinsky classification \cite{MR617466, MR783619, MR863522, MR1648174}.
Namely, consider pairs $(V,A)$ where $V$ is a finite-dimensional $\Cusp$-graded $\C$-vector space $V=\oplus_{\rho\in\Cusp}V_\rho$
and $A$ is a (necessarily nilpotent) $\C$-linear endomorphisms of $V$ such that $A(V_\rho)\subset V_{\rshft{\rho}}$ for all $\rho\in\Cusp$.
(We denote by $E_\rightarrow(V)$ the space of such endomorphisms.)
The isomorphism types of such pairs $(V,A)$ are parameterized by multisegments.
Namely, for any segment $\Delta$ let $V_\Delta$ be the $\Cusp$-graded vector space $\C^\Delta$ with basis $\{\xbasis_\rho:\rho\in\Delta\}$ and let
$\rshft A_\Delta\in E_\rightarrow(V_\Delta)$ be given by $\rshft A_\Delta \xbasis_\rho=\xbasis_{\rshft\rho}$
where by convention $\xbasis_\rho=0$ if $\rho\notin\Delta$.
To any multisegment $\m=\Delta_1+\dots+\Delta_k$ define $V_\m=\oplus_{i=1}^kV_{\Delta_i}$ with basis $\{x^i_\rho:i=1,\dots,k, \rho\in\Delta_i\}$, and
$\rshft A_\m=\oplus_{i=1}^k\rshft A_{\Delta_i}\in E_\rightarrow(V_\m)$.
Then $\{(V_\m,\rshft A_\m):\m\in\Mult\}$ is a set of representatives for the isomorphism types of pairs $(V,A)$ as above.

The previous discussion applies verbatim equally well if we change $\lshft{}$ with $\rshft{}$ throughout.


For any finite-dimensional $\Cusp$-graded vector space $V$,
the group $\GL(V)$ of grading preserving linear automorphisms of $V$ acts with finitely many orbits on each of the spaces $E_\leftrightarrows(V)$.
Note that these spaces are in duality with respect to the $\GL(V)$-invariant pairing $A,B\mapsto\tr AB=\tr BA$.
Consider the algebraic set
\[
\commvar(V)=\{(A,B)\in E_\rightarrow(V)\times E_\leftarrow(V):AB=BA\}
\]
with the canonical $\GL(V)$-equivariant projection maps $p_{\leftrightarrows}:\commvar(V)\rightarrow E_\leftrightarrows(V)$.
The following is a special case of a result of Pyasetskii.
\begin{theorem} \cite{MR0390138} (cf. \cite{MR1371654})
\begin{enumerate}
\item $\dim\commvar(V)=\dim E_\rightarrow(V)=\dim E_\leftarrow(V)$ and the irreducible components of $\commvar(V)$ are equi-dimensional.
\item If $C$ is an irreducible component of $\commvar(V)$ then $p_\rightarrow(C)$ admits a (unique) open $\GL(V)$-orbit which we denote by $p_\rightarrow(C)^{\gen}$.
\item The map $C\mapsto p_\rightarrow(C)^{\gen}$ is a bijection
between the set of irreducible components of $\commvar(V)$ and the set of $\GL(V)$-orbits in $E_\rightarrow(V)$.
\item The inverse map is given by $\OO\mapsto\overline{p_{\rightarrow}^{-1}(\OO)}$ (Zariski closure).
\item Similar statements hold for $p_\leftarrow$.
\item For any $\GL(V)$-orbit $\OO$ of $E_\rightarrow(V)$, $p_{\leftarrow}(p_\rightarrow^{-1}(\OO))$ contains a unique open $\GL(V)$-orbit
$\OO^\#$. Thus, $p_\rightarrow^{-1}(\OO)\cap p_\leftarrow^{-1}(\OO^\#)$ is non-empty and open in both $p_\rightarrow^{-1}(\OO)$
and $p_\leftarrow^{-1}(\OO^\#)$.
\item The map $\OO\rightarrow\OO^\#$ is a bijection between the sets of $\GL(V)$-orbits in $E_\leftrightarrows(V)$.
\end{enumerate}
\end{theorem}

We denote by $\rshft{\OO}_\m$ the $\GL(V_\m)$-orbit of $\rshft{A}_\m$ in $E_\rightarrow(V)$ and similarly for $\lshft{\OO}_\m$.
Then $(\rshft\OO_\m)^\#=\lshft{\OO}_{\m^\#}$ where $\m^\#$ is as in \S\ref{sec: zeleinvo}. (We can identify $V_\m$ and $V_{\m^\#}$.)

The following is a variant (of a special case) of a beautiful conjecture of Geiss--Leclerc--Schr\"oer.
\begin{conjecture}(cf. \cite[Conjecture 18.1]{MR2822235}, \cite{LecChev}) \label{conj: GLS}
An irreducible representation $\pi=Z(\m)$ is \LM\ if and only if $p_\rightarrow^{-1}(\rshft{\OO}_\m)$ admits an open $\GL(V_\m)$-orbit.
(Clearly, such an orbit would necessarily be contained in $p_\leftarrow^{-1}(\lshft\OO_{\m^\#})$.)
\end{conjecture}
We emphasize however that a counterexample to Conjecture \ref{conj: GLS} would not necessarily invalidate the conjecture made in [ibid.].

The pertinent openness condition admits a homological interpretation.
Alternatively, we can rephrase it by saying that the stabilizer $G_\m$ of $\rshft{A}_\m$ in $\GL(V_\m)$ admits an open orbit in
the space $C_\m=\{B\in E_{\leftarrow}(V_\m):B\rshft{A}_\m=\rshft{A}_\m B\}$.
The advantage is that this is a linear action and by passing to the Lie algebra
the condition becomes the existence of $\lambda\in C_\m$ such that $[\Lieg_\m,\lambda]=C_\m$ where $\Lieg_\m=\operatorname{Lie} G_\m$,
viewed as a subalgebra of the Lie algebra of $\Cusp$-grading preserving endomorphisms of $V_\m$.
It is easy to explicate $\Lieg_\m$ and its action on $C_\m$ (cf. \cite[Lemmas II.4 and II.5]{MR863522}).
Let $\X_\m=\{(i,j):\Delta_i\prec\Delta_j\}$ and $\Y_\m=\{(i,j):\lshft{\Delta}_i\prec\Delta_j\}$.
Then $C_\m$ has a basis $\alpha_{i,j}$, $(i,j)\in \X_\m$ given by
\[
\alpha_{i,j}(\xbasis_\rho^l)=\delta_{j,l}\xbasis_{\lshft{\rho}}^i,\ \ \rho\in\Cusp, l=1,\dots,k,
\]
while as a vector space, $\Lieg_\m$ has a basis $\beta_{i,j}$, $(i,j)\in \Y_\m$ given by
\[
\beta_{i,j}(\xbasis_\rho^l)=\delta_{j,l}\xbasis_\rho^i,\ \ \rho\in\Cusp, l=1,\dots,k.
\]
Moreover, we have
\[
[\beta_{i,j},\alpha_{l,m}]=\delta_{j,l}\alpha_{i,m}-\delta_{i,m}\alpha_{l,j},\ \ (i,j)\in \Y_\m,\ (l,m)\in \X_\m
\]
where for convenience we set $\alpha_{i,j}=0$ if $(i,j)\notin \X_\m$.

In other words, any $\lambda\in C_\m$ is determined by its coordinates $\lambda_{i,j}$, $(i,j)\in \X_\m$ satisfying
\[
\lambda(\xbasis_\rho^j)=\sum_{i:(i,j)\in \X_\m,\lshft\rho\in\Delta_i}\lambda_{i,j}\xbasis_{\lshft\rho}^i.
\]
Similarly, any $g\in G_\m$ is determined by its coordinates $g_{i,j}$, $(i,j)\in \Y_\m$ satisfying
\[
g(\xbasis_\rho^j)=\sum_{i:(i,j)\in \Y_\m,\rho\in\Delta_i}g_{i,j}\xbasis_\rho^i.
\]

Thus, we have the following characterization of the condition that $p_\rightarrow^{-1}(\rshft{\OO}_\m)$ admits an open $\GL(V_\m)$-orbit.
(The surjectivity of the map $g\in\Lieg_\m\mapsto [g,\xi]\in C_\m$ is rephrased by the injectivity of the dual map.)

\begin{definition} \label{def: GLS}
Let $\m=\Delta_1+\dots+\Delta_k$ be a multisegment.
Consider the $\C$-vector space $\C^{\Y_\m}$ with basis $\{e_{i,j}:(i,j)\in \Y_\m\}$).
We say that $\m$ satisfies the condition \GLS\ if there exists $\lambda\in C_\m$ such that the vectors
\begin{equation} \label{eq: cxcy}
\xvec_{i,j}(\lambda):=\sum_{r:(r,j)\in \X_\m,(i,r)\in \Y_\m}\lambda_{r,j}e_{i,r}-\sum_{s:(s,j)\in \Y_\m,(i,s)\in \X_\m}\lambda_{i,s}e_{s,j},\ \ (i,j)\in \X_\m
\end{equation}
are linearly independet in $\C^{\Y_\m}$.
\end{definition}

This condition is easy to check (at least probabilistically) on a computer.

All in all, we get the following equivalent reformulation of Conjecture \ref{conj: GLS}.

\begin{conjecture}\label{conj: GLS2}
$Z(\m)$ is \LM\ if and only if $\m$ satisfies (GLS).
\end{conjecture}

\begin{remark} \label{rem: Gmorb}
Clearly, the linear independence of $\{\xvec_{i,j}(\lambda):(i,j)\in\X_\m\}$ is a Zariski open $G_\m$-invariant condition on $\lambda\in C_\m$.
\end{remark}

\begin{remark} \label{rem: leftrightarrow}
In the formulation of Conjecture \ref{conj: GLS} we could have used $p_\leftarrow^{-1}(\lshft{\OO}_\m)$ instead of $p_\rightarrow^{-1}(\rshft{\OO}_\m)$.
Indeed, an analogous argument would yield the restatement made in Conjecture \ref{conj: GLS2}.
\end{remark}

\begin{remark} \label{rem: neighij}
Note that $\X_\m=\X_{\m;\m}$ and $\Y_\m=\Y_{\m;\m}$ in the notation of Definition \ref{def: XandY}.
We continue to write $\rltn$ for the relation defined there. Thus, we get a bipartite graph $\Graph_\m$ whose vertices are $\X_\m\coprod\Y_\m$
(disjoint union) and whose edges are given by $\rltn$.
For any $(i,j)\in \X_\m$ and $\lambda\in C_\m$ we denote by $N_{i,j}(\lambda)\subset\Y_\m$ the set of non-zero coordinates of $\xvec_{i,j}(\lambda)$
and define $N_{i,j}:=\{y\in \Y_\m:(i,j)\rltn y\}$ (the neighbors of $(i,j)$ in $\Graph_\m$). Then $N_{i,j}(\lambda)\subset N_{i,j}$ with equality if $\lambda_{i,j}\ne0$ for all $(i,j)\in \X_\m$.
Note that $N_{i,j}\supset\{(i,i),(j,j)\}$ for all $(i,j)\in \X_\m$.
\end{remark}

\begin{remark} \label{rem: matching}
Clearly, \GLS\ implies the existence of a $\rltn$-matching from $\X_\m$ to $\Y_\m$.
(An example of a multisegment without such a matching is $[4,6]+[1,5]+[2,4]+[3,3]+[0,2]$.)
However, the latter condition is not sufficient. (See Remark \ref{rem: kirred} below.)
\end{remark}

It will be convenient to attach labels to the edges of the graph $\Graph_\m$.
Namely, we write $(i,j)\lrltn{(j',j)}(i,j')$ if $(i,j),(j',j)\in\X_\m$ and $(i,j')\in\Y_\m$; similarly $(i,j)\lrltn{(i,i')}(i',j)$ if $(i,j),(i,i')\in\X_\m$, $(i',j)\in\Y_\m$.

\begin{remark} \label{rem: strongmatch}
Suppose that $f:\X_\m\rightarrow\Y_\m$ is a $\rltn$-matching and let
\[
L_f=\{r\in\X_\m:x\lrltn rf(x)\text{ for some }x\in\X_\m\}.
\]
We say that $f$ is strong (resp., extra-strong) if there exists an enumeration
$r_1,\dots,r_n$ of $\X_\m$ such that $r_i\not\lrltn rf(r_j)$ (resp., $r_i\not\rltn f(r_j)$) for any $i<j$ and $r\in L_f$.
Clearly, if there exists a strong $\rltn$-matching $f$ then $\m$ satisfies \GLS.
Indeed, taking $\lambda_{i,j}\ne0$ if $(i,j)\in L_f$ and $0$ otherwise, the $f(\X_\m)$-coordinates of $\{\xvec_{i,j}(\lambda):(i,j)\in\X_\m\}$
form a lower triangular matrix with non-zero diagonal entries.

We do not know whether \GLS\ implies the existence of a strong $\rltn$-matching.
\end{remark}

\subsection{Some examples}

\begin{example}
Let $\m=[2,2]+[2,2]+[1,1]+[1,1]$.
Here $\X_\m=\{(3,1),(3,2),(4,1),((4,2)\}$ and $\Y_\m=\{(1,1),(1,2),(2,1),(2,2),(3,3),(3,4),(4,3),(4,4)\}$.
The $\xvec_{i,j}(\lambda)$'s are given by the following table:
\begin{table}[h!]
\begin{tabular}{ c || c | c | c | c | c | c | c | c}
  & $1,1$ & $1,2$ & $2,1$ & $2,2$ & $3,3$ & $3,4$ & $4,3$ & $4,4$ \\
    \hline\hline
    $3,1$ & $-\lambda_{3,1}$ & & $-\lambda_{3,2}$ & & $\lambda_{3,1}$ & $\lambda_{4,1}$ & \\
    $3,2$ & & $-\lambda_{3,1}$ & & $-\lambda_{3,2}$ & $\lambda_{3,2}$ & $\lambda_{4,2}$ & \\
    $4,1$ & $-\lambda_{4,1}$ & & $-\lambda_{4,2}$ & & & & $\lambda_{3,1}$ & $\lambda_{4,1}$ \\
    $4,2$ & & $-\lambda_{4,1}$ &  & $-\lambda_{4,2}$ & & & $\lambda_{3,2}$ & $\lambda_{4,2}$
\end{tabular}
\end{table}

Thus, $(3,1)\mapsto (1,1)$, $(3,2)\mapsto (1,2)$, $(4,1)\mapsto (4,3)$, $(4,2)\mapsto (4,4)$ is a strong $\rltn$-matching, hence $\m$ is \GLS.
However, there does not exist an extra-strong $\rltn$-matching.
\end{example}

Next, we make a simple general observation.
\begin{remark} \label{rem: genrem}
Suppose for simplicity that $b(\Delta_i)\le e(\rshft\Delta)$ for all $i,j$. Denote by $E_{i,j}$ the $k\times k$-matrix whose $(i,j)$-th entry is $1$
and all other entries vanish. The linear span $\yy$ of
\[
\{E_{i,j}:b(\Delta_i)\le b(\Delta_j)\text{ and }e(\Delta_i)\le e(\Delta_j)\}
\]
is a Lie subalgebra of the Lie algebra of $k\times k$ matrices and
$\{E_{i,j}:b(\Delta_j)=e(\rshft\Delta_i)\}$ spans a Lie ideal $\zz$ of $\yy$.
We can identify $\Lieg_\m$ with the quotient $\yy/\zz$ and
$C_\m$ with the linear span $\xx$ of $\{E_{i,j}:(i,j)\in \X_\m\}$ which is a nilpotent Lie ideal of $\yy$ whose center contains $\zz$.
The condition \GLS\ is that there exists $h\in\xx$ such that $[\yy,h]=\xx$. Equivalently, if $G$ (resp., $H$) is the subgroup of $\GL_k(\C)$
corresponding to $\yy$ (resp., $\xx$) the condition is that $G$ acts (by conjugation) with an open orbit on $H$.
\end{remark}

\begin{example}
Suppose that $\m=\Delta_1+\dots+\Delta_k$ is a ladder. Assume for simplicity that $\Delta_k\prec\Delta_1$.
Then $\m$ satisfies \GLS. Indeed, in view of Remark \ref{rem: genrem} this reflects the fact that the Borel subgroup of $\GL_k(\C)$
acts (by conjugation) on its unipotent radical with an open orbit (given by $u_{1,2}\dots u_{k-1,k}\ne0$).
Alternatively, $\X_\m=\{(i,j):1\le j<i\le k\}$, $\Y_\m\cup\{(k,1)\}=\{(i,j):1\le j\le i\le k\}$ and the map $f((i,j))=(i-1,j)$ is a strong $\rltn$-matching.
(We enumerate $\X_\m$ as $(k,k-1),\dots,(2,1),(k,k-2),\dots,(3,1),\dots,(k,2),(k-1,1),(k,1)$.)
Note that in this case $p_\rightarrow^{-1}(\rshft\OO_\m)\cap p_\leftarrow^{-1}(\lshft\OO_{\m^\#})$ itself is a $\GL(V_\m)$-orbit.
\end{example}

\begin{example}
For $k>2$ consider
\begin{equation}
\m=[k-1,2k-2]+[k,2k-3]+\sum_{i=1}^{k-2}[k-1-i,2k-3-i].
\end{equation}
Here is a drawing for $k=7$:
\[
\xymatrix@=0.6em{
&&&&&\circ\ar@{-}[r]&\circ\ar@{-}[r]&\circ\ar@{-}[r]&\circ\ar@{-}[r]&\circ\ar@{-}[r]&\circ\ar@{-}[r]&\circ\\
&&&&&&\circ\ar@{-}[r]&\circ\ar@{-}[r]&\circ\ar@{-}[r]&\circ\ar@{-}[r]&\circ\\
&&&&\circ\ar@{-}[r]&\circ\ar@{-}[r]&\circ\ar@{-}[r]&\circ\ar@{-}[r]&\circ\ar@{-}[r]&\circ\\
&&&\circ\ar@{-}[r]&\circ\ar@{-}[r]&\circ\ar@{-}[r]&\circ\ar@{-}[r]&\circ\ar@{-}[r]&\circ\\
&&\circ\ar@{-}[r]&\circ\ar@{-}[r]&\circ\ar@{-}[r]&\circ\ar@{-}[r]&\circ\ar@{-}[r]&\circ\\
&\circ\ar@{-}[r]&\circ\ar@{-}[r]&\circ\ar@{-}[r]&\circ\ar@{-}[r]&\circ\ar@{-}[r]&\circ\\
\circ\ar@{-}[r]&\circ\ar@{-}[r]&\circ\ar@{-}[r]&\circ\ar@{-}[r]&\circ\ar@{-}[r]&\circ}
\]
We claim that $\m$ satisfies \GLS.
Let $P=M\ltimes U$ be the standard parabolic subgroup of $\GL_k(\C)$ of type $(2,1,\dots,1)$
and let $P'$ be the subgroup $T\ltimes U$ of $P$ of codimension $2$ where $T$ is the diagonal torus.
Thus, $P'$ is the inverse image of $T$ under the projection $\pr_M:P\rightarrow M$.
In view of Remark \ref{rem: genrem} we need to check that
\begin{equation} \label{eq: P'openorb}
\text{$P'$ has an open orbit (by conjuagtion) on $U$.}
\end{equation}
Indeed, the element $y=I_k+\sum_{i=2}^{k-1}E_{i,i+1}$ is a Richardson element with respect to $P$ with centralizer
\[
P_y=\{g\in P:g_{1,i}=0\text{ for }1<i<k,\ g_{i,j}=g_{i+1,j+1}\text{ if }1<i\le j<k\},
\]
a group of dimension $\dim M=k+2$.
For instance, for $k=7$
\[
P_y=\{\begin{pmatrix}
a&&&&&&b\\
c&d&e&f&g&h&i\\
&&d&e&f&g&h\\
&&&d&e&f&g\\
&&&&d&e&f\\
&&&&&d&e\\
&&&&&&d
\end{pmatrix}:a,b,c,d,e,f,g,h,i\in\C,ad\ne0\}.
\]
Note that the dimension of $\pr_M(P_y)$ is $3$. Suppose that $x=pyp^{-1}$ for $p\in P$.
Then
\begin{multline*}
\pr_M(P'_x)=\pr_M(P_x\cap P')=\pr_M(pP_yp^{-1}\cap P')=\pr_M(pP_yp^{-1})\cap T\\=\pr_M(p)\pr_M(P_y)\pr_M(p)^{-1}\cap T
\end{multline*}
which is the group of scalar matrices provided that $p_{1,2}p_{2,2}\ne0$.
Hence, under this condition, $P'_x$ is of codimension $2$ in $P_x$ and the assertion \eqref{eq: P'openorb} follows.

Note that the element of $C_\m$ corresponding to $E_{1,3}+\sum_{i=3}^{k-1}E_{i,i+1}$ belongs to $\OO_\m^\#$ but
its orbit under $G_\m$ is not open in $C_\m$. Hence, in this case $p_\rightarrow^{-1}(\rshft\OO_\m)\cap p_\leftarrow^{-1}(\lshft\OO_{\m^\#})$ is not a $\GL(V_\m)$-orbit.

Alternatively, we can deduce that $\m$ satisfies \GLS\ by observing that
\[
\X_\m=\{(i,j):1\le j<i\le k\}\setminus\{(2,1)\},\ \ \Y_\m=\{(i,j):1\le j\le i\le k\}\setminus\{(k,2),(2,1)\}
\]
and the map $f((j,i))=(j-1,i)$ if $(j,i)\ne (3,1)$ and $f((3,1))=(1,1)$ is a strong $\rltn$-matching.
(We enumerate $\X_\m$ as $(k,k-1),\dots,(3,2),(k,k-2),\dots,(3,1),\dots,(k,2),(k-1,1),(k,1)$.)

\end{example}

\begin{example} \label{exam: 34*12}
For $k>4$ consider
\begin{equation} \label{eq: mincase34*12}
\m=[k-1,2k-2]+[k,2k-3]+\sum_{i=1}^{k-4}[k-1-i,2k-3-i]+[1,k]+[2,k-1].
\end{equation}
Here is a drawing for $k=7$:
\[
\xymatrix@=0.6em{
&&&&&\circ\ar@{-}[r]&\circ\ar@{-}[r]&\circ\ar@{-}[r]&\circ\ar@{-}[r]&\circ\ar@{-}[r]&\circ\ar@{-}[r]&\circ\\
&&&&&&\circ\ar@{-}[r]&\circ\ar@{-}[r]&\circ\ar@{-}[r]&\circ\ar@{-}[r]&\circ\\
&&&&\circ\ar@{-}[r]&\circ\ar@{-}[r]&\circ\ar@{-}[r]&\circ\ar@{-}[r]&\circ\ar@{-}[r]&\circ\\
&&&\circ\ar@{-}[r]&\circ\ar@{-}[r]&\circ\ar@{-}[r]&\circ\ar@{-}[r]&\circ\ar@{-}[r]&\circ\\
&&\circ\ar@{-}[r]&\circ\ar@{-}[r]&\circ\ar@{-}[r]&\circ\ar@{-}[r]&\circ\ar@{-}[r]&\circ\\
\circ\ar@{-}[r]&\circ\ar@{-}[r]&\circ\ar@{-}[r]&\circ\ar@{-}[r]&\circ\ar@{-}[r]&\circ\ar@{-}[r]&\circ\\
&\circ\ar@{-}[r]&\circ\ar@{-}[r]&\circ\ar@{-}[r]&\circ\ar@{-}[r]&\circ}
\]
We show that $\m$ does not satisfy \GLS.
Let $P=M\ltimes U$ be the standard parabolic subgroup of $\GL_k(\C)$ of type $(2,1,\dots,1,2)$
and let $P'$ be the subgroup $T\ltimes U$ of $P$ of codimension $4$.
As before, $P'$ is the inverse image of $T$ under the projection $\pr_M:P\rightarrow M$.
In view of Remark \ref{rem: genrem} we need to show that
\begin{equation} \label{eq: P'nopenorb}
\text{$P'$ does not have an open orbit (by conjuagtion) on $U$.}
\end{equation}
Suppose on the contrary that $P'$ admits an open orbit $\OO$.
Clearly, $\OO$ is contained in the Richardson orbit of $P$. Fix a representative $x\in\OO$.
Then the centralizer $P'_x$ of $x$ in $P'$ is of codimension $4$ in the centralizer $P_x$ of $x$ in $P$
and hence (since $P'\supset U$) $\pr_M(P'_x)$ is of codimension $4$ in $\pr_M(P_x)$.
However, the element $y=I_k+\sum_{i=2}^{k-2}E_{i,i+1}$ is a Richardson element with respect to $P$ and
\begin{multline*}
P_y=\{g\in P:g_{1,i}=g_{k+1-i,k}=0\text{ for }1<i<k-2,\\g_{2,1}=g_{3,k},\ g_{k,k-1}=g_{1,k-2},\ g_{1,1}=g_{k,k},\
g_{i,j}=g_{i+1,j+1}\text{ for }1<i\le j<k-1\},
\end{multline*}
a group of dimension $\dim M=k+4$.
For instance, for $k=7$
\[
P_y=\{\begin{pmatrix}
a&&&&b&c&d\\
e&p&q&r&s&t&f\\
&&p&q&r&s&e\\
&&&p&q&r&\\
&&&&p&q&\\
&&&&&p&\\
&&&&&b&a
\end{pmatrix}:a,b,c,d,e,f,p,q,r,s,t\in\C,ap\ne0\}.
\]
Note that the dimension of $\pr_M(P_y)$ is $4$. Suppose that $x=pyp^{-1}$ for $p\in P$.
Then as before
\[
\pr_M(P'_x)=\pr_M(p)\pr_M(P_y)\pr_M(p)^{-1}\cap T
\]
and the latter contains the group of scalar matrices. In particular, $\pr_M(P'_x)$ cannot be $0$-dimensional.
The assertion \eqref{eq: P'nopenorb} follows.
\end{example}

We can also give a simple necessary condition for a multisegment to satisfy \GLS.
\begin{remark} \label{rem: kirred}
We say that a pair $(i,j)\in \X_\m$ is irreducible if $N_{i,j}=\{(i,i),(j,j)\}$.
In this case $\xvec_{i,j}(\lambda)=\lambda_{i,j}(e_{i,i}-e_{j,j})$, and in particular $\xvec_{i,j}(\lambda)$ belongs to the $k-1$-dimensional space
\[
\{\sum_{i=1}^k\alpha_ie_{i,i}:\sum_{i=1}^k\alpha_i=0\}.
\]
Thus, if the number of irreducible pairs in $\X_\m$ is at least $k$ then $\m$ does not satisfy \GLS.

For instance for Leclerc's example $\m=[3,4]+[1,3]+[2,2]+[0,1]$ the set $\X_\m=\{(2,1),(3,1),(4,2),(4,3)\}$ consists entirely
of irreducible pairs. Hence $\m$ does not satisfy \GLS\ even though there is a $\rltn$-matching from $\X_\m$ to $\Y_\m$.
More examples of this kind will be given later. (See Remark \ref{rem: types not GLS}.)

On the other hand, in Example \ref{exam: 34*12} above, the only irreducible pairs are $(3,1)$, $(3,2)$, $(k-1,k-2)$, $(k,k-2)$ so the argument
does not apply in this case.
\end{remark}

\subsection{}
Finally, we provide a few sanity checks for Conjecture \ref{conj: GLS}.
\begin{remark} \label{rem: contractGLS}
Under the assumptions of Proposition \ref{prop: contract} the conditions \GLS\ for $\m$ and $\m'$ are equivalent (in fact identical).
\end{remark}

\begin{lemma} \label{lem: lderivm}
If $\m$ is \GLS\ then so are $\lderiv_{\rho}(\m)$ and $\rderiv_{\rho}(\m)$ for any $\rho\in\Cusp$.
\end{lemma}

\begin{proof}
We prove it for $\m'=\lderiv_{\rho}(\m)$. (The argument for $\rderiv_{\rho}(\m)$ is similar.)
Write $\m=\Delta_1+\dots+\Delta_k$.
Let $I$, $J$ and $f$ be as in Lemma \ref{lem: desclder}. Then
\[
\m'=\Delta_1'+\dots+\Delta'_k\text{ where }\Delta'_j=\begin{cases}^-\Delta_j\text{ (possibly empty)}&\text{if }j\in J,\\
\Delta_j&\text{otherwise.}\end{cases}
\]
In the course of the proof we will freely use the properties of $f$ described in Lemma \ref{lem: desclder} without further notice. Let
\[
A=\X_\m\cap (J\times f(I)),\ \ \tilde A=\Y_\m\cap (J\times I).
\]
Then $(j,i)\mapsto (j,f(i))$ is a bijection between $\tilde A$ and $A$ and in particular $\#A=\#\tilde A$.
For simplicity we order the $\Delta_i$'s by $\ge_e$ and write $I_{\ge i}=\{j\in I:j\ge i\}$.
We first show that there exists $\lambda\in C_\m$ for which $\{\xvec_{i,j}(\lambda):(i,j)\in\X_\m\}$ (as in \eqref{eq: cxcy})
are linearly independent and for all $i\in I$ we have
\begin{equation} \label{eq: speclambda}
\lambda_{j,f(i)}=\delta_{i,j}\text{ for any $j\in J\cup I_{\ge i}$ such that }(j,f(i))\in \X_\m.
\end{equation}
To that end, we prove by induction on $l\ge0$ that there exists $\lambda\in C_\m$ for which $\{\xvec_{i,j}(\lambda):(i,j)\in\X_\m\}$
are linearly independent and \eqref{eq: speclambda} is satisfied for the first $l$ elements of $I$.
The base of the induction is the assumption that $\m$ satisfies \GLS.
For the induction step, let $r\in I$ and suppose that $\tilde\lambda\in C_\m$ is such that $\{\xvec_{i,j}(\tilde\lambda):(i,j)\in\X_\m\}$ are linearly independent
and \eqref{eq: speclambda} holds for $\tilde\lambda$ for all $i\in I$ with $i<r$.
We may assume in addition that $\tilde\lambda_{r,f(r)}\ne0$. 
Let $g\in G_\m$ be the element given by
\[
g\xbasis_{\rho'}^i=\begin{cases}\sum_{j\in J\cup I_{\ge r}:(j,r)\in \Y_\m\text{ and }\rho'\in\Delta_j}\tilde\lambda_{j,f(r)}\xbasis_{\rho'}^j&\text{if }i=r\\
\xbasis_{\rho'}^i&\text{otherwise.}\end{cases}
\]
Thus,
\[
g_{j,i}=\begin{cases}\tilde\lambda_{j,f(r)}&\text{if }i=r\text{ and }j\in J\cup I_{\ge r},\\\delta_{i,j}&\text{otherwise,}\end{cases}\ \ (j,i)\in \Y_\m.
\]
Clearly, $g$ is invertible since $\tilde\lambda_{r,f(r)}\ne0$.
Let $\lambda=g^{-1}\tilde\lambda g$. By Remark \ref{rem: Gmorb} $\{\xvec_{i,j}(\lambda):(i,j)\in\X_\m\}$ are linearly independent.
We show that \eqref{eq: speclambda} is satisfied for all $i\in I$ with $i\le r$.
Fix $i\in I$. Clearly, $g\xbasis_{\rshft{\rho}}^{f(i)}=\xbasis_{\rshft{\rho}}^{f(i)}$.
If $i<r$ then by induction hypothesis, the coordinate of $\tilde\lambda \xbasis_{\rshft{\rho}}^{f(i)}$ at $x_\rho^i$ is $1$
and the coordinates at $x_\rho^j$, $j\in J\cup I_{>i}$, and in particular at $j=r$, vanish. Thus,
$\lambda \xbasis_{\rshft{\rho}}^{f(i)}=\tilde\lambda\xbasis_{\rshft{\rho}}^{f(i)}$ and \eqref{eq: speclambda} is satisfied.
On the other hand, we can write $\tilde\lambda\xbasis_{\rshft\rho}^{f(r)}=\xi_1+\xi_2$ where
\[
\xi_1=\sum_{j\in J\cup I_{\ge r}:(j,f(r))\in \X_\m}\tilde\lambda_{j,f(r)}\xbasis_\rho^j\ \ \ \text{   and   }\ \ \
\xi_2=\sum_{j\notin J\cup I_{\ge r}:(j,f(r))\in \X_\m}\tilde\lambda_{j,f(r)}\xbasis_\rho^j=g\xi_2.
\]
Note that if $j\in J\cup I_{\ge r}$ then $(j,r)\in \Y_\m\iff(j,f(r))\in \X_\m$. Thus, $\xi_1=g\xbasis_\rho^i$ and
hence $\lambda\xbasis_{\rshft\rho}^{f(r)}=\xbasis_\rho^r+\xi_2$. It follows that $\lambda$ satisfies \eqref{eq: speclambda} for $i=r$ as well,
completing the induction step.

Now let
\begin{gather*}
A'=\X_{\m'}\cap (I\times J),\ \ \widetilde{A'}=\Y_{\m'}\cap(f(I)\times J),\\
B=\X_\m\cap(\{i:e(\Delta_i)=\lshft\rho\}\times J), \ \ \tilde B=\Y_\m\cap(\{i:e(\Delta_i)=\rho\}\times J).
\end{gather*}
As before, $(i,j)\mapsto (f(i),j)$ is a bijection between $A'$ and $\widetilde{A'}$.
In particular, $\#A'=\#\widetilde{A'}$. It is also easy to see that
\[
\X_{\m'}\setminus\X_\m=A',\ \X_\m\setminus\X_{\m'}=A\cup B,\ \Y_{\m'}\setminus\Y_\m=\widetilde{A'},\ \ \Y_\m\setminus\Y_{\m'}=\tilde A\cup\tilde B.
\]
Suppose that $\lambda$ satisfies \eqref{eq: speclambda} for all $i\in I$. We claim that
\begin{enumerate}
\item \label{part: BB'} $N_{i,j}(\lambda)\subset\tilde A$ for any $(i,j)\in A$.
\item $N_{i,j}(\lambda)\cap\tilde B=\emptyset$ for any $(i,j)\in \X_\m\setminus B$.
\item If $(i',j')\lrltn{(i,j)}(i'',j'')$ with $(i',j')\in \X_\m\setminus B$ and $(i,j)\in B$ then $i'=i$
and $(i'',j'')=(j,j')\in\tilde A$.
\end{enumerate}

The first part follows from \eqref{eq: speclambda}.

For the second part, assume on the contrary that $(i,j)\in \X_\m\setminus B$ and
$(i',j')\in N_{i,j}(\lambda)\cap\tilde B$. In particular, $(i,j)\rltn(i',j')$.
By \eqref{eq: speclambda} we cannot have $i'=i$. Hence $j'=j$ which is also impossible since $(i,j)\notin B$.

The third part is also easy.

Now let $\lambda'$  be the element of $C_{\m'}$ whose coordinates are given by
\[
\lambda'_{i,j}=\begin{cases}\lambda_{i,j}&(i,j)\in \X_\m\cap\X_{\m'},\\0&(i,j)\in A',\end{cases}\ \ \ (i,j)\in \X_{\m'}
\]
and let
\[
\xvec'_{i,j}(\lambda')=\sum_{r:(r,j)\in \X_{\m'},(i,r)\in \Y_{\m'}}\lambda'_{r,j}e'_{i,r}-\sum_{s:(s,j)\in \Y_{\m'},(i,s)\in \X_{\m'}}\lambda_{i,s}e'_{s,j},\ \ (i,j)\in \X_{\m'}
\]
where $\{e'_{i,j}:(i,j)\in\Y_{\m'}\}$ is the standard basis for $\C^{\Y_{\m'}}$.
We show that $\{\xvec'_{i,j}(\lambda'):(i,j)\in X_{\m'}\}$ are linearly independent in $\C^{\Y_{\m'}}$.
For any $(i,j)\in \X_\m\cap\X_{\m'}=X_\m\setminus(A\cup B)$ the coordinates of $\xvec_{i,j}(\lambda)$ at $\tilde B$ vanish while the coordinates
at $\Y_\m\cap\Y_{\m'}=\Y_\m\setminus (\tilde A\cup\tilde B)$ coincide with those of $\xvec'_{i,j}(\lambda')$.
On the other hand, the non-zero coordinates of $\xvec'_{i,j}(\lambda')$, $(i,j)\in A'$ are confined to $\widetilde{A'}$.
Moreover, by the assumption on $\lambda$, the square submatrix pertaining to the rows in $A'$ and the columns in $\widetilde{A'}$ is
lower unitriangular for a suitable enumeration of the rows and columns:
the entry of $\xvec'_{i,j}(\lambda')$, $(i,j)\in A'$ in the $(f(i'),j')$-column is $\lambda_{i,f(i')}$ if $j'=j$, $i\le i'$
and $(i,f(i'))\in X_\m$ and $0$ otherwise.
Thus, the linear independence of $\{\xvec'_{i,j}(\lambda'):(i,j)\in X_{\m'}\}$ follows from the fact that the non-zero coordinates of $\xvec_{i,j}(\lambda)$, $(i,j)\in A$
are confined to $\tilde A$ and $\#A=\#\tilde A$. The lemma follows.
\end{proof}

\begin{table}[h!]
\begin{tabular}{ c || c | c | c }
& $\Y_\m\cap\Y_{\m'}$ & $\tilde A$ & $\tilde B$ \\
\hline\hline
$\X_\m\cap\X_{\m'}$ & $M_1$ & $*$ & $0$ \\ \hline
$A$ & $0$ & $M_2$ & $0$ \\ \hline
$B$ & $*$ & $*$ & $*$ \\
\end{tabular}
\quad
\begin{tabular}{ c || c | c }
& $\Y_\m\cap\Y_{\m'}$ & $\tilde A'$ \\
\hline\hline
$\X_\m\cap\X_{\m'}$ & $M_1$ & $*$ \\ \hline
$A'$ & $0$ & $M_3$ \\
\end{tabular}
\caption{Comparing $\xvec_{i,j}(\lambda)$ (left) and $\xvec'_{i,j}(\lambda')$ (right) for $\lambda$, $\lambda'$ as above. The full-rank matrix $M_1$ is common to both
while $M_2$ and $M_3$ are unitriangular up to permutation}
\end{table}

\begin{remark} \label{rem: GLS^t}
The condition \GLS\ is clearly invariant under $\m\mapsto\m^\vee$. It is also invariant under $\m\mapsto\m^\#$.
(This follows from Remark \ref{rem: leftrightarrow} and the fact that $p_\rightarrow^{-1}(\rshft{\OO}_\m)\cap p_\leftarrow^{-1}(\lshft{\OO}_{\m^\#})$
is non-empty and open in $p_\rightarrow^{-1}(\rshft{\OO}_\m)$.)
\end{remark}

\begin{remark} \label{rem: GLScomb}
Suppose that $\Delta$ is a detachable segment of a multisegment $\m$. (See Definition \ref{def: detachable}.)
It is clear that if $\m$ satisfies \GLS\ then so does $\m-\Delta$. Indeed, suppose that $\m=\Delta_1+\dots+\Delta_k$ with $\Delta=\Delta_k$
and let $\m'=\m-\Delta$. Thus, $\X_{\m'}=\{(i,j)\in \X_\m:i,j\ne k\}$ and similarly for $\Y_{\m'}$.
Since $\Delta$ is detachable, $N_{i,j}\subset \Y_{\m'}$ for all $(i,j)\in \X_{\m'}$.
The claim follows.
\end{remark}

\part{}
In the second part of the paper we state and prove our main result, which is to characterize, for regular multisegments $\m$,
the condition that $\zele{\m}\in\IrrS$ and in particular to show that Conjecture \ref{conj: GLS} holds in this case.
This will involve both geometry and combinatorics. In the next couple of sections we recall the interplay between the two in the context
of Schubert varieties and interpret it for the case of regular multisegments.

\section{Smooth pairs} \label{sec: smth pairs}

In this section we recall some well-known facts about singularities of Schubert varieties of type A.

\subsection{}

Let $B_k$ be the Borel subgroup of upper triangular matrices in $\GL_k$ over $\C$ and consider the $B_k$ action on the flag
variety $B_k\bs\GL_k$. For any $\sigma\in S_k$
let $\cell_\sigma$ be the corresponding Schubert cell, i.e., the $B_k$-orbit of the permutation matrix corresponding to $\sigma$,
and let $X_\sigma$ be the corresponding Schubert variety (the Zariski closure of $\cell_\sigma$).
Recall that $\cell_\sigma$ is open in $X_\sigma$ and has dimension $\ell(\sigma)=\#\{i<j:\sigma(i)>\sigma(j)\}$.
Also, $\cell_{\sigma_0}\subset X_\sigma$ if and only if $\sigma_0\le\sigma$ in the Bruhat order. We write $[\sigma_0,\sigma]$ for the Bruhat interval
\[
[\sigma_0,\sigma]=\{\sigma_1\in S_k:\sigma_0\le\sigma_1\le\sigma\}.
\]
\begin{definition}
Let $\sigma,\sigma_0\in S_k$ with $\sigma\ge\sigma_0$.
We say that $(\sigma,\sigma_0)$ is a smooth pair if
$\cell_{\sigma_0}$ is contained in the smooth locus $X_\sigma^{\smth}$ of $X_\sigma$.
\end{definition}

In particular, if $\sigma_0=e$ then $(\sigma,\sigma_0)$ is a smooth pair if and only if $X_\sigma$ is smooth.
In this case we simply say that $\sigma$ is smooth.

Since $X_\sigma^{\smth}$ is open in $X_\sigma$,
\begin{equation} \label{eq: redrelsmth}
\text{if $(\sigma,\sigma_0)$ is a smooth pair and $\sigma_1\in [\sigma_0,\sigma]$ then $(\sigma,\sigma_1)$ is also a smooth pair.}
\end{equation}

It is known that $(\sigma,\sigma_0)$ is a smooth pair if and only if the Kazhdan--Lusztig polynomial
$P_{\sigma_0,\sigma}$ with respect to $S_k$ is $1$. (See \cite{MR1782635}.)
However, there is a much simpler well-known combinatorial criterion for smoothness which we recall next.

Denote by $\rflx$ the set of reflexions in the symmetric group $S_k$.
The elements of $\rflx$ are the transpositions $t_{i,j}$, $1\le i<j\le k$.
Recall that for any $\sigma\in S_k$ and $i<j$ we have $\sigma t_{i,j}>\sigma$ if and only if
$\sigma(i)<\sigma(j)$, otherwise $\sigma t_{i,j}<\sigma$. Thus,
\[
\#\{t\in\rflx:\sigma t<\sigma\}=\ell(\sigma).
\]

For $\sigma,\sigma_0\in S_k$ with $\sigma\ge\sigma_0$ we define
\[
\rsig(\sigma_0,\sigma)=\{t\in\rflx:\sigma_0t\in [\sigma_0,\sigma]\}
\]
and
\[
\asig(\sigma_0,\sigma)=\{t\in\rflx:\sigma_0t\le\sigma\}=\rsig(\sigma_0,\sigma)\cup\{t_{i,j}:i<j\text{ and }\sigma_0(i)>\sigma_0(j)\}.
\]

The following result due to Lakshmibai--Seshadri is well known and admits many generalizations.
(Cf.~\cite{MR1653040} or \cite{MR1782635} for more details.)

\begin{proposition}[\cite{MR752799}] \label{prop: LS}
The dimension of the tangent space of $X_\sigma$ at any point of $\cell_{\sigma_0}$ is equal to $\#\asig(\sigma_0,\sigma)$.
Thus, $\#\asig(\sigma_0,\sigma)\ge\ell(\sigma)$ (or equivalently, $\#\rsig(\sigma_0,\sigma)\ge\ell(\sigma)-\ell(\sigma_0)$
and equality holds if and only if $(\sigma,\sigma_0)$ is a smooth pair.
\end{proposition}

This proposition provides an efficient algorithm for deciding whether a given pair is smooth
since the sets $\rsig(\sigma_0,\sigma)$ and $\asig(\sigma_0,\sigma)$ are easily computable.

We also remark that the inequality $\#\asig(\sigma_0,\sigma)\ge\ell(\sigma)$ and the fact that equality implies that
$\#\asig(\sigma_1,\sigma)=\ell(\sigma)$ for any $\sigma_1\in[\sigma_0,\sigma]$ can also be proved combinatorially
using \cite[Lemma 2.2]{MR1827861} and induction on $\ell(\sigma)-\ell(\sigma_0)$.

\subsection{} \label{sec: minsingBL}
In \cite{MR1051089} a combinatorial criterion for the smoothness of Schubert variety of type $A_n$ was given in terms of pattern avoidance.
Namely, $\sigma$ is smooth if and only if $\sigma$ avoids the pattern $4231$ and $3412$.
Moreover, in the non-smooth case, a conjectural description of the irreducible components of the smooth locus was given as well.
This conjecture was solved independently in \cite{MR1990570, MR1994224, MR2015302, MR1853139}
(with an important earlier contribution in \cite{MR1827861}).
In order to state the main result of these papers we first recall certain permutations introduced in \cite[\S9]{MR1990570}.
For $r,s\ge2$ and $t=1,2,3$, with $s=2$ if $t=3$, let $\tau_{r,s}^{(t)},\delta_{r,s}^{(t)}\in S_k$ with $k=r+s$ be the pairs of permutations given by
\begin{subequations}
\begin{equation} \label{eq: taurs1}
\tau_{r,s}^{(1)}(i)=\begin{cases}k&i=1,\\
r+2-i&1<i\le r,\\
r+k-i&r<i<k,\\
1&i=k.\end{cases}\ \
\delta_{r,s}^{(1)}(i)=\begin{cases}r+1-i&i\le r,\\r+k+1-i&i>r.\end{cases}
\end{equation}
\begin{equation} \label{eq: taurs2}
\tau_{r,s}^{(2)}(i)=\begin{cases}r+1&i=1,\\
r+1-i&1<i<r,\\
k&i=r,\\
1&i=r+1,\\
r+k+1-i&r+1<i<k,\\
r&i=k,
\end{cases}\ \
\delta_{r,s}^{(2)}(i)=\begin{cases}r-i&i<r,\\
r+1&i=r,\\
r&i=r+1,\\
r+k+2-i&i>r+1.\end{cases}
\end{equation}
\begin{equation} \label{eq: taurs3}
\tau_{r,2}^{(3)}(i)=\begin{cases}
r+1&i=1,\\
k&i=2,\\
k+1-i&2<i\le r,\\
1&i=r+1,\\
2&i=k.
\end{cases}\ \
\delta_{r,2}^{(3)}(i)=\begin{cases}
1&i=1,\\
k+1-i&1<i<k,\\
k&i=k.\end{cases}
\end{equation}
\end{subequations}
In the notation of \cite[\S9]{MR1990570} we have $\tau_{r,s}^{(1)}=w_{r,s}$, $\tau_{r,s}^{(2)}=w_{r-1,2,s-1}$,
$\tau_{r,2}^{(3)}=w_{1,r,1}$ and similarly $\delta_{r,s}^{(1)}=x_{r,s}$, $\delta_{r,s}^{(2)}=x_{r-1,2,s-1}$,
$\delta_{r,2}^{(3)}=x_{1,r,1}$.\footnote{Note the following typo in \cite[(9.2)]{MR1990570}: the last entry of $w_{k,m}$, which is $1$, is missing.}
In particular, $\tau_{2,2}^{(2)}=\tau_{2,2}^{(3)}$ and $\delta_{2,2}^{(2)}=\delta_{2,2}^{(3)}$.

It follows from \cite{MR1827861} (cf.~\cite[Theorem 37]{MR1990570}) that
\begin{equation} \label{eq: pairs not smooth}
\text{the pairs $(\tau_{r,s}^{(1)},\delta_{r,s}^{(1)})$, $(\tau_{r,s}^{(2)},\delta_{r,s}^{(2)})$, $(\tau_{r,2}^{(3)},\delta_{r,2}^{(3)})$ are not smooth.}
\end{equation}

For a subset $I\subset\{1,\dots,k\}$ of size $l$ we write $\rmv_I$ for the ``flattened'' permutation in $S_{k-l}$ obtained from $\sigma$ by removing the entries
$(i,\sigma(i))$, $i\in I$ and keeping the relative order of all other entries. In other words, $\rmv_I(\sigma)=\jmath\circ\sigma\circ\imath$ where
$\imath:\{1,\dots,k-l\}\rightarrow\{1,\dots,k\}\setminus I$
and $\jmath:\{1,\dots,k\}\setminus\sigma(I)\rightarrow\{1,\dots,k-l\}$ are the monotone bijections.

\begin{theorem}\cite{MR1990570, MR1994224, MR2015302, MR1853139} \label{thm: BW}
Suppose that $\sigma_0\le\sigma$ but $(\sigma,\sigma_0)$ is not a smooth pair. Then
there exist $\sigma_1\in[\sigma_0,\sigma]$, a subset $I\subset\{1,\dots,k\}$ and integers $r,s\ge2$ and $t=1,2,3$, with $s=2$ if $t=3$, such that
\begin{enumerate}
\item $\sigma_1(i)=\sigma(i)$ for all $i\in I$.
\item $\rmv_I(\sigma)=\tau_{r,s}^{(t)}$ and $\rmv_I(\sigma_1)=\delta_{r,s}^{(t)}$.
\item The Bruhat intervals $[\sigma_1,\sigma]$ and $[\delta_{r,s}^{(t)},\tau_{r,s}^{(t)}]$ are isomorphic as posets.
(Equivalently, $\ell(\sigma)-\ell(\sigma_1)=\ell(\tau_{r,s}^{(t)})-\ell(\delta_{r,s}^{(t)})$.)
\end{enumerate}
(In the terminology of \cite{MR2422304} this means that $[\delta_{r,s}^{(t)},\tau_{r,s}^{(t)}]$ interval pattern embeds into $[\sigma_1,\sigma]$.)
\end{theorem}

Note that the case $\sigma_0=e$ is essentially a reformulation of the original result of \cite{MR1051089}.

\subsection{} \label{sec: rmvi}
Given $\sigma\in S_k$ and an index $i$, we write for simplicity $\rmv_i(\sigma)=\rmv_{\{i\}}(\sigma)$.
The following result is probably well known. For convenience we include the proof.
\begin{lemma} \label{lem: rmvi}
Suppose that $(\sigma,\sigma_0)$ is a smooth pair in $S_k$ and $i$ is an index such that $\sigma(i)=\sigma_0(i)$.
Then $(\rmv_i(\sigma),\rmv_i(\sigma_0))$ is a smooth pair in $S_{k-1}$.
\end{lemma}

Before giving the proof, we first recall the following elementary fact.
\begin{lemma} (\cite[Lemma 17]{MR1990570}) \label{lem: elemi}
Suppose that $\sigma_0,\sigma\in S_k$ and $i$ is an index such that $\sigma(i)=\sigma_0(i)$. Then $\sigma_0\le\sigma$ if and only if
$\rmv_i(\sigma_0)\le\rmv_i(\sigma)$.
\end{lemma}

Next, we introduce some notation.
For any index $i$ let
\[
\rflx_i=\{t\in\rflx:t(i)\ne i\}
\]
(so that $\#\rflx_i=k-1$).
For any $\sigma\in S_k$ let
\[
\ell_i(\sigma)=\#\{t\in\rflx_i:\sigma t<\sigma\}=\#\{r<i:\sigma(r)>\sigma(i)\}+\#\{r>i:\sigma(r)<\sigma(i)\}
\]
so that $2\ell(\sigma)=\sum_i\ell_i(\sigma)$.
For any $\sigma_0\le\sigma$ let $\rsig_i(\sigma_0,\sigma)=\rsig(\sigma_0,\sigma)\cap\rflx_i$
and $\asig_i(\sigma_0,\sigma)=\asig(\sigma_0,\sigma)\cap\rflx_i$.

We need another (probably well-known) result.
\begin{lemma} \label{lem: lcldeod}
Suppose that $\sigma_0\le\sigma$ and let $i$ be an index such that $\sigma(i)=\sigma_0(i)$. Then
\[
\#\asig_i(\sigma_0,\sigma)\ge\ell_i(\sigma).
\]
Equivalently, $\#\rsig_i(\sigma_0,\sigma)\ge\ell_i(\sigma)-\ell_i(\sigma_0)$.
\end{lemma}

\begin{proof}
We prove the result by induction on $\ell(\sigma)-\ell(\sigma_0)$.
If $\sigma=\sigma_0$ the assertion is trivial.
Otherwise, it follows from \cite[Proposition 14]{MR1990570} and the fact that $\sigma_0(i)=\sigma(i)$
that $\rsig(\sigma_0,\sigma)\ne\rsig_i(\sigma_0,\sigma)$.
For the induction step, take any $t\in\rsig(\sigma_0,\sigma)\setminus\rflx_i$ and let
\[
\phi_t=\phi_t^{\sigma_0t,\sigma}:\asig(\sigma_0t,\sigma)\rightarrow\asig(\sigma_0,\sigma)
\]
be the injective map defined in \cite[Lemma 2.2]{MR1827861} (cf.~\cite[\S6]{MR1990570}).
It is easy to see from the definition that $\phi_t(\asig_i(\sigma_0t,\sigma))\subset\asig_i(\sigma_0,\sigma)$.
Thus, $\#\asig_i(\sigma_0,\sigma)\ge\#\asig_i(\sigma_0t,\sigma)$ and the assertion follows from the induction hypothesis.
\end{proof}

Finally, we can prove Lemma \ref{lem: rmvi}.

\begin{proof}[Proof of Lemma \ref{lem: rmvi}]
By Proposition \ref{prop: LS} and the assumption we have
\[
\#\{t\in\rflx:\sigma_0t\le\sigma\}=\ell(\sigma).
\]
For simplicity, write $\tilde\sigma_0=\rmv_i(\sigma_0)$ and $\tilde\sigma=\rmv_i(\sigma)$.
By Lemma \ref{lem: elemi} we have $\tilde\sigma_0\le\tilde\sigma$.
If $t\in\rflx\setminus\rflx_i$ then $\sigma_0t(i)=\sigma_0(i)=\sigma(i)$ and $\rmv_i(\sigma_0t)=
\rmv_i(\sigma_0)t'$ where $t'=\rmv_i(t)$ is a reflexion in $S_{k-1}$.
Thus, again by Lemma \ref{lem: elemi},
$\sigma_0 t\le\sigma$ if and only if $\tilde\sigma_0t'\le\tilde\sigma$.
Clearly $\rmv_i$ induces a bijection between $\rflx\setminus\rflx_i$ and the set $\rflx'$ of reflexions in $S_{k-1}$.
It follows from Lemma \ref{lem: lcldeod} that
\begin{multline*}
\#\{t'\in\rflx':\tilde\sigma_0t'\le\tilde\sigma\}=
\#\{t\in\rflx:\sigma_0t\le\sigma\}-\#\{t\in\rflx_i:\sigma_0t\le\sigma\}\\
=\ell(\sigma)-\#\{t\in\rflx_i:\sigma_0t\le\sigma\}\le\ell(\sigma)-\ell_i(\sigma)=\ell(\tilde\sigma).
\end{multline*}
By Proposition \ref{prop: LS} once again it follows that $(\tilde\sigma,\tilde\sigma_0)$ is a smooth pair as required.
\end{proof}

\section{Balanced multisegments} \label{sec: combi}
In this section we introduce the main combinatorial condition on multisegments for which our main result applies
and relate it to the results of the previous section.

\subsection{} \label{sec: biseq}
Henceforth we fix an integer $k\ge1$.
Let $\tseq$ be a pair of two sequences of integers $a_1\le\dots\le a_k$ and $b_1\ge\dots\ge b_k$ such that
$a_{k+1-i}\le b_i+1$ for all $i$.
We write $\tseq$ as $\bitmplt$ and refer to it simply as a \emph{\biseq}.
By \cite[Lemma 15]{MR3163355}, there exists a unique $\sigma_0=\sigma_0(\tseq)\in S_k$ such that for any $\sigma\in S_k$
we have
\begin{equation} \label{def: propsigma0}
a_{\sigma^{-1}(i)}\le b_i+1\text{ for all $i$ if and only if }\sigma\ge\sigma_0
\end{equation}
(in the Bruhat order).\footnote{This was stated in \cite{MR3163355} under the assumption that $a_1<\dots<a_k$
and $b_1>\dots>b_k$, but the proof, which is in any case elementary, works under the weaker assumption.}
Moreover, $\sigma_0$ is $213$-avoiding, i.e. there do not exist indices $a<b<c$ such that
$\sigma_0(b)<\sigma_0(a)<\sigma_0(c)$. In other words, $i\mapsto k+1-\sigma_0(i)$ is s stack-sortable permutation in the sense of Knuth \cite[\S2.2.1]{MR3077152}.
The permutation $\sigma_0$ is defined recursively as follows. Given $\sigma_0^{-1}(k),\dots,\sigma_0^{-1}(i+1)$ we set
\begin{equation} \label{def: sigma0}
\sigma_0^{-1}(i)=\max\{j\notin\sigma_0^{-1}(\{i+1,\dots,k\}):a_j\le b_i+1\}
\end{equation}
which is well-defined since $a_{k+1-i}\le b_i+1$.
It follows that  $\sigma_0(i)<\sigma_0(i+1)$ whenever $a_i=a_{i+1}$ and $\sigma_0^{-1}(i)<\sigma_0^{-1}(i+1)$ whenever $b_i=b_{i+1}$.

\begin{example}
For $l\ge0$ let
\begin{equation} \label{def: akl}
\tseq_{k,l}=\begin{pmatrix}1&2&\dots&k\\k+l-1&k+l-2&\dots&l\end{pmatrix}.
\end{equation}
Then
\[
\sigma_0(\tseq_{k,l})(i)=\begin{cases}i&i\le l,\\k+l+1-i&\text{otherwise.}\end{cases}
\]
The condition $\sigma\ge\sigma_0(\tseq)$ becomes $\sigma(i)\le k+l+1-i$ for all $i>l+1$.
\end{example}

For any $\sigma\in S_k$ let
\[
\m_\sigma=\m_\sigma(\tseq)=\sum_{i=1}^k[a_{\sigma^{-1}(i)},b_i]
\]
and $\pi_\sigma=\zele{\m_\sigma}$. Thus, $\m_\sigma\in\Mult$ (i.e., $\pi_\sigma\ne0$) if and only if $\sigma\ge\sigma_0$.
Note that $\m_{\sigma_0}$ is pairwise unlinked.
If we want to stress the dependence on $\rho$ we will write $\m_\sigma^{(\rho)}(\tseq)$.

\begin{remark}
Clearly, as we vary $\tseq$ and $\sigma\in S_k$ (with $\sigma\ge\sigma_0(\tseq)$), $\m_\sigma(\tseq)$ range over all
multisegments with $\le k$ segments.
More precisely, given $\m=\Delta_1+\dots+\Delta_k$, we may assume that $e(\Delta_1)\ge\dots\ge e(\Delta_k)$.
We take $\sigma\in S_k$ such that $b(\Delta_{\sigma(i)})\le b(\Delta_{\sigma(j)})$ whenever $i>j$ and set
$b_i=e(\Delta_i)$ and $a_i=b(\Delta_{\sigma(i)})$. Finally, if we sort the sequence $b(\Delta_1),\dots,b(\Delta_k),e(\Delta_1)+\frac32,\dots,e(\Delta_k)+\frac32$
as $c_1\le\dots\le c_{2k}$ and replace $c_i$ by the letter $X$ if $c_i\in\{b(\Delta_1),\dots,b(\Delta_k)\}$ and by the letter $Y$ otherwise
then we get a Dyck word $\word$ of length $2k$ and $\sigma_0(\tseq)$ is the $213$-avoiding permutation corresponding to $\word$ (see Lemma \ref{lem: comb12}
below).
\end{remark}

Note that
\begin{equation} \label{eq: contrasigma}
\m_\sigma^{(\rho)}(\tseq_{k,l})^\vee=\m_{\sigma^{-1}}^{(\rho^*)}(\tseq_{k,l})\text{ where }\rho^*=(\rho\nu_\rho^{k+l})^\vee.
\end{equation}

We say that a \biseq\ $\tseq=\bitmplt$ is regular if $a_1<\dots<a_k$ and $b_1>\dots>b_k$.
In this case the multisegments $\m_\sigma(\tseq)$, $\sigma \geq \sigma_0$, are regular and distinct.
Moreover, $\m_{\sigma_1}(\tseq)\obt\m_{\sigma_2}(\tseq)$ if and only if $\sigma_1\le\sigma_2$.
(In the non-regular case it is still true that $\m_{\sigma_1}(\tseq)\obt\m_{\sigma_2}(\tseq)$ if $\sigma_1\le\sigma_2$.)

For completeness we recall a standard combinatorial result (cf. \cite[Exercise 6.19]{MR1676282}).
A Dyck word of length $2k$ is a string composed of the letters $X$ and $Y$, each appearing $k$ times, such that in any initial segment,
the number of $Y$'s does not exceed the number of $X$'s. It is well-known that the number of Dyck words of length $2k$ is
the Catalan number $C_k={2k\choose k}-{2k\choose k+1}$.
\begin{lemma} \label{lem: comb12}
The following sets are in natural bijections:
\begin{enumerate}
\item Dyck words of length $2k$.
\item Regular \biseq s $\tseq=\bitmplt$ such that $a_1=2$ and $b_1=2k-1$.
\item $213$-avoiding permutations in $S_k$.
\end{enumerate}
\end{lemma}

\begin{proof}
Given a Dyck word $\word$ of length $2k$ let $a_i'$ (resp., $b_i'$) $i=1,\dots,k$ be the position of the $i$-th $X$ (resp., $Y$)
from the left (resp., from the right). Then $a'_i<b'_{k+1-i}$.
Letting $a_i=a_i'+1$ and $b_i=b'_i-1$, $i=1,\dots,k$, we get a regular  \biseq\ $\tseq(\word)=\bitmplt$ with $a_1=2$, $b_1=2k-1$.

To any \biseq\ $\tseq=\bitmplt$ 
we assign the $213$-avoiding permutation $\sigma_0(\tseq)$ defined by \eqref{def: propsigma0} and \eqref{def: sigma0}.

Finally, given any permutation $\sigma_0\in S_k$ we assign the Dyck word $\word(\sigma_0)$ such that for any $i$, the number of $X$'s to the left of the
$i$-th $Y$ from the right is $x_i=\max_{j\ge i}\sigma^{-1}(j)$. (Clearly, $x_1\ge\dots\ge x_k$ and $x_i\ge k+1-i$ for all $i$.)

It is easy to see that $\tseq(\word(\sigma_0(\tseq')))=\tseq'$, $\word(\sigma_0(\tseq(\word')))=\word'$ and
$\sigma_0(\tseq(\word(\sigma)))=\sigma$ for any regular \biseq\ $\tseq'=\bitmplt$ with $a_1=2$ and $b_1=2k-1$,
a Dyck word $\word'$ of length $2k$ and a $213$-avoiding permutation $\sigma\in S_k$.
\end{proof}

\begin{remark}
For any permutation $\sigma\in S_k$, the permutation $\sigma_0(\tseq(\word(\sigma)))$ is the unique
maximal (with respect to Bruhat order) $213$-avoiding permutation $\le\sigma$.
It is obtained from $\sigma$ by repeatedly interchanging $\sigma(i)$ and $\sigma(j)$ whenever $i<j<k$ and $\sigma(j)<\sigma(i)<\sigma(k)$.
\end{remark}

\subsection{}
We now introduce the key combinatorial property.

\begin{definition}
Let $\m$ be a multisegment.
\begin{enumerate}
\item We say that $\m$ is almost pairwise unlinked (\APU) if there exists a pairwise unlinked multisegment $\m'$ such that $\m'\adj\m$.
\item The complexity of $\m$ (denoted $\cmplx(\m)$) is the maximal integer $l\ge0$ for which there exists a chain of multisegments
$\m_l\adj\dots\adj\m_1\adj\m$.
\item The depth of $\m$ (denoted $\depth(\m)$) is the number of \APU\ multisegments $\obt\m$.
\item If $\m$ is regular, we say that $\m$ is balanced $\depth(\m)=\cmplx(\m)$.
\end{enumerate}
\end{definition}

Note that $\cmplx(\m)=0$ if and only if $\depth(\m)=0$ if and only if $\m$ is pairwise unlinked.

\begin{example} \label{ex: lecorig}
Let $\tseq=\tseq_{4,2}$. Here $\sigma_0(\tseq)=(1243)$ (where we use the notational convention $\sigma=(\sigma(1)\dots\sigma(k))$)
and $\m_{\sigma_0}(\tseq)=[1,5]+[2,4]$. In the following table we list the \APU\ multisegments $\m$ such that $\m_{\sigma_0}\adj\m$
and the corresponding permutation $\sigma$ such that $\m=\m_\sigma(\tseq)$.

\begin{tabular}{c|l}
 $\sigma$ & $\m_\sigma(\tseq)$\\
 \hline
 $(1342)$ & $[1,5]+[4,4]+[2,3]$\\
 $(3241)$ & $[4,5]+[2,4]+[1,3]$\\
 $(1423)$ & $[1,5]+[3,4]+[2,2]$\\
 $(4213)$ & $[3,5]+[2,4]+[1,2]$\\
 $(2143)$ & $[2,5]+[1,4]$
\end{tabular}

Let
\[
\m=\m_\sigma(\tseq)=[4,5]+[2,4]+[3,3]+[1,2]
\]
where $\sigma=(4231)$.
All the \APU\ multisegments in the table above are $\obt\m$ and therefore $\depth(\m)=5$.
On the other hand, the chain
\begin{align*}
[1,5]+[2,4]&\adj\\ [2,5]+[1,4]&\adj\\ [2,5]+[3,4]+[1,2]&\adj\\ [2,5]+[4,4]+[3,3]+[1,2]&\adj\m
\end{align*}
is of maximal length and therefore $\cmplx(\m)=4$. In conclusion, $\m$ is not balanced.
\end{example}

In general, it is easy to see by induction on the number of segments that if $\m=\Delta_1+\dots+\Delta_k$ is regular then
$\cmplx(\m)=\#\X_\m$. Indeed, if $\X_{\m}=\emptyset$ then $\m$ is pairwise unlinked and the assertion is clear.
Otherwise, let $(i,j)\in\X_\m$ and let $\m'\adj\m$ be the multisegment obtained from $\m$
by replacing the pair $(\Delta_i,\Delta_j)$ with its offspring. It is easy to see that $\# \X_{\m'}\le\# \X_{\m}-1$
with an equality if there does not exist an index $l$ such that $(i,l),(l,j)\in\X_\m$. The induction step follows.

Let $\tseq=\bitmplt$ be a regular \biseq\ and let $\sigma_0=\sigma_0(\tseq)$.
Then for any $\sigma\ge\sigma_0$ we have
\[
\depth(\m_\sigma(\tseq))=\#\{t\in\rflx:\sigma_0t\in [\sigma_0,\sigma]\}
\]
and
\[
\cmplx(\m_\sigma(\tseq))=\ell(\sigma)-\ell(\sigma_0).
\]
Thus, we get the following consequence of Proposition \ref{prop: LS}.
\begin{corollary}\label{cor: fgd and smth}
For any regular multisegment $\m$ we have $\depth(\m)\ge\cmplx(\m)$.
Moreover, if $\tseq=\bitmplt$ is a regular \biseq\ and $\sigma\ge\sigma_0(\tseq)$ then
$\m_\sigma(\tseq)$ is balanced if and only if $(\sigma,\sigma_0(\tseq))$ is a smooth pair
\end{corollary}

In Proposition \ref{prop: nonsmth} below we will give a simpler combinatorial characterization of balanced multisegments,
using the results of \cite{MR1990570, MR1994224, MR2015302, MR1853139}.

The following is an immediate consequence of \eqref{eq: redrelsmth}.

\begin{corollary} \label{cor: si0le}
Let $\m$ and $\m'$ be two regular multisegments. Write $\m=\Delta_1+\dots+\Delta_k$ and $\m'=\Delta_1'+\dots+\Delta'_{k'}$
with $e(\Delta_1)>\dots>e(\Delta_k)$, $e(\Delta'_1)>\dots>e(\Delta'_{k'})$. Assume that
\begin{enumerate}
\item $k'=k$.
\item $e(\Delta_i')=e(\Delta_i)$ for all $i$.
\item $b(\Delta'_i)\ge b(\Delta_i)$ for all $i$.
\item For all $i\ne j$ we have $b(\Delta_i)<b(\Delta_j)$ if and only if $b(\Delta'_i)<b(\Delta'_j)$.
\item $\m$ is balanced.
\end{enumerate}
Then $\m'$ is balanced.
\end{corollary}

Indeed, if we write $\m=\m_\sigma(\tseq)$ for a \biseq\ $\tseq=\bitmplt$ and $\sigma\in S_k$
then $\m'=\m_\sigma(\tseq)$ for some \biseq\ $\tseq'$ such that $\sigma_0(\tseq')\ge\sigma_0(\tseq)$.
Hence, the corollary follows from \eqref{eq: redrelsmth}.

Similarly, we can infer the following from the results of \S\ref{sec: rmvi}.

\begin{lemma} \label{lem: submlt}
A sub-multisegment of a balanced multisegment is balanced.
\end{lemma}

\begin{proof}
By induction, it is enough to check that if $\m$ is balanced then $\m'=\m-\Delta$ is balanced for any segment $\Delta$ in $\m$.
We may assume that $\m=\m_\sigma(\tseq)$ for a \biseq\ $\tseq=\bitmplt$ and let $\sigma_0=\sigma_0(\tseq)$.
Then for some $i$, $\m'=\m_{\sigma'}(\tseq')$ where $\sigma'=\rmv_i(\sigma)$ and $\tseq'$ is obtained from $\tseq$ by removing $a_i$ and $b_{\sigma(i)}$.
Let $\sigma_0'=\sigma_0(\tseq')\in S_{k-1}$ and let $\tilde\sigma$ be ``unflattening'' of $\sigma_0'$, namely the (unique) permutation in $S_k$ such that
$\tilde\sigma(i)=\sigma(i)$ and $\rmv_i(\tilde\sigma)=\sigma_0'$.
By Lemma \ref{lem: elemi} $\tilde\sigma\le\sigma$. It is also easy to see that $\sigma_0\le\tilde\sigma$.
Indeed, (cf.~\eqref{def: propsigma0}) the relations $a_{\tilde\sigma^{-1}(j)}\le b_j+1$, $j\ne\sigma(i)$ amount to the property $\sigma_0'\le\sigma'$,
while the corresponding inequality for $j=\sigma(i)$ also holds because $\tilde\sigma(i)=\sigma(i)$.
Thus, $(\sigma,\tilde\sigma)$ is a smooth pair (by \eqref{eq: redrelsmth}), and hence $(\sigma',\sigma_0')$ is a smooth pair by Lemma \ref{lem: rmvi}.
The lemma follows.
\end{proof}

\subsection{}
Corollary \ref{cor: fgd and smth} gives an efficient way to detect whether a given regular multisegment is balanced.
However, it will be useful to have another combinatorial criterion for balanced (regular) multisegments.
\begin{definition}
We say that a regular multisegment $\m=\Delta_1+\dots+\Delta_k$, $k\ge4$
with $e(\Delta_1)>\dots>e(\Delta_k)$ is of type $4231$ (resp., $3412$) if
\[
\Delta_{i+1}\prec\Delta_i,\ i=3,\dots,k-1,\
\Delta_3\prec \Delta_1\text{ and }b(\Delta_k)<b(\Delta_2)<b(\Delta_{k-1})
\]
(resp.,
\[
\Delta_{i+1}\prec\Delta_i,\ i=4,\dots,k-1,\
\Delta_4\prec\Delta_2,\text{ and }b(\Delta_3)<b(\Delta_k)<b(\Delta_1)<b(\Delta_l)
\]
where $l=2$ if $k=4$ and $l=k-1$ otherwise).
\end{definition}

\begin{example} \label{ex: bexam}
The ``minimal'' examples of multisegments of type $4231$ and $3412$ for $k\ge4$ are given by
\begin{gather*}
\m=[k,k+1]+[2,k]+[k-1]+[k-2]+\dots+[3]+[1,2]=\m_\sigma(\tseq_{k,2})\\
\text{where }\sigma(i)=\tau_{2,k-2}^{(1)}(i)=\begin{cases}k&i=1,\\2&i=2,\\k-i+2&i=3,\dots,k-1,\\1&i=k,\end{cases}
\end{gather*}
and
\begin{gather*}
\m=[3,k+2]+[k,k+1]+[1,k]+[k-1]+[k-2]+\dots+[4]+[2,3]=\m_\sigma(\tseq_{k,3})\\
\text{ where }\sigma(i)=\tau_{2,k-2}^{(2)}(i)=\begin{cases}3&i=1,\\k&i=2,\\1&i=3,\\
k-i+3&i=4,\dots,k-1,\\2&i=k,\end{cases}
\end{gather*}
respectively.

The corresponding drawings for $k=6$ are
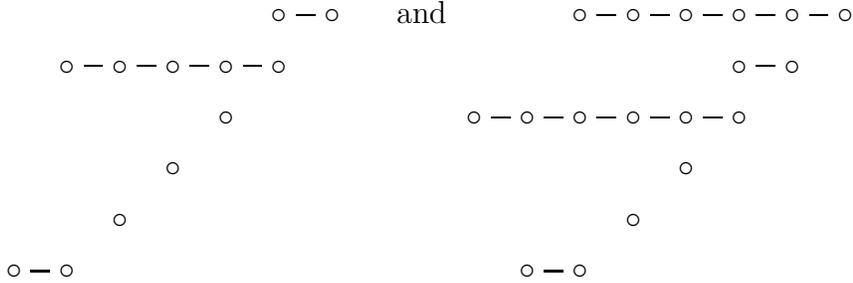
\[
\xymatrix@=0.6em{
&&&&&\circ\ar@{-}[r]&\circ\\
&\circ\ar@{-}[r]&\circ\ar@{-}[r]&\circ\ar@{-}[r]&\circ\ar@{-}[r]&\circ&\\
&&&&\circ&&&\\
&&&\circ&&&&\\
&&\circ&&&&&\\
\circ\ar@{-}[r]&\circ&&&&&}
\text{ and }
\xymatrix@=0.6em{
&&\circ\ar@{-}[r]&\circ\ar@{-}[r]&\circ\ar@{-}[r]&\circ\ar@{-}[r]&\circ\ar@{-}[r]&\circ\\
&&&&&\circ\ar@{-}[r]&\circ&\\
\circ\ar@{-}[r]&\circ\ar@{-}[r]&\circ\ar@{-}[r]&\circ\ar@{-}[r]&\circ\ar@{-}[r]&\circ&&&\\
&&&&\circ&&&\\
&&&\circ&&&&\\
&\circ\ar@{-}[r]&\circ&&&&&}
\]
\end{example}

More generally, using the notation of \S\ref{sec: minsingBL}, for any \biseq\ $\tseq=\bitmplt$ we have
\begin{multline} \label{eq: 4231 or 3412}
\text{$\m_\sigma(\tseq)$ is of type $4231$ (resp., $3412$) if and only if $\sigma=\tau_{2,k-2}^{(t)}$
and $\sigma_0(\tseq)\le\delta_{2,k-2}^{(t)}$}
\\\text{where $t=1$ (resp., $t=2$).}
\end{multline}
This easily follows from the defining property \eqref{def: propsigma0} of $\sigma_0(\tseq)$.
Note that since $\sigma_0(\tseq)$ is $213$-avoiding, $\sigma_0(\tseq)\le\delta_{2,k-2}^{(t)}$ if and only if $\sigma_0(\tseq)\le\sigma_0(\tseq_{k,t+1})$.

\begin{remark} \label{rem: types not GLS}
If $\m$ is of type $4231$ or $3412$ then $\m$ does not satisfy \GLS. Indeed, in the terminology of Remark \ref{rem: kirred},
$\X_\m$ contains the $k$ irreducible pairs $(2,1)$, $(3,1)$, $(i+1,i)$, $i=3,\dots,k-1$, $(k,2)$ (resp.,
$(3,1)$, $(3,2)$, $(4,2)$, $(i+1,i)$, $i=4,\dots,k-1$, $(k,1)$) in the $4231$ (resp., $3412$) case.
\end{remark}

\begin{proposition} \label{prop: nonsmth}
A regular multisegment $\m$ is balanced if and only if $\m$ does not admit a sub-multisegment of type $4231$ or $3412$.
\end{proposition}

\begin{proof}
For the `only if' direction we may assume, by Lemma \ref{lem: submlt}, that $\m$ itself is of type  $4231$ or $3412$.
In this case the claim follows from \eqref{eq: 4231 or 3412}, \eqref{eq: redrelsmth}, \eqref{eq: pairs not smooth} and Corollary \ref{cor: fgd and smth}.

For the converse direction, assume that $\m$ is regular but not balanced. Write $\m=\m_\sigma(\tseq)$ for a \biseq\ $\tseq=\bitmplt$
and let $\sigma_0=\sigma_0(\tseq)$. By Corollary \ref{cor: fgd and smth}, $(\sigma,\sigma_0)$ is not a smooth pair.
Let $\sigma_1,I,r,s,t$ be as in Theorem \ref{thm: BW}.
By passing to the sub-multisegment determined by $I$ we may assume that $I=\emptyset$.
Removing the segments $\Delta_i$, $2<i\le r$ (if $r>2$) we obtain a sub-multisegment of type $4231$
if $t=1$ and of type $3412$ if $t$ is either $2$ or $3$.
\end{proof}

\begin{remark} \label{rem: contrblncd}
Suppose that $\m$ is a regular multisegment which is $\rho$-contractible for some $\rho\in\supp\m$.
(See Definition \ref{def: contractible}.) Then $\m$ is balanced if and only if the $\rho$-contraction of $\m$ is balanced.
\end{remark}

\section{The main result} \label{sec: main}

Our main result is the following.
\begin{theorem} \label{thm: main}
Suppose that $\m$ is a regular multisegment. Then the following conditions are equivalent:
\begin{enumerate}
\item \label{item: mbal} $\m$ is balanced.
\item \label{item: zmLM} $\zele{\m}$ is \LM.
\item \label{item: mGLS} $\m$ satisfies \GLS.
\end{enumerate}
\end{theorem}

In this section we will prove the implications \ref{item: mbal}$\implies$\ref{item: zmLM} and \ref{item: mbal}$\implies$\ref{item: mGLS}.

In view of Propositions \ref{prop: LS} and \ref{prop: nonsmth}, Corollary \ref{cor: fgd and smth} and formula \eqref{def: sigma0},
Theorem \ref{thm: main} implies Theorem \ref{thm: maini} of the introduction, except for conditions \ref{cond: Lreal} and \ref{cond: detform}
which will be dealt with in the remark below and in \S\ref{sec: KLid} respectively.
(Recall that conditions \ref{cond: smlocus}--\ref{cond: KLall} are equivalent by \cite{MR788771}.)

\begin{remark}
Using the quantum Schur--Weyl duality \cite{MR1405590} we may translate Theorem \ref{thm: main} to the language of representation theory
of the quantum affine algebra $U_q(\widehat{\sll}_N)$ when $q$ is not a root of unity.
Recall that the finite-dimensional irreducible representations of $U_q(\widehat{\sll}_N)$ are parameterized by Drinfeld polynomials,
or what amounts to the same, by monomials in the formal variables $Y_{i,a}$, $i=1,\dots,N-1$, $a\in\C^*$ (e.g., \cite{MR2642561}).
We write $L(M)$ for the irreducible representation corresponding to a monomial $M$.

\begin{corollary} \label{cor: mainq}
Suppose that $\m=\sum_{i=1}^k[a_i,b_i]$ is a regular multisegment and $N>b_i-a_i+1$ for all $i$.
Let $M=\prod_{i=1}^kY_{b_i-a_i+1,q^{a_i+b_i}}$.
Then $L(M)$ is real (i.e., $L(M)\otimes L(M)$ is irreducible) if $\m$ is balanced. The converse also holds provided that
$N>2\sum_{i=1}^k(b_i-a_i+1)$.\footnote{This lower bound is far from optimal. See Remark \ref{rem: lwrbndN}.}
\end{corollary}
\end{remark}

\begin{proof}[Proof of the implication \ref{item: mbal}$\implies$\ref{item: zmLM} of Theorem \ref{thm: main}]
We argue by induction on the number of segments $k$ in $\m$. The base of the induction is the trivial case $k=0$.
For the induction step we use Lemma \ref{lem: albertoidea}.
Write $\m=\m_\sigma(\tseq)$ where $\tseq=\bitmplt$ is a \biseq\ and $\sigma\in S_k$.
Write $\Delta_i=[a_{\sigma^{-1}(i)},b_i]$ so that $\m=\Delta_1+\dots+\Delta_k$ and let $\sigma_0=\sigma_0(\tseq)\in S_k$.
By assumption $(\sigma,\sigma_0)$ is smooth.
For convenience write $a'_i=a_{k+1-i}$ and set $\sigma'(i)=\sigma(k+1-i)$ so that $b(\Delta_{\sigma'(1)})=a'_1>\dots>b(\Delta_{\sigma'(k)})=a'_k$.
We construct $\pi_1$ and $\pi_2$ as follows. Let $m\ge1$ be the largest integer such that $\sigma'(1)<\dots<\sigma'(m)$.
We define indices $n_1<\dots<n_m$ with $n_i\ge\sigma'(i)$ for all $i$ recursively as follows.
We take $n_m=\sigma'(m)$ and given $n_{i+1}$, $1\le i<m$ we define $n_i$ to be the largest index $\sigma'(i)\le j<n_{i+1}$
such that $\lshft{\Delta}_j\prec\Delta_{\sigma'(i)}$.
Let $\Delta_i'=[a'_i,b_{n_i}]\subset\Delta_{\sigma'(i)}$, $i=1,\dots,m$ and let $l>1$ be the largest integer $\le m$ such that
$\Delta'_l\not\prec\Delta'_{l-1}$ (i.e., such that $a'_{l-1}>b_{n_l}+1$) if such an index exists; otherwise let $l=1$.
We take $\m_1=\sum_{i=l}^m\Delta_i'$ and $\m_2=\sum_{i=1}^k\Delta''_i$ where $\Delta''_{\sigma'(i)}=
\Delta_{\sigma'(i)}\setminus\Delta'_i=[b_{n_i}+1,b_{\sigma'(i)}]$, $i=l,\dots,m$ and $\Delta''_j=\Delta_j$
if $j\notin\sigma'(\{l,\dots,m\})$.
Let $\pi_i=\zele{\m_i}$, $i=1,2$. Clearly $\pi_1$ is a ladder and $\m_2$ is regular.
The induction step will follow from Lemma \ref{lem: albertoidea} and the lemma below.
\end{proof}

\begin{lemma} \label{lem: indstep}
We have
\begin{enumerate}
\item $\pi\hookrightarrow\pi_1\times\pi_2$.
\item $\m_2$ is balanced. Hence, by induction hypothesis $\pi_2$ is \LM.
\item $\pi\times\pi_1$ is irreducible.
\end{enumerate}
\end{lemma}

\begin{proof}
Let $\m_3=\sum_{i=1}^{l-1}\Delta_{\sigma'(i)}$, $\m_4=\sum_{i=l}^m\Delta''_{\sigma'(i)}$ and
$\m_5=\sum_{j\notin\sigma'(\{1,\dots,m\})}\Delta_j$ so that $\m_2=\m_3+\m_4+\m_5$ and $\m_5<_b\m_4<_b\m_3$.
Set $\pi_i=\zele{\m_i}$, $i=3,4,5$. Thus,
\[
\pi_2\hookrightarrow\pi_3\times\pi_4\times\pi_5
\]
and therefore
\[
\pi_1\times\pi_2\hookrightarrow\pi_1\times\pi_3\times\pi_4\times\pi_5.
\]
Note that $\pi_1\times\pi_3$ is irreducible since no segment in $\m_1$ in linked with any segment in $\m_3$
(by the definition of $l$). Thus, $\pi_1\times\pi_3\simeq\pi_3\times\pi_1$ is a ladder, and in particular, \LM.
Thus, $\pi_1\times\pi_3\times\pi_4$ is \SI. Since $\lnrset(\pi_5)\cap\supp(\pi_i)=\emptyset$, $i=1,3,4$, it follows that
$\pi_1\times\pi_3\times\pi_4\times\pi_5$ is \SI\ (\cite[Lemma 1.5]{MR3573961}). Thus,
\[
\soc(\pi_1\times\pi_2)=\soc(\pi_1\times\pi_3\times\pi_4\times\pi_5)
\simeq\soc(\pi_3\times\pi_1\times\pi_4\times\pi_5)=\soc(\soc(\pi_3\times\soc(\pi_1\times\pi_4))\times\pi_5).
\]
Let $\m_6=\sum_{i=l}^m\Delta_{\sigma'(i)}$ and $\pi_6=\zele{\m_6}$. Note that $\pi_6$ is a ladder.
Moreover, by Lemma \ref{lem: chopladders} we have $\pi_6=\soc(\pi_1\times\pi_4)$. (Note that $\pi_4$ is a  ladder.)
Since $\m_5<_b\m_6<_b\m_3$ it follows that
\[
\soc(\pi_1\times\pi_2)=\zele{\m_3+\m_6+\m_5}=\zele{\m}
\]
and the first part follows.

Next, we show that $\m_2$ is balanced.
Note that if $i>l$ is such that $\Delta'_i\ne\Delta_{\sigma'(i)}$, i.e., $n_i>\sigma'(i)$, then it follows from the definition of $n_{i-1}$
and $l$ that $n_{i-1}\ge\sigma'(i)>\sigma'(i-1)$ and therefore $\Delta'_{i-1}\ne\Delta_{\sigma'(i-1)}$.
Let $l' \leq m$ be the smallest index $\ge l$ such that $\Delta'_{l'}=\Delta_{\sigma'(l')}$, i.e., such that $n_{l'}=\sigma'(l')$.
By the above, $n_i=\sigma'(i)$ for all $i\ge l'$.
Let $a''_i=b(\Delta_i)$ if $i\notin\sigma'(\{l,\dots,m\})$ and $a''_{\sigma'(i)}=b_{n_i}+1$ for $l\le i\le m$.
Thus, $\Delta''_i=[a''_i,b_i]$ for all $i$. It is easy to see that $a''_i<a''_j$ if and only if
$b(\Delta_i)<b(\Delta_j)$ (i.e., if and only if $\sigma'^{-1}(i)>\sigma'^{-1}(j)$).
It follows from Lemma \ref{lem: submlt} and Corollary \ref{cor: si0le} that $\m_2$ is balanced.

Finally, we show the irreducibility of $\pi\times\pi_1$ using the combinatorial condition given by Theorem \ref{thm: laddercomb}.
The condition $\LI(\pi_1,\pi)$ is easy to check and does not depend on the condition on $\m$.
Indeed, recall that (cf.~Definition \ref{def: XandY})
\[
\X_{\m_1;\m}=\{(i,j)\in\{l,\dots,m\}\times\{1,\dots,k\}:\Delta'_i\prec\Delta_j\}
\]
and
\[
\Y_{\m_1;\m}=\{(i,j)\in\{l,\dots,m\}\times\{1,\dots,k\}:\lshft{\Delta'_i}\prec\Delta_j\}.
\]
If $(i,j)\in \X_{\m_1;\m}$ then $j=\sigma'(i')$ for some $l<i'<i$ and hence $(i-1,j)\in \Y_{\m_1;\m}$.
Thus $(i,j)\mapsto(i-1,j)$ is a matching from $\X_{\m_1;\m}$ to $\Y_{\m_1;\m}$.

The condition $\RI(\pi_1,\pi)$ is more delicate and relies on the assumption on $\m$.
Recall
\[
\X_{\m;\m_1}=\{(i,j)\in\{1,\dots,k\}\times\{l,\dots,m\}:\Delta_i\prec\Delta'_j\}
\]
and
\[
\Y_{\m;\m_1}=\{(i,j)\in\{1,\dots,k\}\times\{l,\dots,m\}:\lshft{\Delta}_i\prec\Delta'_j\}.
\]
Let $l'$ be as above.
Consider first the case where $l'=l$, i.e., $\Delta'_i=\Delta_{\sigma'(i)}$ for all $i=l,\dots,m$.
For any $(i,j)\in \X_{\m;\m_1}$ let $h(i,j)$ be the largest index $r$ such that
$\Delta_i\prec\Delta_r$ and $\lshft{\Delta}_r\prec\Delta'_j$. Since $\Delta'_j=\Delta_{\sigma'(j)}$,
$h(i,j)$ is well-defined and $\sigma'(j)\le h(i,j)<i$.
We claim that the function $f:\X_{\m;\m_1}\rightarrow \Y_{\m;\m_1}$ given by $f(i,j)=(h(i,j),j)$ is a matching.
The only issue is injectivity.
Suppose on the contrary that $i_1<i_2$, $(i_1,j),(i_2,j)\in \X_{\m;\m_1}$ and $h(i_1,j)=h(i_2,j)$.
We cannot have $\Delta_{i_2}\prec\Delta_{i_1}$ since otherwise $h(i_2,j)\ge i_1$ while $h(i_1,j)<i_1$.
On the other hand, $b(\Delta_{i_1})<b(\Delta'_j)\le e(\Delta_{i_2})+1$, hence $b(\Delta_{i_1})\le e(\Delta_{i_2})$.
Since $i_1<i_2$ we must therefore have $b(\Delta_{i_1})<b(\Delta_{i_2})$.
Let $j'$ be the largest index $\ge j$ such that $\Delta_{i_1}\prec\Delta'_{j'}$ and $\Delta_{i_2}\prec\Delta'_{j'}$.
If $j'=m$ then $\Delta_{\sigma'(m+1)}+\Delta_{\sigma'(m)}+\Delta_{i_1}+\Delta_{i_2}$ forms a submultisegment of type $3412$
in contradiction with the assumption that $\m$ is balanced (Proposition \ref{prop: nonsmth}).
Thus, $j'<m$. We cannot have $\sigma'(j'+1)<i_1$ since otherwise $\Delta_{i_1},\Delta_{i_2}\prec\Delta'_{j'+1}$,
rebutting the maximality of $j'$. Also, we cannot have $\sigma'(j'+1)=i_1$ since $b(\Delta_{i_1})<b(\Delta_{i_2})$
and $i_2>\sigma'(j')$. Thus, $\sigma'(j'+1)>i_1$ and the definition of $n_{j'}$ would give $n_{j'}\ge i_1$, controverting our assumption.

Suppose now that $l'>l$. We claim that $(i,j)\mapsto(i,j+1)$ is a matching from $\X_{\m;\m_1}$ to $\Y_{\m;\m_1}$.
That is, for any $(i,j)\in \X_{\m;\m_1}$ we have $j<m$ and $\lshft{\Delta}_i\prec\Delta'_{j+1}$.
Suppose on the contrary that $(i,j)\in \X_{\m;\m_1}$ and either $j=m$ or $\lshft{\Delta}_i\not\prec\Delta'_{j+1}$.

Assume first that $j<l'$, i.e., that $n_j>\sigma'(j)$. Since $\Delta_i\prec\Delta'_j$ we must have $n_j<i$.
In particular, $j<m$. Since $\lshft{\Delta}_i\not\prec\Delta'_{j+1}$ we also have $n_{j+1}>i$.
From the definition of $n_j$ and the fact that $n_j<i$ it follows that $\lshft{\Delta}_i\not\prec\Delta'_j$.
Therefore, $e(\Delta_i)+1=b(\Delta'_j)$. However, then
$\Delta'_{j+1}\not\prec\Delta'_j$ since $e(\Delta'_{j+1})<e(\Delta_i)=b(\Delta'_j)-1$, repudiating the assumption that $j'\ge l$.


Assume now that $j\ge l'$, so that $\Delta'_j=\Delta_{\sigma'(j)}$.
For simplicity write $i_0=n_{l'-1}$ so that $\sigma'(l'-1)<i_0<\sigma'(l')$ and $\lshft{\Delta}_{i_0}\prec\Delta_{\sigma'(l'-1)}$.
If $b(\Delta_{i_0})>b(\Delta_i)$ (i.e., if $\sigma'^{-1}(i_0)>\sigma'^{-1}(i)$) then
\[
\Delta_{\sigma'(l'-1)}+\Delta_{i_0}+\Delta_{\sigma'(l')}+\dots+\Delta_{\sigma'(j)}+\Delta_i
\]
is a submultisegment of type $4231$ which violates the assumption that $\m$ is balanced by Proposition \ref{prop: nonsmth}.
Therefore $b(\Delta_{i_0})<b(\Delta_i)$.
Assume first that $j=m$ and let $i_1=\sigma'(m+1)$. by the definition of $m$, $i_1<\sigma'(m)$. By the definition of $l'$
we also have $i_1<\sigma'(l')$. Suppose first that $i_1>\sigma'(l'-1)$. Then as before,
\[
\Delta_{\sigma'(l'-1)}+\Delta_{i_1}+\Delta_{\sigma'(l')}+\dots+\Delta_{\sigma'(j)}+\Delta_i
\]
is a submultisegment of type $4231$, denying the assumption that $\m$ is balanced.
On the other hand, if $i_1<\sigma'(l'-1)$ then
\[
\Delta_{i_1}+\Delta_{\sigma'(l'-1)}+\Delta_{i_0}+\Delta_{\sigma'(l')}+\dots+\Delta_{\sigma'(j)}+\Delta_i
\]
is a submultisegment of type $3412$,
which once again violates the assumption on $\m$.
Thus $j<m$. It is now clear that $\lshft{\Delta}_i\prec\Delta'_{j+1}$, that is, $\sigma'(j+1)\le i$
for otherwise $n_j\le i<\sigma'(j+1)$, gainsaying the assumption $j\ge l'$.

This finishes the proof of the lemma, and hence the implication \ref{item: mbal}$\implies$\ref{item: zmLM} of Theorem \ref{thm: main}.
\end{proof}

\begin{example}
Consider the balanced multisegment
\[
\m=[12,14]+[9,13]+[6,12]+[3,10]+[1,9]+[2,8]+[5,6]+[4,5]
\]
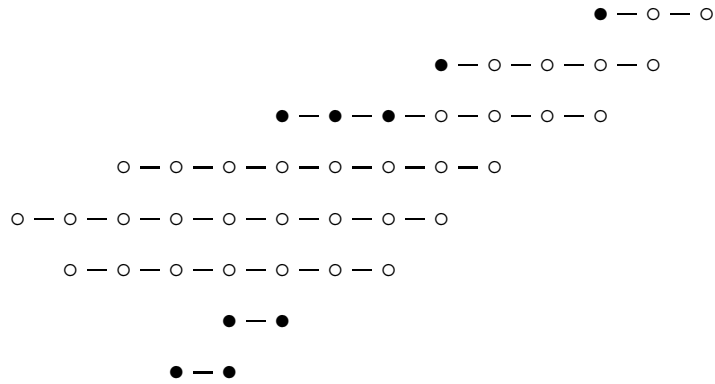
\[
\xymatrix@=0.6em{
&&&&&&&&&&&\bullet\ar@{-}[r]&\circ\ar@{-}[r]&\circ\\
&&&&&&&&\bullet\ar@{-}[r]&\circ\ar@{-}[r]&\circ\ar@{-}[r]&\circ\ar@{-}[r]&\circ\\
&&&&&\bullet\ar@{-}[r]&\bullet\ar@{-}[r]&\bullet\ar@{-}[r]&\circ\ar@{-}[r]&\circ\ar@{-}[r]&\circ\ar@{-}[r]&\circ\\
&&\circ\ar@{-}[r]&\circ\ar@{-}[r]&\circ\ar@{-}[r]&\circ\ar@{-}[r]&\circ\ar@{-}[r]&\circ\ar@{-}[r]&\circ\ar@{-}[r]&\circ\\
\circ\ar@{-}[r]&\circ\ar@{-}[r]&\circ\ar@{-}[r]&\circ\ar@{-}[r]&\circ\ar@{-}[r]&\circ\ar@{-}[r]&\circ\ar@{-}[r]&\circ\ar@{-}[r]&\circ\\
&\circ\ar@{-}[r]&\circ\ar@{-}[r]&\circ\ar@{-}[r]&\circ\ar@{-}[r]&\circ\ar@{-}[r]&\circ\ar@{-}[r]&\circ\\
&&&&\bullet\ar@{-}[r]&\bullet\\
&&&\bullet\ar@{-}[r]&\bullet}
\]
Here, in the notation of the proof of Theorem \ref{thm: main} we have
$\sigma'=(12378465)$, $m=5$, $n_1=3$, $n_2=5$, $n_3=6$, $n_4=7$, $n_5=8$, $l=2$, $l'=4$.
We marked by solid dots the part of $\Delta_{\sigma'(i)}$ which belongs to $\Delta'_i$.
Thus, $\m_1=[9,9]+[6,8]+[5,6]+[4,5]$ and
\[
\m_2=[12,14]+[10,13]+[9,12]+[3,10]+[1,9]+[2,8]
\]
\[
\xymatrix@=0.6em{
&&&&&&&&&&&\circ\ar@{-}[r]&\circ\ar@{-}[r]&\circ\\
&&&&&&&&&\circ\ar@{-}[r]&\circ\ar@{-}[r]&\circ\ar@{-}[r]&\circ\\
&&&&&&&&\circ\ar@{-}[r]&\circ\ar@{-}[r]&\circ\ar@{-}[r]&\circ\\
&&\circ\ar@{-}[r]&\circ\ar@{-}[r]&\circ\ar@{-}[r]&\circ\ar@{-}[r]&\circ\ar@{-}[r]&\circ\ar@{-}[r]&\circ\ar@{-}[r]&\circ\\
\circ\ar@{-}[r]&\circ\ar@{-}[r]&\circ\ar@{-}[r]&\circ\ar@{-}[r]&\circ\ar@{-}[r]&\circ\ar@{-}[r]&\circ\ar@{-}[r]&\circ\ar@{-}[r]&\circ\\
&\circ\ar@{-}[r]&\circ\ar@{-}[r]&\circ\ar@{-}[r]&\circ\ar@{-}[r]&\circ\ar@{-}[r]&\circ\ar@{-}[r]&\circ}
\]
\end{example}

\begin{proof}[Proof of the implication \ref{item: mbal}$\implies$\ref{item: mGLS} of Theorem \ref{thm: main}]
We will show in fact that there exists an extra-strong $\rltn$-matching from $\X_\m$ to $\Y_\m$.
(See Remark \ref{rem: strongmatch}.)
The argument parallels the one above for the implication \ref{item: mbal}$\implies$\ref{item: zmLM}.
In particular, we argue by induction on $k$.
We use the same notation as in the proof above.
Recall that $\Delta''_i=[a''_i,b_i]$ where $b(\Delta_i)\le a''_i$ for all $i$ and $a''_i<a''_j$ if and only if $b(\Delta_i)<b(\Delta_j)$.
It easily follows that
\begin{enumerate}
\item $\X_{\m_2}\subset \X_\m$; $\Y_{\m_2}\subset \Y_\m$.
\item If $(i,r)\in \X_\m$, $(r,j)\in \Y_\m$ and $(i,j)\in \X_{\m_2}$ then $(i,r)\in \X_{\m_2}$ and $(r,j)\in \Y_{\m_2}$.
\item Similarly, if $(r,j)\in \X_\m$, $(i,r)\in \Y_\m$ and $(i,j)\in \X_{\m_2}$ then necessarily $(r,j)\in \X_{\m_2}$ and $(i,r)\in \Y_{\m_2}$.
\end{enumerate}
Thus, the non-zero coordinates of $\{\xvec_{i,j}(\lambda):(i,j)\in \X_{\m_2}\}$ are confined to $\Y_{\m_2}$ and the entries coincide with those
of $\xvec_{i,j}(\lambda)$ with respect to $\m_2$ (and the same $\lambda_{i,j}$).
We have already shown in Lemma \ref{lem: indstep} that $\m_2$ is balanced. Hence, by induction hypothesis it suffices to check that
there exists a strong $\rltn$-matching $f:\X'\rightarrow\Y'$ where $\X'=\X_\m\setminus \X_{\m_2}$ and $\Y'=\Y_\m\setminus \Y_{\m_2}$.

Consider first the case $l'=l$. Thus, $\X'=\{(i,\sigma'(j))\in \X_\m:l\le j\le m\}$.
Let $f:\X'\rightarrow\Y'$ be defined by $f(i,\sigma'(j))=(h(i,j),\sigma'(j))$ where $h$ is as in Lemma \ref{lem: indstep}. Then $f$ is injective.
Let $\somele$ be the skewed lexicographic order on $\Y_\m$ given by $(i,j)\somele(i',j')$ if either $j'>j$ or ($j=j'$ and $i'\le i$).
It follows from the definition of $f$ that if $(i,j)\rltn(i',j')\in \Y_\m$ with $(i,j)\in\X'$ then $f(i,j)\somele(i',j')$.

Suppose now that $l'>l$. It is easy to see that
\[
\X'=\{(i,\sigma'(j))\in \X_\m:l\le j\le m,\Delta_i\prec\Delta'_j\}.
\]
Thus, the proof of Lemma \ref{lem: indstep} shows that the rule $(i,\sigma'(j))\mapsto (i,\sigma'(j+1))$ defines an injective function $f:\X'\rightarrow\Y'$.
Let $\somele$ be the skewed lexicographic order on $\Y_\m$ given by $(i,j)\somele(i',j')$ if either $i'<i$ or ($i'=i$ and $j'\ge j$).
Clearly, for any $(i,\sigma'(j))\in\X'$ and $(i,\sigma'(j))\rltn(i',\sigma'(j'))$ with $l\le j,j'\le m$ we have $(i,\sigma'(j+1))\somele(i',\sigma'(j'))$.

In both cases the proof is complete.
\end{proof}

\begin{remark}
Once proved, Theorem \ref{thm: main}, together with Lemma \ref{lem: submlt}, imply that if $\m$ is regular and $\zele{\m}\in\IrrS$ then
$\zele{\n}\in\IrrS$ for any sub-multisegment $\n$ of $\m$.
However, this is no longer true in the non-regular case. For instance, we can take $\n=[4,5]+[2,4]+[3,3]+[1,2]$ and $\m=\n+[2,3]$.
(It can be proved using Lemma \ref{lem: albertoidea} that $\zele{\m}$ is \LM, but in the next section we show that $\zele{\n}$ is not \LM.)
\end{remark}

\begin{remark}
Lemma \ref{lem: albertoidea}, though simple, provides a powerful tool for proving $\square$-irreducibility.
While we do not yet have sufficient evidence to make precise conjectures we may ask the following question.
Given $\pi\in\IrrS$ which is not supercuspidal, do there always exist $1\ne\pi_1,\pi_2\in\IrrS$ such that
$\pi\hookrightarrow\pi_1\times\pi_2$ and $\pi\times\pi_1$ is irreducible?
\end{remark}

\section{Basic cases} \label{sec: basicases}
It remains to prove the other implications of Theorem \ref{thm: main}.
In this section we carry out the main step by showing that for certain ``basic'' unbalanced regular multisegments $\m$
(generalizing Example \ref{ex: bexam}) which are introduced below, $\zele{\m}\notin\IrrS$ and $\m$ is not \GLS.

\subsection{}
The idea of the proof is the following.
Suppose that $\pi=\zele{\m}=\soc(\pi_1\times\pi_2)$ with $\pi_1,\pi_2\in\IrrS$. Then the \wtness\
$\Pi=\soc(\pi_1\times\soc(\pi_2\times\pi))\hookrightarrow\pi_1\times\pi_2\times\pi$ is irreducible, and hence
$\Pi\hookrightarrow\omega\times\pi$ for some $\omega\in\JH(\pi_1\times\pi_2)$.
If we can show that this is not possible unless $\omega=\pi$ then necessarily $\Pi\hookrightarrow\pi\times\pi$ and hence $\pi\notin\IrrS$
provided that $\Pi\ne\zele{\m+\m}$.

To facilitate the argument it is useful to introduce the following concept.
Let $\pi\in\Irr$ and $\pi_1,\pi_2,\Pi$ as before. We say that $(\pi_1,\pi_2)$ is a \emph{\spltng}\ for $\pi$ with \wtness\ $\Pi$ if in addition
\begin{enumerate}
\item $\m(\pi)=\m(\pi_1)+\m(\pi_2)$.
\item $\rnrset(\pi)=\rnrset(\pi_1)\cup\rnrset(\pi_2)$.
\item $\lnrset(\Pi)=\lnrset(\pi)$.
\end{enumerate}
In practice, determining $\Pi$ is the most technically involved step.

The properties above limit the possible $\omega$'s in $\JH(\pi_1\times\pi_2)$ such that $\Pi\hookrightarrow\omega\times\pi$.
Before making this more precise, we recall that the ascent set of a permutation $\sigma\in S_k$ is defined by
\[
\desl(\sigma):=\{1\le i<k:\sigma(i)<\sigma(i+1)\}\cup\{k\}.
\]
The following is an immediate consequence of Lemma \ref{lem: spclcaseextrho}. (Recall the notation of \S\ref{sec: biseq}.)
\begin{lemma} \label{lem: lnrdes}
Let $l\ge0$ and $\sigma\in S_k$ such that $\sigma(i)\le k+l+1-i$ for all $i>l+1$.
Let $\m=\m_\sigma(\tseq_{k,l})$ (see \eqref{def: akl}).
Then $\lnrset(\m)=\{\rho\nu_\rho^i:i\in\desl(\sigma)\}$ and $\rnrset(\m)=\{\rho\nu_\rho^{k+l-i}:i\in\desl(\sigma^{-1})\}$.
\end{lemma}

\begin{corollary} \label{cor: sigmarest}
Let $k,l\ge1$ and $\tau\in S_k$ be such that $\tau(i)\le k+l+1-i$ for all $i>l+1$.
Let $\pi=\zele{\m_\tau(\tseq_{k,l})}$ and assume that $(\pi_1,\pi_2)$ is a \spltng\ for $\pi$ with \wtness\ $\Pi\in\Irr$.
Let $\omega$ be an irreducible subquotient of $\pi_1\times\pi_2$ such that $\Pi\hookrightarrow\omega\times\pi$.
Then $\omega=\m_\sigma(\tseq_{k,l})$ where $\sigma\le\tau$, $\desl(\sigma)\subset\desl(\tau)$ and $\desl(\sigma^{-1})\subset\desl(\tau^{-1})$.
\end{corollary}

\begin{proof}
Since $\pi_1\times\pi_2\le\std(\m_\tau(\tseq_{k,l}))$ we have $\omega=\m_\sigma(\tseq_{k,l})$ for some $\sigma\le\tau$.
By Lemma \ref{lem: lnrprop} we have $\lnrset(\omega)\subset\lnrset(\Pi)=\lnrset(\pi)$ and $\rnrset(\omega)\subset\rnrset(\pi_1)\cup\rnrset(\pi_2)=\rnrset(\pi)$.
Thus, by Lemma \ref{lem: lnrdes}, $\desl(\sigma)\subset\desl(\tau)$ and $\desl(\sigma^{-1})\subset\desl(\tau^{-1})$ as required.
\end{proof}

It will be convenient to use the following notation: given a segment $\Delta$ and $k\ge0$ let
$\speh{\Delta}k=\Delta_1+\dots+\Delta_k$ where $\Delta_1=\Delta$ and $\Delta_{i+1}=\lshft{\Delta}_i$, $i=1,\dots,k-1$.

The ``basic'' multisegments come in three families which are introduced and analyzed in the following subsections.
\subsection{Basic multisegments of type $423{*}1$}

As in example \ref{ex: bexam} for $k\ge4$ let $\pi=\zele{\m}$ with
\begin{equation} \label{eq: basic4231}
\m=[k,k+1]+[2,k]+\speh{[k-1]}{k-3}+[1,2].
\end{equation}
By Remark \ref{rem: types not GLS}, $\m$ does not satisfy \GLS.

Let
\[
\Pi=\zele{\speh{[2,k+1]}2+\speh{[k,k+1]}k}=\zele{\speh{[2,k+1]}2}\times\zele{\speh{[k,k+1]}k}.
\]

\begin{proposition} \label{prop: basictype4231}
We have
\[
\Pi\hookrightarrow\pi\times\pi.
\]
In particular, $\pi\notin\IrrS$.
\end{proposition}

\begin{remark}
The case $k=4$ (where $\m=[4,5]+[2,4]+[3]+[1,2]$, see Example \ref{ex: lecorig}) is the original example given by Leclerc for an ``imaginary'' representation \cite{MR1959765}.
\end{remark}

Following the above-mentioned strategy we first show the following.

\begin{lemma} \label{lem: ws4231}
Let
\[
\pi_1=\zele{[k,k+1]+[2,k]},\ \ \pi_2=\zele{\speh{[k-1]}{k-3}+[1,2]}.
\]
Then $(\pi_1,\pi_2)$ is a \spltng\ for $\pi$ with \wtness\ $\Pi$.
\end{lemma}

\begin{proof}
Note that $\pi_1$ and $\pi_2$ are ladders and in particular $\pi_1,\pi_2\in\IrrS$.
Also, $\pi=\soc(\pi_1\times\pi_2)$, $\rnrset(\pi_1)\cup\rnrset(\pi_2)=\{[k]\}\cup\{[2]\}=\rnrset(\pi)$ and $\lnrset(\Pi)=\lnrset(\pi)=\{[2],[k]\}$.

Let
\[
\pi_3=\zele{[k-1,k+1]+[1,k]+\speh{[k-2,k-1]}{k-2}}.
\]
By Lemma \ref{lem: extractsegment} (applied repeatedly) we have
\[
\soc(\pi_1\times\pi_3)=\zele{[k,k+1]+[1,k]+\speh{[k-2,k-1]}{k-2}+\m(\soc(\zele{[2,k]}\times\zele{[k-1,k+1]}))}=\Pi.
\]
It remains to show that
\[
\soc(\pi_2\times\pi)=\pi_3.
\]
Note that $\del(\pi)=\del(\pi_3)=[k,k+1]$ and $\del(\pi^-)=\del(\pi_3^-)=[2,k]$.
Hence, by Lemma \ref{lem: suppn>suppm} (applied twice) it suffices to show that
\[
\soc(\pi_2\times(\pi^-)^-)=(\pi_3^-)^-.
\]
This is straightforward. Indeed,
\[
(\pi^-)^-=\zele{[2,k-1]+[1]}
\]
and by Lemma \ref{lem: extractsegment} we have
\begin{multline*}
\soc(\pi_2\times(\pi^-)^-)=\zele{\speh{[k-1]}{k-3}+[1]+\m(\soc(\zele{[1,2]}\times\zele{[2,k-1]}))}
\\=\zele{[1,k-1]+\speh{[k-1]}{k-1}}=(\pi_3^-)^-
\end{multline*}
as required.
\end{proof}

In order to conclude Proposition \ref{prop: 3412basic} it remains to show that
\[
\Pi\not\hookrightarrow\omega\times\pi
\]
for any irreducible subquotient $\omega$ of $\pi_1\times\pi_2$, other than $\pi$.

Recall that $\m=\m_{\sigma_1}(\tseq_{k,2})$ where
\[
\sigma_1(i)=\begin{cases}k&i=1,\\2&i=2,\\k+2-i&i=3,\dots,k-1,\\1&i=k.\end{cases}
\]
Note that $\sigma_1^{-1}=\sigma_1$ and $\desl(\sigma_1)=\{2,k\}$.

\begin{lemma} \label{lem: sigma14231}
Suppose that $\sigma\le\sigma_1$ and $\desl(\sigma)\cup\desl(\sigma^{-1})\subset\desl(\sigma_1)$.
Then either $\sigma=\sigma_1$ or
\[
\sigma(i)=\begin{cases}3-i&i=1,2,\\k+3-i&i=3,\dots,k.\end{cases}
\]
\end{lemma}

\begin{proof}
Let $i=\sigma(1)$ and $j=\sigma(2)$. Note that $i>j$ and $i,j$ determine $\sigma$ uniquely since $\sigma(3)>\sigma(4)>\dots>\sigma(k)$.
Suppose first that $i\ne k$. Then since $\sigma^{-1}(i)=1<\sigma^{-1}(i+1)$, we have $i=2$, in which case $j=1$.
On the other hand, if $i=k$ then $j\le\sigma_1(2)=2$. If $j=1$ then $2=\sigma^{-1}(1)>\sigma^{-1}(2)$
and therefore $i=2$, a contradiction. Hence, $j=2$ and $\sigma=\sigma_1$.
\end{proof}

It follows from Corollary \ref{cor: sigmarest} and Lemma \ref{lem: sigma14231} that if $\omega$ is an irreducible subquotient of $\pi_1\times\pi_2$,
other than $\pi$, such that $\Pi\hookrightarrow\omega\times\pi$ then necessarily $\omega=\zele{\speh{[2,k+1]}2}$. However, in this case
$\omega\times\pi$ is irreducible (e.g., \cite[Proposition 6.6]{MR3573961}) and is not equal to $\Pi$ (since $\m(\omega)+\m(\pi)\ne\m(\Pi)$).
We get a contradiction. This finishes the proof of Proposition \ref{prop: basictype4231}.

\begin{remark}
In the notation of Proposition \ref{prop: basictype4231}, we expect that $\pi\times\pi$ decomposes as a direct sum of $\Pi$
and $\zele{\m+\m}$. We will not say more about that here.
\end{remark}

\subsection{Basic multisegments of type $3{*}41{*}2$} \label{sec: 3*41*2}

Next, for $k>l>2$ consider $\pi=\zele{\m}$ where
\begin{equation} \label{eq: basic3412}
\m=[l,l+k-1]+\speh{[l-2,l+k-2]}{l-3}+[k,k+1]+[1,k]+\speh{[k-1]}{k-l-1}+[l-1,l].
\end{equation}
This is a generalization of Example \ref{ex: bexam} (in which $l=3$). Here is a drawing for $k=8$, $l=5$:
\[
\xymatrix@=0.6em{
&&&&&\circ\ar@{-}[r]&\circ\ar@{-}[r]&\circ\ar@{-}[r]&\circ\ar@{-}[r]&\circ\ar@{-}[r]&\circ\ar@{-}[r]&\circ\ar@{-}[r]&\circ\\
&&&\circ\ar@{-}[r]&\circ\ar@{-}[r]&\circ\ar@{-}[r]&\circ\ar@{-}[r]&\circ\ar@{-}[r]&\circ\ar@{-}[r]&\circ\ar@{-}[r]&\circ\ar@{-}[r]&\circ\\
&&\circ\ar@{-}[r]&\circ\ar@{-}[r]&\circ\ar@{-}[r]&\circ\ar@{-}[r]&\circ\ar@{-}[r]&\circ\ar@{-}[r]&\circ\ar@{-}[r]&\circ\ar@{-}[r]&\circ\\
&&&&&&&&\circ\ar@{-}[r]&\circ\\
&\circ\ar@{-}[r]&\circ\ar@{-}[r]&\circ\ar@{-}[r]&\circ\ar@{-}[r]&\circ\ar@{-}[r]&\circ\ar@{-}[r]&\circ\ar@{-}[r]&\circ\\
&&&&&&&\circ\\
&&&&&&\circ\\
&&&&\circ\ar@{-}[r]&\circ}
\]
By Remark \ref{rem: kirred} $\m$ does not satisfy \GLS. Indeed, $\X_\m$ contains the $k$ irreducible pairs
$(i+1,i)$, $i=1,\dots,l-3$, $(l,l-2)$, $(l,l-1)$, $(l+1,l-1)$, $(i+1,i)$, $i=l+1,\dots,k-1$, $(k,1)$.

Note that $\lnrset(\pi)=\{[l-2],[l],[k]\}$ and $\rnrset(\pi)=\{[l],[k],[k+2]\}$.

\begin{proposition} \label{prop: 3412basic}
Let
\begin{multline*}
\Pi=\zele{\speh{[l-2,l+k-1]}{l-2}+\speh{[l,l+k-1]}{l}+\speh{[k,k+1]}{k+2-l}}=\\
\zele{\speh{[l-2,l+k-1]}{l-2}}\times\zele{\speh{[l,l+k-1]}{l}}\times\zele{\speh{[k,k+1]}{k+2-l}}.
\end{multline*}
Then
\[
\Pi\hookrightarrow\pi\times\pi.
\]
In particular, $\pi\notin\IrrS$.
\end{proposition}

As before, the main step is the following.

\begin{lemma} \label{lem: weakerbasic3412}
Let
\begin{gather*}
\pi_1=\zele{[l,l+k-1]+\speh{[l-2,l+k-2]}{l-3}}\\
\pi_2=\zele{[k,k+1]+[1,k]+\speh{[k-1]}{k-l-1}+[l-1,l]}.
\end{gather*}
Then $(\pi_1,\pi_2)$ is a \spltng\ for $\pi$ with \wtness\ $\Pi$.
\end{lemma}

\begin{proof}
By the ``if'' part of Theorem \ref{thm: main} $\pi_1,\pi_2\in\IrrS$.
It is clear that $\pi=\soc(\pi_1\times\pi_2)$, $\rnrset(\pi_1)\cup\rnrset(\pi_2)=\{[k+2]\}\cup\{[l],[k]\}=\rnrset(\pi)$
and $\lnrset(\Pi)=\lnrset(\pi)$. Let
\[
\pi_3=\zele{\speh{[l-1,l+k-1]}{l-3}+\speh{[2,k+1]}2+[1,k+2]+\speh{[k,k+1]}{k+2-l}}.
\]
We first show that $\soc(\pi_1\times\pi_3)=\Pi$.

Let $\pi_4=\zele{\speh{[l-2,l+k-2]}{l-3}}$ and $\pi_5=\soc(\pi_4\times\pi_3)$. By Lemma \ref{lem: extractsegment} we have
\[
\soc(\pi_1\times\pi_3)=\zele{[l,l+k-1]+\m(\pi_5)}.
\]
Note that $\del(\pi_3)=[k+2,k+l-1]$ and
\[
\pi_3^-=\zele{\speh{[l-1,l+k-2]}{l-1}+[1,k+1]+\speh{[k,k+1]}{k+2-l}}.
\]
Hence, by Lemma \ref{lem: suppn>suppm} $\del(\pi_5)=[k+2,k+l-1]$ and $\pi_5^-=\soc(\pi_4\times\pi_3^-)$.
By Theorem \ref{thm: laddercomb} we have $\LI(\pi_4,\pi_3^-)$ and therefore
\[
\pi_5^-=\zele{\speh{[l-2,l+k-2]}{l-2}+\speh{[l-1,l+k-2]}{l-1}+\speh{[k,k+1]}{k+2-l}}.
\]
Thus,
\[
\pi_5=\zele{\speh{[l-2,l+k-1]}{l-2}+\speh{[l-1,l+k-2]}{l-1}+\speh{[k,k+1]}{k+2-l}}.
\]
All in all, $\soc(\pi_1\times\pi_3)=\Pi$ as claimed.

Ir remains to show that
\begin{equation} \label{eq: tau2tausigma}
\soc(\pi_2\times\pi)=\pi_3.
\end{equation}

Assume first that $l>3$. Then, since $[1],[l-1]\notin\lnrset(\pi)$ we have by Corollary \ref{cor: extractrho}
\begin{equation} \label{eq: soctau2tau}
\soc(\pi_2\times\pi)=\soc([l-1]\times\soc([1]\times\soc(\pi_6\times\pi)))
\end{equation}
where (see Lemma \ref{lem: spclcaseextrho})
\[
\pi_6=\lderiv_{[l-1]}(\lderiv_{[1]}(\pi_2))=\zele{[k,k+1]+[2,k]+\speh{[k-1]}{k-l}}.
\]
Note that by Theorem \ref{thm: laddercomb} $\pi_6\times\pi_1$ is irreducible. Therefore,
\begin{multline*}
\soc(\pi_6\times\pi)=\soc(\pi_6\times\pi_1\times\pi_2)=
\soc(\pi_1\times\soc(\pi_6\times\pi_2))\\=
\zele{[l,l+k-1]+\speh{[l-2,l+k-2]}{l-3}+\m(\soc(\pi_6\times\pi_2))}.
\end{multline*}
Let $\pi_7=\zele{[k,k+1]+\speh{[k-1]}{k-l-1}}$. By Lemma \ref{lem: extractsegment} we have
\[
\soc(\pi_6\times\pi_2)=\zele{[l-1,l]+[1,k]+\m(\soc(\pi_6\times\pi_7))}
\]
and by Corollary \ref{cor: extractrho}
\[
\soc(\pi_6\times\pi_7)=\soc(\soc(\pi_6\times\zele{\speh{[k]}{k-l}})\times[k+1]).
\]
Moreover, by Lemma \ref{lem: extractsegment} and Lemma \ref{lem: chopladders}
\begin{multline*}
\soc(\pi_6\times\zele{\speh{[k]}{k-l}})=\zele{[k,k+1]+\m(\soc(\zele{[2,k]+\speh{[k-1]}{k-l}}\times \zele{\speh{[k]}{k-l}}))}
\\=\zele{[2,k]+\speh{[k,k+1]}{k+1-l}}.
\end{multline*}
Using Lemma \ref{lem: soctimesrho} it follows that
\[
\soc(\pi_6\times\pi_7)=\zele{[2,k+1]+\speh{[k,k+1]}{k+1-l}}
\]
and hence
\[
\soc(\pi_6\times\pi_2)=\zele{\speh{[2,k+1]}2+\speh{[k,k+1]}{k+2-l}}
\]
and
\[
\soc(\pi_6\times\pi)=\zele{[l,l+k-1]+\speh{[l-2,l+k-2]}{l-3}+\speh{[2,k+1]}2+\speh{[k,k+1]}{k+2-l}}.
\]
The relation \eqref{eq: tau2tausigma} now follows from \eqref{eq: soctau2tau} using Lemma \ref{lem: soctimesrho}.

Consider now the remaining case $l=3$. We first write using Corollary \ref{cor: extractrho}
\[
\soc(\pi_2\times\pi)=\soc([2]\times\soc(\pi_8\times\pi))
\]
where
\[
\pi_8=\lderiv_{[2]}(\pi_2)=\zele{[k,k+1]+[1,k]+\speh{[k-1]}{k-3}}.
\]
By Lemma \ref{lem: extractsegment} we have
\[
\soc(\pi_8\times\pi)=\zele{[2,3]+\m(\soc(\pi_8\times\pi_9))}
\]
where
\[
\pi_9=\zele{[3,k+2]+[k,k+1]+[1,k]+\speh{[k-1]}{k-4}}.
\]
Now, since $[1]\times\pi_8$ is irreducible we have by Corollary \ref{cor: extractrho}
\[
\soc(\pi_8\times\pi_9)=\soc([1]\times[1]\times\soc(\pi_{10}\times\pi_{11}))
\]
where
\[
\pi_{10}=\lderiv_{[1]}(\pi_8)=\zele{[k,k+1]+[2,k]+\speh{[k-1]}{k-3}}
\]
and
\[
\pi_{11}=\lderiv_{[1]}(\pi_9)=\zele{[3,k+2]+[k,k+1]+[2,k]+\speh{[k-1]}{k-4}}.
\]
Again by Corollary \ref{cor: extractrho}
\[
\soc(\pi_{10}\times\pi_{11})=\soc([2]\times\soc(\pi_{12}\times\pi_{11}))
\]
where (by Lemma \ref{lem: spclcaseextrho})
\[
\pi_{12}=\lderiv_{[2]}(\pi_{10})=\zele{[k,k+1]+[3,k]+\speh{[k-1]}{k-3}}.
\]
By Lemma \ref{lem: extractsegment} we have
\[
\soc(\pi_{12}\times\pi_{11})=\zele{[2,k]+\m(\soc(\pi_{12}\times\pi_{13}))}
\]
where
\[
\pi_{13}=\zele{[3,k+2]+[k,k+1]+\speh{[k-1]}{k-4}}=\zele{[3,k+2]}\times\zele{[k,k+1]+\speh{[k-1]}{k-4}}.
\]
Clearly,
\[
\soc(\pi_{12}\times\pi_{13})=\zele{[3,k+2]+\m(\soc(\pi_{12}\times\zele{[k,k+1]+\speh{[k-1]}{k-4}}))}
\]
while by Corollary \ref{cor: extractrho}
\[
\soc(\pi_{12}\times\zele{[k,k+1]+\speh{[k-1]}{k-4}})=\soc(\soc(\pi_{12}\times\zele{\speh{[k]}{k-3}})\times[k+1]).
\]
Since
\begin{multline*}
\soc(\pi_{12}\times\zele{\speh{[k]}{k-3}})=\zele{[k,k+1]+[3,k]+\m(\soc(\zele{\speh{[k-1]}{k-3}}\times\zele{\speh{[k]}{k-3}}))}\\
=\zele{[3,k]+\speh{[k,k+1]}{k-2}}
\end{multline*}
we obtain (using Lemma \ref{lem: soctimesrho})
\begin{gather*}
\soc(\pi_{12}\times\zele{[k,k+1]+\speh{[k-1]}{k-4}})=\zele{[3,k+1]+\speh{[k,k+1]}{k-2}},\\
\soc(\pi_{12}\times\pi_{13})=\zele{[3,k+2]+[3,k+1]+\speh{[k,k+1]}{k-2}},\\
\soc(\pi_{12}\times\pi_{11})=\zele{[2,k]+[3,k+2]+[3,k+1]+\speh{[k,k+1]}{k-2}},\\
\soc(\pi_{10}\times\pi_{11})=\zele{[2,k]+[2,k+2]+[3,k+1]+\speh{[k,k+1]}{k-2}},\\
\soc(\pi_8\times\pi_9)=\zele{[1,k]+[1,k+2]+[3,k+1]+\speh{[k,k+1]}{k-2}},\\
\soc(\pi_8\times\pi)=\zele{[1,k]+[1,k+2]+[3,k+1]+\speh{[k,k+1]}{k-1}},
\end{gather*}
and finally
\[
\soc(\pi_2\times\pi)=\zele{[1,k+2]+\speh{[2,k+1]}2+\speh{[k,k+1]}{k-1}}=\pi_3
\]
as required.
\end{proof}

In order to conclude Proposition \ref{prop: 3412basic} it remains to show that
\[
\Pi\not\hookrightarrow\omega\times\pi
\]
for any irreducible subquotient $\omega$ of $\pi_1\times\pi_2$, other than $\pi$.
Recall that $\m(\pi)=\m_{\sigma_1}(\tseq_{k,l})$ where
\[
\sigma_1(i)=\begin{cases}l&i=1,\\l-i&i=2,\dots,l-2,\\k&i=l-1,\\1&i=l,\\l+k-i&i=l+1,\dots,k-1,\\l-1&i=k.\end{cases}
\]
Note that $\sigma_1^{-1}=\sigma_1$ and $\desl(\sigma_1)=\{l-2,l,k\}$.

\begin{lemma} \label{lem: sigma3412}
Suppose that $\sigma\le\sigma_1$ and $\desl(\sigma)\cup\desl(\sigma^{-1})\subset\desl(\sigma_1)$.
Then $\sigma$ is one of the following four permutations:
\begin{gather*}
\sigma=\sigma_1,\\
\sigma(i)=\begin{cases}l&i=1,\\l-i&i=2,\dots,l-2,\\l-1&i=l-1,\\1&i=l,\\k+l+1-i&i=l+1,\dots,k,\end{cases}\\
\sigma(i)=\begin{cases}l-1-i&i=1,\dots,l-2,\\k&i=l-1,\\l&i=l,\\k+l-i&i=l+1,\dots,k-1,\\l-1&i=k,\end{cases}\\
\sigma(i)=\begin{cases}l-1-i&i=1,\dots,l-2,\\l&i=l-1,\\l-1&i=l,\\k+l+1-i&i=l+1,\dots,k.\end{cases}
\end{gather*}
In particular, $\sigma=\sigma^{-1}$ and if $\sigma\ne\sigma_1$ then $\sigma$ is smooth.
\end{lemma}

\begin{proof}
We have $\sigma(1)\le\sigma_1(1)=l$ and $\sigma(i)>\sigma(i+1)$ for all $1\le i<l-2$.
Also $\sigma(1)\ne l-1$ since $\sigma^{-1}(l-1)>\sigma^{-1}(l)$. Thus, either
$\sigma(i)=l-1-i$ for all $1\le i\le l-2$ or $\sigma(1)=l$. In the latter case, if $l>3$ then $\sigma(2)\le\sigma_1(2)=l-2$ and
therefore there exists $1<j<l$ such that $\sigma(i)=l-i$ for all $1<i<j$ and $\sigma(i)=l-1-i$ for all $j\le i\le l-2$.
In fact, $j=l-1$ for otherwise $\sigma^{-1}(l-j)\ge l-1>j=\sigma^{-1}(l-j-1)$ in contradiction with the assumption on $\sigma$.
Thus, either $\sigma(i)=l-1-i$ for all $1\le i\le l-2$ or $\sigma(1)=l$ and $\sigma(i)=l-i$ for all $1<i\le l-2$.
By a similar reasoning, either $\sigma(i)=k+l+1-i$ for all $l<i\le k$ or $\sigma(k)=l-1$ and $\sigma(i)=k+l-i$ for all $l<i<k$.
Taking into account the condition $\sigma(l-1)>\sigma(l)$, $\sigma$ must be one of the four possibilities listed in the statement of the lemma.
\end{proof}

Assume for simplicity that $\rho^\vee=\rho\nu_\rho^{k+l}$, or equivalently, that $\pi^\vee=\pi$.
(We may do so since by Theorem \ref{thm: indepcspline}, the validity of Theorem \ref{thm: main} is independent of the choice of $\rho$.)
It follows from Corollary \ref{cor: sigmarest}, Lemma \ref{lem: sigma3412}, the ``if'' direction of Theorem \ref{thm: main} and \eqref{eq: contrasigma}
that if $\omega$ is an irreducible subquotient of $\pi_1\times\pi_2$, other than $\pi$, such that $\Pi\hookrightarrow\omega\times\pi$
then $\omega\in\IrrS$ and $\omega=\omega^\vee$.
However, in this case, it would follow from Lemma \ref{lem: selfdualcases} below\footnote{Alternatively, we could compute
$\soc(\omega\times\pi)$ directly. This will unnecessitate the assumption that $\rho$ is essentially self-dual.}
that $\omega\times\pi$ is irreducible and we obtain a contradiction
(since $\m(\pi)$ is not a sub-multisegment of $\m(\Pi)$). This finishes the proof of Proposition \ref{prop: 3412basic}.

\begin{lemma} \label{lem: selfdualcases}
Suppose that $\pi_1$, $\pi_2$ and $\pi$ are irreducible and self-dual and $\pi\hookrightarrow\pi_1\times\pi_2$.
If at least one of $\pi_1$ and $\pi_2$ is \LM\ then $\pi_1\times\pi_2$ is irreducible. In particular, $\m(\pi)=\m(\pi_1)+\m(\pi_2)$.
\end{lemma}

\begin{proof}
Since $\pi$ is a subrepresentation of $\pi_1\times\pi_2$, $\pi^\vee$ is a quotient of $\pi_1^\vee\times\pi_2^\vee$.
Thus, by the self-duality assumption, $\pi$ is a quotient of $\pi_1\times\pi_2$ as well.
Since by assumption $\pi_1\times\pi_2$ is \SI, it must be irreducible.
\end{proof}

\subsection{Basic multisegments of type $34{*}12$}

Finally, as in Example \ref{exam: 34*12} consider for $k>4$
\[
\pi=\zele{\m},\ \m=[k-1,2k-2]+[k,2k-3]+\speh{[k-2,2k-4]}{k-4}+[1,k]+[2,k-1].
\]

\begin{proposition} \label{prop: lastcase3412}
We have $\Pi\hookrightarrow\pi\times\pi$ where
\begin{multline*}
\Pi=\zele{\speh{[k,2k-3]}{k-1}+[1,2k-2]+\speh{[k-1,2k-2]}{k-1}}=\\
\zele{\speh{[k,2k-3]}{k-1}}\times\zele{[1,2k-2]}\times\zele{\speh{[k-1,2k-2]}{k-1}}.
\end{multline*}
In particular, $\pi$ is not \LM.
\end{proposition}

To that end we first show

\begin{lemma} \label{lem: weakspec2}
Let
\[
\pi_1=\zele{[k-1,2k-2]},\ \ \pi_2=\zele{[k,2k-3]+\speh{[k-2,2k-4]}{k-4}+[1,k]+[2,k-1]}.
\]
Then $(\pi_1,\pi_2)$ is a \spltng\ for $\pi$ with \wtness\ $\Pi$.
\end{lemma}

\begin{proof}
Clearly $\pi_1$ is \LM\ and the same is true for $\pi_2$ by the ``if'' part of Theorem \ref{thm: main}.
It is also clear that $\pi=\soc(\pi_1\times\pi_2)$, $\rnrset(\pi_1)\cup\rnrset(\pi_2)=\{[2k-2]\}\cup\{[k-1],[k]\}=\rnrset(\pi)$
and $\lnrset(\Pi)=\lnrset(\pi)=\{[1],[k-1],[k]\}$.
Let
\[
\pi_3=\zele{\speh{[k,2k-3]}{k-1}+[1,2k-2]+\speh{[k-2,2k-3]}{k-2}}.
\]
Clearly, $\Pi=\soc(\pi_1\times\pi_3)$.
It remains to show that
\[
\soc(\pi_2\times\pi)=\pi_3.
\]
Since $\del(\pi_3)=\del(\pi)=[2k-2]$ it is enough to show by Lemma \ref{lem: suppn>suppm} that $\soc(\pi_2\times\pi^-)=\pi_3^-$.
For brevity set $\Delta=[k,2k-3]$.
We have
\[
\pi^-=\zele{\Delta+\speh{^+\Delta}{k-3}+[1,k]+[2,k-1]}.
\]
Since $\lnrset(\pi^-)=\{[1],[k-1],[k]\}$ we have by Corollary \ref{cor: extractrho}
\[
\soc(\pi_2\times\pi^-)=\soc([k-2]\times\soc([k-3]\times\dots\times\soc([2]\times\soc(\pi_4\times\pi^-))\dots))
\]
where (using Lemma \ref{lem: spclcaseextrho})
\[
\pi_4=\lderiv_{[2]}(\dots\lderiv_{[k-3]}(\lderiv_{[k-2]}(\pi_2))\dots)=\zele{\speh{\Delta}{k-3}+[1,k]+[3,k-1]}.
\]
Since $[1]\times\pi_4$ is irreducible and $\lmlt_{[1]}(\pi_4)=\lmlt_{[1]}(\pi^-)=1$ we have by Corollary \ref{cor: extractrho}
\[
\soc(\pi_4\times\pi^-)=\soc([1]\times[1]\times\soc(\pi_5\times\pi_6))
\]
where
\[
\pi_5=\lderiv_{[1]}(\pi_4)=\zele{\speh{\Delta}{k-3}+[2,k]+[3,k-1]}
\]
and
\[
\pi_6=\lderiv_{[1]}(\pi^-)=\zele{\Delta+\speh{^+\Delta}{k-2}+[2,k-1]}.
\]
By Lemma \ref{lem: extractsegment} we have
\[
\soc(\pi_5\times\pi_6)=\zele{[2,k-1]+\m(\soc(\pi_5\times\pi_7))}
\]
where
\[
\pi_7=\zele{\Delta+\speh{^+\Delta}{k-2}}=\zele{\Delta}\times\zele{\speh{^+\Delta}{k-2}}.
\]
We have
\[
\soc(\pi_5\times\pi_7)=\zele{\speh{^+\Delta}{k-2}}\times\soc(\pi_5\times\zele{\Delta})
\]
and since $^-\Delta\cap\rnrset(\pi_5)=\emptyset$,
\[
\soc(\pi_5\times\zele{\Delta})=\soc(\soc(\pi_5\times[k])\times\zele{^-\Delta}).
\]
Also,
\[
\soc(\pi_5\times[k])=\zele{\speh{\Delta}{k-2}+[2,k]}.
\]
Thus,
\[
\soc(\pi_5\times\zele{\Delta})=\zele{\speh{\Delta}{k-2}+[2,2k-3]}
\]
and hence by Lemma \ref{lem: soctimesrho}
\[
\soc(\pi_5\times\pi_7)=\zele{\speh{^+\Delta}{k-2}+\speh{\Delta}{k-2}+[2,2k-3]},
\]
\[
\soc(\pi_5\times\pi_6)=\zele{\speh{^+\Delta}{k-2}+\speh{\Delta}{k-1}+[2,2k-3]},
\]
\[
\soc(\pi_4\times\pi^-)=\zele{[1,2k-3]+[1,k]+\speh{^+\Delta}{k-3}+\speh{\Delta}{k-1}},
\]
and finally (again by Lemma \ref{lem: soctimesrho})
\[
\soc(\pi_2\times\pi^-)=\zele{[1,2k-3]+\speh{[k-2,2k-3]}{k-2}+\speh{\Delta}{k-1}}=\pi_3^-
\]
as required.
\end{proof}

We can write $\pi=\m_{\sigma_1}(\tseq_{k,k-1})$ where
\[
\sigma_1(i)=\begin{cases}i+k-2&i=1,2,\\k+1-i&i=3,\dots,k-2,\\i-k+2&i=k-1,k.\end{cases}
\]
Note that $\sigma_1^{-1}=\sigma_1$ and $\desl(\sigma_1)=\{1,k-1,k\}$.
Proposition \ref{prop: lastcase3412} is concluded from Lemma \ref{lem: weakspec2} exactly as before using the following elementary lemma.

\begin{lemma}
Suppose that $\sigma\le\sigma_1$ and $\desl(\sigma)\cup\desl(\sigma^{-1})\subset\desl(\sigma_1)$.
Then $\sigma$ is one of the following four permutations:
\begin{gather*}
\sigma=\sigma_1,\\
\sigma(i)=\begin{cases}1&i=1,\\k+2-i&i=2,\dots,k,\end{cases}\\
\sigma(i)=\begin{cases}k-i&i=1,\dots,k-1,\\k&i=k,\end{cases}\\
\sigma(i)=\begin{cases}i&i=1,k,\\k+1-i&i=2,\dots,k-1.\end{cases}
\end{gather*}
In particular, $\sigma=\sigma^{-1}$ and if $\sigma\ne\sigma_1$ then $\sigma$ is smooth.
\end{lemma}

\begin{proof}
Let $i=\sigma(1)$ and $j=\sigma(k)$.
Then $i\le\sigma_1(1)=k-1$ and $\sigma^{-1}(i)=1<\sigma^{-1}(i+1)$, hence $i\in\{1,k-1\}$.
Similarly $j\in\{2,k\}$. Since $\sigma(2)>\sigma(3)>\dots>\sigma(k-1)$, $\sigma$ must be one of the four
possibilities above.
\end{proof}

\section{End of proof of Theorem \ref{thm: main}} \label{sec: comproof}

In this section we complete the proof of the remaining parts of Theorem \ref{thm: main}.
Namely, we show that if $\m$ is a regular unbalanced multisegment then $\zele{\m}$ is not \LM\ and $\m$ does not satisfy \GLS.
We will achieve this by reducing the statement to the cases considered in the previous section.
The first reduction uses Lemma \ref{lem: 1stred}. It motivates the following definition.

\begin{definition}
Let $\m=\Delta_1+\dots+\Delta_k$ be a regular unbalanced multisegment.
We say that $\m$ is \emph{minimal unbalanced} if $\m-\Delta$ is balanced for every detachable segment $\Delta$ of $\m$. (See Definition \ref{def: detachable}.)
\end{definition}

We can explicate the minimal unbalanced multisegments as follows.

\begin{lemma} \label{lem: MU}
Suppose that $\m=\Delta_1+\dots+\Delta_k$ is a regular multisegment with $e(\Delta_1)>\dots>e(\Delta_k)$.
Then $\m$ is minimal unbalanced if and only if precisely one of the following three conditions holds.
\begin{enumerate}
\item (case $4{*}23{*}1$)
\begin{enumerate}
\item $b(\Delta_k)<b(\Delta_i)<b(\Delta_1)$ for all $1<i<k$.
\item There do not exist $1<i,j<k-1$ such that $b(\Delta_{i+1})<b(\Delta_j)<b(\Delta_i)$.
\item There exists $i$ such that $\Delta_{i+1}\not\prec\Delta_i$.
\item Let $r=\max\{i:\Delta_{i+1}\not\prec\Delta_i\}$.
Then $\Delta_{r+1}\prec\Delta_1$ and $r<k-1$.
\end{enumerate}

\item (case $3{*}41{*}2$) There exists $1<r<k-1$ such that if $\tau$ is the transposition $r\leftrightarrow r+1$ then
\begin{enumerate}
\item $\Delta_{\tau(i+1)}\prec\Delta_{\tau(i)}$, $i=1,\dots,r-1,r+1,\dots,k-1$.
\item $b(\Delta_{\tau(2)})<b(\Delta_k)<b(\Delta_1)<b(\Delta_{\tau(k-1)})$.
\end{enumerate}

\item (case $34{*}12$) $k>4$, $\Delta_2\subset\Delta_1$, $\Delta_3\prec\Delta_1$, $\Delta_{i+1}\prec\Delta_i$, $i=3,\dots,k-3$,
$\Delta_k\prec\Delta_{k-2}$, $\Delta_k\subset\Delta_{k-1}$, $\Delta_k\prec\Delta_2$.
\end{enumerate}
\end{lemma}

Here is an example of the case $4{*}23{*}1$ with $k=8$ and $r=5$:
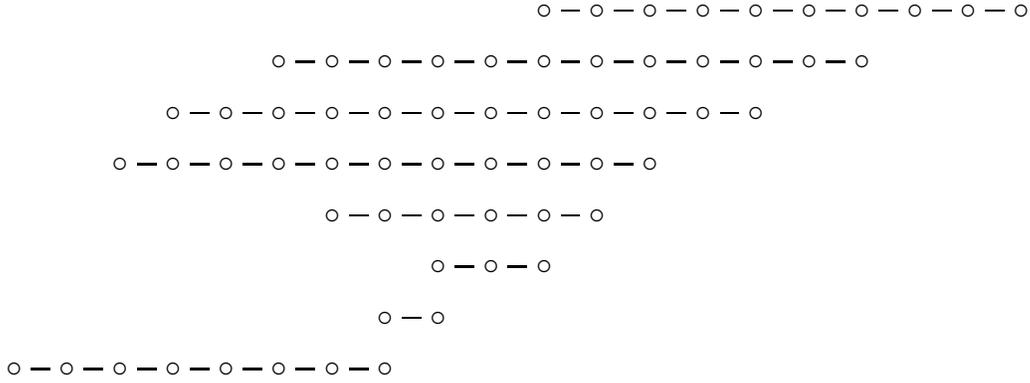
\[
\xymatrix@=0.6em{
&&&&&&&&&&\circ\ar@{-}[r]&\circ\ar@{-}[r]&\circ\ar@{-}[r]&\circ\ar@{-}[r]&\circ\ar@{-}[r]&\circ\ar@{-}[r]&\circ\ar@{-}[r]&\circ\ar@{-}[r]&\circ\ar@{-}[r]&\circ&\\
&&&&&\circ\ar@{-}[r]&\circ\ar@{-}[r]&\circ\ar@{-}[r]&\circ\ar@{-}[r]&\circ\ar@{-}[r]&\circ\ar@{-}[r]&\circ\ar@{-}[r]&\circ\ar@{-}[r]&
\circ\ar@{-}[r]&\circ\ar@{-}[r]&\circ\ar@{-}[r]&\circ&\\
&&&\circ\ar@{-}[r]&\circ\ar@{-}[r]&\circ\ar@{-}[r]&\circ\ar@{-}[r]&\circ\ar@{-}[r]&\circ\ar@{-}[r]&\circ\ar@{-}[r]&\circ\ar@{-}[r]&
\circ\ar@{-}[r]&\circ\ar@{-}[r]&\circ\ar@{-}[r]&\circ&\\
&&\circ\ar@{-}[r]&\circ\ar@{-}[r]&\circ\ar@{-}[r]&\circ\ar@{-}[r]&\circ\ar@{-}[r]&\circ\ar@{-}[r]&\circ\ar@{-}[r]&\circ\ar@{-}[r]&
\circ\ar@{-}[r]&\circ\ar@{-}[r]&\circ&\\
&&&&&&\circ\ar@{-}[r]&\circ\ar@{-}[r]&\circ\ar@{-}[r]&\circ\ar@{-}[r]&\circ\ar@{-}[r]&\circ&\\
&&&&&&&&\circ\ar@{-}[r]&\circ\ar@{-}[r]&\circ&\\
&&&&&&&\circ\ar@{-}[r]&\circ&\\
\circ\ar@{-}[r]&\circ\ar@{-}[r]&\circ\ar@{-}[r]&\circ\ar@{-}[r]&\circ\ar@{-}[r]&\circ\ar@{-}[r]&\circ\ar@{-}[r]&\circ&}
\]
Next is an example of the case $3{*}41{*}2$ with $k=8$ and $r=4$:
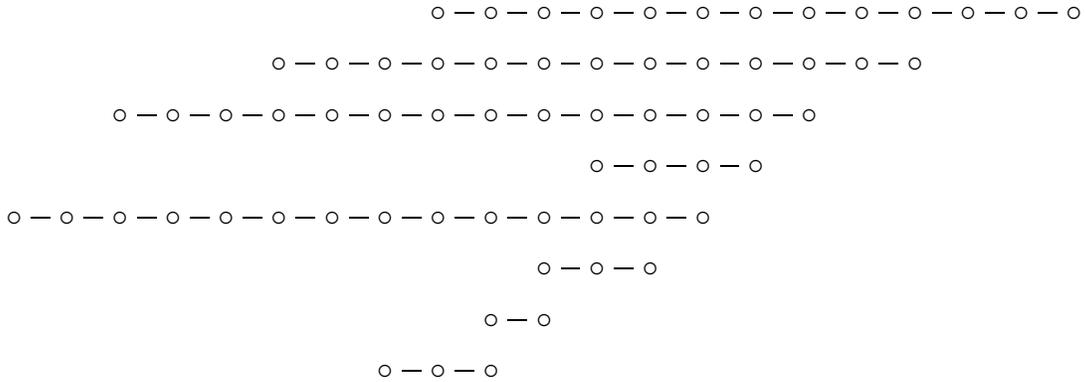
\[
\xymatrix@=0.6em{
&&&&&&&&\circ\ar@{-}[r]&\circ\ar@{-}[r]&\circ\ar@{-}[r]&\circ\ar@{-}[r]&\circ\ar@{-}[r]&\circ\ar@{-}[r]&\circ\ar@{-}[r]&\circ\ar@{-}[r]&\circ\ar@{-}[r]&
\circ\ar@{-}[r]&\circ\ar@{-}[r]&\circ\ar@{-}[r]&\circ\\
&&&&&\circ\ar@{-}[r]&\circ\ar@{-}[r]&\circ\ar@{-}[r]&\circ\ar@{-}[r]&\circ\ar@{-}[r]&\circ\ar@{-}[r]&\circ\ar@{-}[r]&\circ\ar@{-}[r]&\circ\ar@{-}[r]&
\circ\ar@{-}[r]&\circ\ar@{-}[r]&\circ\ar@{-}[r]&\circ\\
&&\circ\ar@{-}[r]&\circ\ar@{-}[r]&\circ\ar@{-}[r]&\circ\ar@{-}[r]&\circ\ar@{-}[r]&\circ\ar@{-}[r]&\circ\ar@{-}[r]&\circ\ar@{-}[r]&\circ\ar@{-}[r]&
\circ\ar@{-}[r]&\circ\ar@{-}[r]&\circ\ar@{-}[r]&\circ\ar@{-}[r]&\circ\\
&&&&&&&&&&&\circ\ar@{-}[r]&\circ\ar@{-}[r]&\circ\ar@{-}[r]&\circ\\
\circ\ar@{-}[r]&\circ\ar@{-}[r]&\circ\ar@{-}[r]&\circ\ar@{-}[r]&\circ\ar@{-}[r]&\circ\ar@{-}[r]&\circ\ar@{-}[r]&\circ\ar@{-}[r]&\circ\ar@{-}[r]&\circ\ar@{-}[r]&
\circ\ar@{-}[r]&\circ\ar@{-}[r]&\circ\ar@{-}[r]&\circ\\
&&&&&&&&&&\circ\ar@{-}[r]&\circ\ar@{-}[r]&\circ\\
&&&&&&&&&\circ\ar@{-}[r]&\circ\\
&&&&&&&\circ\ar@{-}[r]&\circ\ar@{-}[r]&\circ}
\]
Finally, here is an example of the case $34{*}12$ with $k=7$:
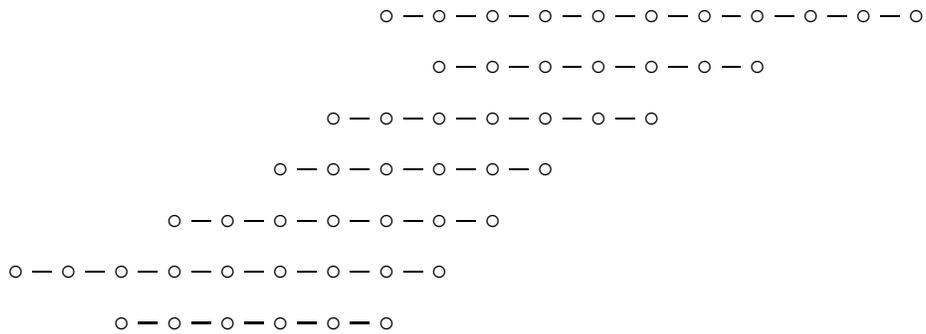
\[
\xymatrix@=0.6em{
&&&&&&&\circ\ar@{-}[r]&\circ\ar@{-}[r]&\circ\ar@{-}[r]&\circ\ar@{-}[r]&\circ\ar@{-}[r]&\circ\ar@{-}[r]&\circ\ar@{-}[r]&\circ\ar@{-}[r]&\circ\ar@{-}[r]&\circ\ar@{-}[r]&\circ\\
&&&&&&&&\circ\ar@{-}[r]&\circ\ar@{-}[r]&\circ\ar@{-}[r]&\circ\ar@{-}[r]&\circ\ar@{-}[r]&\circ\ar@{-}[r]&\circ\\
&&&&&&\circ\ar@{-}[r]&\circ\ar@{-}[r]&\circ\ar@{-}[r]&\circ\ar@{-}[r]&\circ\ar@{-}[r]&\circ\ar@{-}[r]&\circ\\
&&&&&\circ\ar@{-}[r]&\circ\ar@{-}[r]&\circ\ar@{-}[r]&\circ\ar@{-}[r]&\circ\ar@{-}[r]&\circ\\
&&&\circ\ar@{-}[r]&\circ\ar@{-}[r]&\circ\ar@{-}[r]&\circ\ar@{-}[r]&\circ\ar@{-}[r]&\circ\ar@{-}[r]&\circ\\
\circ\ar@{-}[r]&\circ\ar@{-}[r]&\circ\ar@{-}[r]&\circ\ar@{-}[r]&\circ\ar@{-}[r]&\circ\ar@{-}[r]&\circ\ar@{-}[r]&\circ\ar@{-}[r]&\circ\\
&&\circ\ar@{-}[r]&\circ\ar@{-}[r]&\circ\ar@{-}[r]&\circ\ar@{-}[r]&\circ\ar@{-}[r]&\circ}
\]

\begin{proof}
Let $S=\{i:\Delta_i\text{ is detachable in }\m\}$.
Note that since $\m$ is regular, $i\in S$ if and only if $\Delta_i\not\prec\Delta_j$ for all $j<i$
or $\Delta_j\not\prec\Delta_i$ for all $j>i$. In particular, $\{1,k\}\subset S$.
Moreover, by Proposition \ref{prop: nonsmth}, $\m$ is minimal unbalanced if and only if there exists a submultisegment
$\m_A=\sum_{i\in A}\Delta_i$, $A\subset\{1,\dots,k\}$ of $\m$ which is either of type $4231$ or $3412$ and
\begin{equation} \label{eq: equmin}
\text{for any $A$ such that $\m_A$ is of type $4231$ or $3412$ we have $A\supset S$}
\end{equation}
(and in particular, $1,k\in A$).

Consider the families above.
In the $4{*}23{*}1$ case $S=\{1,k\}$ and $\m_{\{1,r,r+1,\dots,k\}}$ forms a sub-multisegment of type $4231$.
On other hand, for any sub-multisegment $\m_A$ of type $4231$ we have $\{1,k\}\subset A$ and there is no sub-multisegment of type $3412$.
In the $3{*}41{*}2$ case $S=\{1,r,r+1,k\}$ and $\m_{\{1,r,r+1,\dots,k\}}$ is a sub-multisegment of type $3412$.
Any sub-multisegment of type $4231$ necessarily contains $S$ and there is no sub-multisegment of type $4231$.
In the $34{*}12$ case, $S=\{1,2,k-1,k\}$ and $\m_S$ is the unique sub-multisegment of type $3412$;
there is no sub-multisegment of type $4231$.
Thus in all cases $\m$ is a minimal unbalanced multisegment.
It is also clear that the three cases are disjoint.

Conversely, suppose that $\m$ is minimal unbalanced and let $\m_A$ be a sub-multisegment of type $4231$ or $3412$.
By the minimality assumption $1,k\in A$.
Let $j_{\min}$ (resp., $j_{\max}$) be the index $j$ for which $b(\Delta_j)$ is minimal (resp., maximal). Then $j_{\min},j_{\max}\in S\subset A$.

Assume first that $\m_A$ is of type $4231$. In this case $j_{\min}=k$ and $j_{\max}=1$.
In other words $b(\Delta_k)<b(\Delta_i)<b(\Delta_1)$ for all $1<i<k$.

Note that for any $i<k$ there exists $j>i$ such that $\Delta_j\prec\Delta_i$. Indeed, if $i\in A$, we can choose
$j\in A$ as well, while if $i\notin A$ then $i\notin S$ and the claim is clear.
It follows that for any $i$ there exists a sequence $i_0<\dots<i_m$, $m\ge0$ such that $i_0=i$, $i_m=k$
and $\Delta_{i_{j+1}}\prec\Delta_{i_j}$, $j=0,\dots,m-1$.
Similarly, for any $i$ there exists a sequence $i_0<\dots<i_m$, $m\ge0$ such that $i_0=1$, $i_m=i$
and $\Delta_{i_{j+1}}\prec\Delta_{i_j}$, $j=0,\dots,m-1$.

Next we show that we cannot have $b(\Delta_l)<b(\Delta_j)<b(\Delta_i)$ for any $1<i<l<j$.
Assume otherwise, and consider a counterexample with $j$ minimal and with $b(\Delta_i)$ minimal (with respect to $j$).
We first claim that $\Delta_j\prec\Delta_i$.
Indeed, by the above, there exists $s<j$ such that $\Delta_j\prec\Delta_s$.
If $b(\Delta_s)>b(\Delta_i)$ then $\Delta_j\prec\Delta_i$ as required.
Otherwise, $s>1$ and by the minimality of $j$ we have $s<l$ for otherwise we could replace $j$ by $s$.
However, this contradicts the minimality of $b(\Delta_i)$, since we can now replace $i$ by $s$.

Let $j_0,\dots,j_m$ be a sequence such that $j_0=j$, $j_m=k$ and $\Delta_{j_{l+1}}\prec\Delta_{j_l}$ for $l=0,\dots,m-1$.
Let $s>0$ be the first index such that $b(\Delta_{j_s})<b(\Delta_l)$.
Then $\m_{\{i,l,j_0,\dots,j_s\}}$ is a sub-multisegment of type $4231$, which repudiates \eqref{eq: equmin} and the assumption that $i>1$.

By passing to the contragredient we also conclude that we cannot have $b(\Delta_j)<b(\Delta_i)<b(\Delta_l)$ for any $i<l<j<k$.

Clearly, there exists some $i<k$ such that $\Delta_{i+1}\not\prec\Delta_i$. Let $r$ be the maximal such index.
We have $b(\Delta_{r+1})>b(\Delta_r)$ for otherwise $r\in S\setminus A$.
In particular $r<k-1$. Also by what we showed before we have $b(\Delta_{r+1})>b(\Delta_i)$ for all $1<i<r$.
Thus, $\Delta_{r+1}\prec\Delta_1$ for otherwise $r+1\in S\setminus A$. This concludes the case where $\m_A$ is of type $4231$.

Assume now that $\m_A$ is of type $3412$. Write $A=\{1,r,s,t,\dots,\}$ where $1<r<s<t<\dots$. In this case $j_{\min}=s$ and $j_{\max}=r$.
In particular, $S\supset\{1,r,s,k\}$.
We first claim that $\Delta_{i+1}\prec\Delta_i$ for all $i<r-1$. Assume on the contrary that $i$ is a minimal counterexample.
Since $i+1\notin A\supset S$ there exists $j<i$ such that $\Delta_{i+1}\prec\Delta_j$. Let $j$ be maximal with respect to this property.
Then $\m_{\{j,j+1,i+1,s\}}$ is of type $4231$, a contradiction to \eqref{eq: equmin}.

Moreover, if $r>2$ then $b(\Delta_2)<b(\Delta_k)$ for otherwise $\m_{A\cup\{2\}\setminus\{1\}}$ is a sub-multisegment of type $3412$
in violation of \eqref{eq: equmin}.

Consider first the case where $r=2$, $\#A=4$ (i.e., $t=k$) and $b(\Delta_k)<b(\Delta_3)<b(\Delta_1)$.
We claim that in this case we have $s=k-1$. Otherwise, $k-1\notin A\supset S$ and therefore $\Delta_k\prec\Delta_{k-1}$.
Necessarily $b(\Delta_{k-1})<b(\Delta_2)$ (since $j_{\max}=2$) and hence $\Delta_{k-1}\prec\Delta_2$ (since $\Delta_k\prec\Delta_2$).
If $b(\Delta_{k-1})>b(\Delta_3)$ then $\m_{\{2,3,k-1,k\}}$ is of type $4231$.
Otherwise $\m_{\{1,2,s,k-1\}}$ is of type $3412$. Both cases rebut \eqref{eq: equmin}. Hence $s=k-1$ as claimed.
Suppose that $i$ is such that $b(\Delta_1)<b(\Delta_i)<b(\Delta_2)$. Then $\m_{\{1,i,k-1,k\}}$
would be a sub-multisegment of type $3412$ which is excluded by \eqref{eq: equmin}.
By a similar reasoning we conclude that $b(\Delta_k)<b(\Delta_i)<b(\Delta_1)$ for all $2<i<k-1$.
Finally, we have $b(\Delta_{i+1})<b(\Delta_i)$ (and hence $\Delta_{i+1}\prec\Delta_i$) for all $2<i<k-2$ otherwise
$\Delta_{\{1,i,i+1,k-1\}}$ would be a sub-multisegment of type $4231$.
Thus, we are in the case $34{*}12$ of the lemma.

From now on we assume that $\#A>4$ or $r>2$ or $b(\Delta_k)>b(\Delta_3)$ or $b(\Delta_1)<b(\Delta_3)$.

We show that $r=s-1$. Assume on the contrary that $r<s-1$ and let $i$ be the index in $A\setminus\{1\}$ such that $b(\Delta_i)<b(\Delta_{r+1})$
and $b(\Delta_i)$ is maximal with respect to this property. If $i=s$ then $\m_{A\cup\{r+1\}\setminus\{s\}}$ is of type $3412$.
If $i=t$ and either $\#A>4$ or $\#A=4$ and $b(\Delta_1)<b(\Delta_{r+1})$ then $\m_{A\cup\{r+1\}\setminus\{r\}}$ is of type $3412$.
If $i=t$, $\#A=4$, $b(\Delta_{r+1})<b(\Delta_1)$ and $r>2$ then $\m_{\{1,2,r+1,s\}}$ is of type $4213$.
If $i>t$ then $\m_{\{r,r+1,t,\dots,i\}}$ is of type $4231$. All these cases lead to a contradiction to \eqref{eq: equmin}.
Thus $s=r+1$.

Assume that $A$ is a maximal subset of $\{1,\dots,k\}$ with respect to inclusion such that $\m_A$ is of type $3412$.
It remains to show that $A\supset\{s,\dots,k\}$. Assume on the contrary that this is not the case and let $s<j<k$ be the maximal element not in $A$.

Suppose first that $b(\Delta_j)<b(\Delta_{j+1})$. If $b(\Delta_j)<b(\Delta_k)$ then $j\in S\setminus A$ and we get a contradiction.
Otherwise, let $l$ be the first index $>j$ such that $b(\Delta_l)<b(\Delta_j)$. Let $j^-=r$ if $j=r+2$ and $j^-=j-1$ otherwise.
Then $\m_{\{j^-,j,j+1,\dots,l\}}$ is a multisegment of type $4231$, controverting \eqref{eq: equmin}.
Thus $b(\Delta_j)>b(\Delta_{j+1})$.

Assume now that $b(\Delta_j)>b(\Delta_{j^-})$. Since $j\notin A\supset S$ we must have $\Delta_{j+1}\prec\Delta_j$
and there exists $l$ such that $\Delta_j\prec\Delta_l$. Necessarily $r\le l<j^-$ and $l\ne s$. Now
$\m_{\{l,j^-,j,j+1\}}$ is of type $4231$, gainsaying \eqref{eq: equmin}.
Thus $b(\Delta_j)<b(\Delta_{j^-})$.

Since $\Delta_{j+1}\prec\Delta_{j^-}$ and $b(\Delta_{j+1})<b(\Delta_j)<b(\Delta_{j^-})$
we infer that $\Delta_j\prec\Delta_{j^-}$ and $\Delta_{j+1}\prec\Delta_j$.
By the maximality of $A$ we necessarily have $j=k-1$ and $b(\Delta_{k-1})<b(\Delta_1)$.
But then, $\m_{A\cup\{k-1\}\setminus\{k\}}$ is of type $3412$, denying \eqref{eq: equmin}.

This concludes the proof of the lemma.
\end{proof}

Next, we will use Corollary \ref{cor: derisLM} (and its terminology) to motivate the following definition.

\begin{definition}
Let $\m$ be a minimal unbalanced (regular) multisegment. We say that $\m$ is absolutely minimal unbalanced if no descendant $\m'$ of $\m$ is regular unbalanced.
\end{definition}

Recall that two segments $\Delta'\prec\Delta$ are juxtaposed if $e(\Delta')=b(\lshft{\Delta})$.
We say that the segments $\Delta_1,\dots,\Delta_k$ are back-to-back juxtaposed if $e(\Delta_{i+1})=b(\lshft\Delta_i)$ for all $i=1,\dots,k-1$.

\begin{lemma} \label{lem: absmincls}
Let $\m=\Delta_1+\dots+\Delta_k$ be an absolutely minimal unbalanced multisegment with $e(\Delta_1)>\dots>e(\Delta_k)$.
Then exactly one of the following conditions holds.
\begin{enumerate}
\item (case $4{*}23{*}1$) There exists $1<r<k-1$ such that
\begin{enumerate}
\item $\Delta_{i+1}=\lshft{\Delta}_i$ for all $1<i<r$.
\item $\Delta_1,\Delta_{r+1},\Delta_{r+2},\dots,\Delta_k$ are back-to-back juxtaposed.
\item $b(\Delta_k)=b(\lshft{\Delta}_r)$, $e(\Delta_2)=e(\lshft{\Delta}_1)$, $b(\Delta_2)=b(\lshft{\Delta}_{k-1})$, $e(\Delta_{r+1})=e(\lshft{\Delta}_r)$.
\end{enumerate}
\item (case $3{*}41{*}2$)
There exists $1<r<k-1$ such that if $\tau$ is the transposition $r\leftrightarrow r+1$ then
\begin{enumerate}
\item $\Delta_{\tau(r+1)},\Delta_{\tau(r+2)},\dots,\Delta_{\tau(k)}$ are back-to-back juxtaposed.
\item $b(\Delta_{\tau(2)})=b(\lshft{\Delta}_k)$, $b(\Delta_k)=b(\lshft{\Delta}_1)$, $b(\Delta_1)=b(\lshft{\Delta}_{\tau(k-1)})$
and $b(\Delta_{\tau(i+1)})=b(\lshft{\Delta}_{\tau(i)})$, $i=2,\dots,r-1$.
\item $e(\Delta_{i+1})=e(\lshft{\Delta}_i)$, $i=1,\dots,r+1$.
\end{enumerate}
\item (case $34{*}12$) $\m$ is of the form \eqref{eq: mincase34*12}.
\end{enumerate}
\end{lemma}

An example of the case $4{*}23{*}1$ with $k=8$ and $r=4$ is:
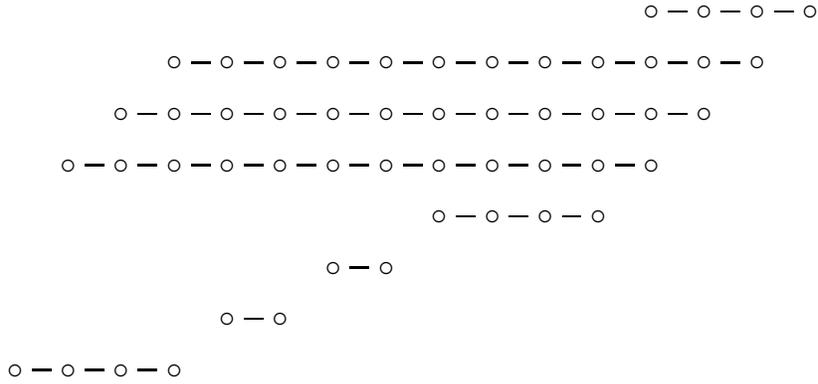
\[
\xymatrix@=0.6em{
&&&&&&&&&&&&&&\circ\ar@{-}[r]&\circ\ar@{-}[r]&\circ\ar@{-}[r]&\circ&\\
&&&&&\circ\ar@{-}[r]&\circ\ar@{-}[r]&\circ\ar@{-}[r]&\circ\ar@{-}[r]&\circ\ar@{-}[r]&\circ\ar@{-}[r]&\circ\ar@{-}[r]&\circ\ar@{-}[r]&
\circ\ar@{-}[r]&\circ\ar@{-}[r]&\circ\ar@{-}[r]&\circ&\\
&&&&\circ\ar@{-}[r]&\circ\ar@{-}[r]&\circ\ar@{-}[r]&\circ\ar@{-}[r]&\circ\ar@{-}[r]&\circ\ar@{-}[r]&\circ\ar@{-}[r]&\circ\ar@{-}[r]&\circ\ar@{-}[r]&\circ\ar@{-}[r]&
\circ\ar@{-}[r]&\circ&\\
&&&\circ\ar@{-}[r]&\circ\ar@{-}[r]&\circ\ar@{-}[r]&\circ\ar@{-}[r]&\circ\ar@{-}[r]&\circ\ar@{-}[r]&\circ\ar@{-}[r]&\circ\ar@{-}[r]&\circ\ar@{-}[r]&
\circ\ar@{-}[r]&\circ\ar@{-}[r]&\circ&\\
&&&&&&&&&&\circ\ar@{-}[r]&\circ\ar@{-}[r]&\circ\ar@{-}[r]&\circ&\\
&&&&&&&&\circ\ar@{-}[r]&\circ&\\
&&&&&&\circ\ar@{-}[r]&\circ&\\
&&\circ\ar@{-}[r]&\circ\ar@{-}[r]&\circ\ar@{-}[r]&\circ&}
\]
An example of the case $3{*}41{*}2$ with $k=8$ and $r=4$ is:
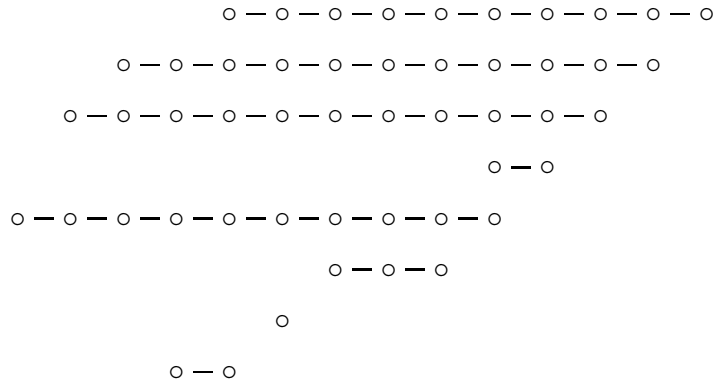
\[
\xymatrix@=0.6em{
&&&&\circ\ar@{-}[r]&\circ\ar@{-}[r]&\circ\ar@{-}[r]&\circ\ar@{-}[r]&\circ\ar@{-}[r]&\circ\ar@{-}[r]&\circ\ar@{-}[r]&\circ\ar@{-}[r]&\circ\ar@{-}[r]&\circ\\
&&\circ\ar@{-}[r]&\circ\ar@{-}[r]&\circ\ar@{-}[r]&\circ\ar@{-}[r]&\circ\ar@{-}[r]&\circ\ar@{-}[r]&\circ\ar@{-}[r]&\circ\ar@{-}[r]&\circ\ar@{-}[r]&\circ\ar@{-}[r]&\circ\\
&\circ\ar@{-}[r]&\circ\ar@{-}[r]&\circ\ar@{-}[r]&\circ\ar@{-}[r]&\circ\ar@{-}[r]&\circ\ar@{-}[r]&\circ\ar@{-}[r]&\circ\ar@{-}[r]&\circ\ar@{-}[r]&\circ\ar@{-}[r]&\circ\\
&&&&&&&&&\circ\ar@{-}[r]&\circ\\
\circ\ar@{-}[r]&\circ\ar@{-}[r]&\circ\ar@{-}[r]&\circ\ar@{-}[r]&\circ\ar@{-}[r]&\circ\ar@{-}[r]&\circ\ar@{-}[r]&\circ\ar@{-}[r]&\circ\ar@{-}[r]&\circ\\
&&&&&&\circ\ar@{-}[r]&\circ\ar@{-}[r]&\circ\\
&&&&&\circ\\
&&&\circ\ar@{-}[r]&\circ}
\]

\begin{remark}
One can show that the converse to the lemma holds as well, but we will not need this fact.
\end{remark}

\begin{proof}
We separate into the cases provided by Lemma \ref{lem: MU}.

Consider first the $4{*}23{*}1$ case. As before, let $r>1$ be the maximal index such that $\Delta_{r+1}\not\prec\Delta_r$.
If $\Delta_{i+1}\prec\Delta_i$ for some $1<i<r$ then $b(\Delta_{i+1})=b(\lshft{\Delta}_i)$, for otherwise $\lderiv_{b(\Delta_{i+1})}(\m)$
is regular unbalanced, in contradiction to the assumption on $\m$. Similarly, $e(\Delta_{i+1})=e(\lshft{\Delta}_i)$. Thus,
$\Delta_{i+1}=\lshft{\Delta}_i$ for all $1<i<r$ such that $\Delta_{i+1}\prec\Delta_i$.

Next, we show that $\Delta_1,\Delta_{r+1},\Delta_{r+2},\dots,\Delta_k$ are back-to-back juxtaposed.
If $e(\Delta_{r+1})\ne b(\lshft{\Delta}_1)$ then $\lderiv_{b(\Delta_1)}(\m)$ is regular unbalanced and we get a contradiction.
Suppose on the contrary that $e(\Delta_{i+1})\ne b(\lshft{\Delta}_i)$ for some $r<i<k$ and let $i$ be the minimal such index.
Then $\Delta_i$ is not a singleton, i.e. $b(\Delta_i)\ne e(\Delta_i)$ and by the minimality of $i$, this amounts to
$b(\Delta_i)\ne b(\lshft{\Delta}_j)$ where $j=1$ if $i=r+1$ and $j=i-1$ otherwise. Thus, $b(\Delta_i)\in\lnrset(\m)$ and
$\lderiv_{b(\Delta_i)}(\m)$ is regular unbalanced, a contradiction.

Suppose now that the set $\{i:\Delta_{i+1}\not\prec\Delta_i\}$ is not a singleton and let $s>1$ be the
penultimate element of this set. Then $\Delta_1+\Delta_s+\Delta_r+\Delta_k$ is a sub-multisegment of type $4231$.
For every $r<i<k$ $\Delta_i$ is a singleton for otherwise $b(\Delta_i)\in\lnrset(\m)$ and $\lderiv_{b(\Delta_i)}$ is regular
unbalanced, a contradiction. Write $\Delta_i=\{\rho_i\}$, $i=r+1,\dots,k-1$ and set $\rho_k=\lshft{\rho}_{k-1}$.
We have $b(\Delta_{s+1})=\rho_k$ since otherwise $b(\Delta_{s+1})\in\lnrset(\m)$ and $\lderiv_{b(\Delta_{s+1})}$ is regular unbalanced.
Let $\m_k=\m$ and define inductively $\m_{i-1}=\lderiv_{\rho_i}(\m_i)$, $i=k,\dots,r+1$.
It easily follows from Lemma \ref{lem: spclcaseextrho} that $\rho_i\in\lnrset(\m_i)$, $i=r+1,\dots,k$
and $\m_r$ is obtained from $\m$ by removing $\Delta_{r+1}$ and replacing $\Delta_{s+1}$ by $[\rho_{r+1},e(\Delta_{s+1})]$.
Thus, $\m_r$ is regular unbalanced and we get a contradiction. In conclusion $\Delta_{i+1}\prec\Delta_i$ for all $i<r$.

Finally, $e(\Delta_{j+1})=e(\lshft{\Delta}_j)$, $j=1,r$ for otherwise $\rderiv_{e(\Delta_j)}(\m)$ is regular unbalanced.
Similarly, $b(\Delta_k)=b(\lshft{\Delta}_r)$ and $b(\Delta_2)=b(\lshft{\Delta}_{k-1})$

In the $3{*}41{*}2$ case, we have $e(\Delta_{i+1})=e(\lshft{\Delta}_i)$ for all $i\le r+1$, otherwise $\rderiv_{e(\Delta_i)}(\m)$ is unbalanced.
By passing to the contragredient we get the analogous relations for the $b(\Delta_i)$'s.
Also, $e(\Delta_{r+2})=b(\lshft{\Delta}_r)$ for otherwise $\lderiv_{b(\Delta_r)}(\m)$ would be a regular unbalanced multisegment.
Suppose on the contrary that $e(\Delta_{i+1})\ne b(\lshft{\Delta}_i)$ for some $i>r+1$ and let $i$ be the minimal such index.
Then $\Delta_i$ is not a singleton, that is $b(\Delta_i)\ne e(\Delta_i)$. Equivalently, by the minimality of $i$,
$b(\Delta_i)\ne b(\lshft{\Delta}_{\tau(i-1)})$. Hence, $\lderiv_{b(\Delta_i)}(\m)$ is regular unbalanced
in contradiction with the absolute minimality assumption. In conclusion, $\Delta_{\tau(i)}$, $i=r+1,\dots,k$ are back-to-back juxtaposed.

In the $34{*}12$ case, we have $e(\Delta_{i+1})=e(\lshft{\Delta}_i)$ for all $1\le i\le k-1$ since otherwise $\rderiv_{e(\Delta_i)}(\m)$ is regular unbalanced.
Analogously, by passing to the contragredient, we have $b(\Delta_{i+1})=b(\lshft{\Delta_i})$ for all $2<i<k-2$,
$b(\rshft{\Delta}_{k-1})=b(\Delta_k)=b(\lshft{\Delta}_{k-2})$ and $b(\rshft{\Delta}_3)=b(\Delta_1)=b(\lshft{\Delta}_2)$.
Finally, $e(\Delta_k)=b(\lshft{\Delta}_2)$, for otherwise $\lderiv_{b(\Delta_2)}(\m)$ is regular unbalanced.
Thus, $\m$ is of the form \eqref{eq: mincase34*12}.
\end{proof}

\begin{corollary} \label{cor: absmin}
Suppose that $\m$ is absolutely minimal unbalanced multisegment. Then at least one of the following conditions holds.
\begin{enumerate}
\item $\m$ is of the form \eqref{eq: basic4231}, \eqref{eq: basic3412} or \eqref{eq: mincase34*12}.
\item $\m$ is contractible.
\item $\m^\#$ is regular unbalanced but not minimal unbalanced.
\end{enumerate}
\end{corollary}

\begin{proof}
This is trivial if $\m$ is of type $34{*}12$.

If $\m$ is of type $3{*}41{*}2$ then $\m$ is contractible unless every $\Delta_i$, $r+1<i<k$ is a singleton,
in which case $\m$ is of the form \eqref{eq: basic3412}.

Finally, if $\m$ is of type $4{*}23{*}1$ then once again, $\m$ is contractible unless every $\Delta_i$, $r<i<k$ is a singleton, in which case
\begin{equation} \label{eq: minproof}
\m=[k,k+r-1]+\speh{[r,k+r-2]}{r-1}+\speh{[k-1]}{k-r-1}+[1,r].
\end{equation}
It is then easy to see from the recipe of $\m^\#$ (\S\ref{sec: zeleinvo}) that
\[
\m^\#=\speh{[k,k+r-1]}{r-1}+[r,k]+\speh{[k-r+1,k-1]}{k-2r+1}+\speh{[r-1,2r-2]}{r-1}
\]
if $k\ge 2r$ and
\[
\m^\#=\speh{[k,k+r-1]}{k-r-1}+\speh{[r,2r]}{2r-k}+[r+1,k]+\speh{[k-r,k-1]}{k-r}
\]
otherwise. Thus, $\m^\#$ is regular unbalanced but upon removing its last segment we remain with an unbalanced multisegment unless $r=2$
in which case $\m$ is of the form \eqref{eq: basic4231}.
The result follows.
\end{proof}

Here is a drawing for $\m$ as in \eqref{eq: minproof} for $k=8$ and $r=4$:
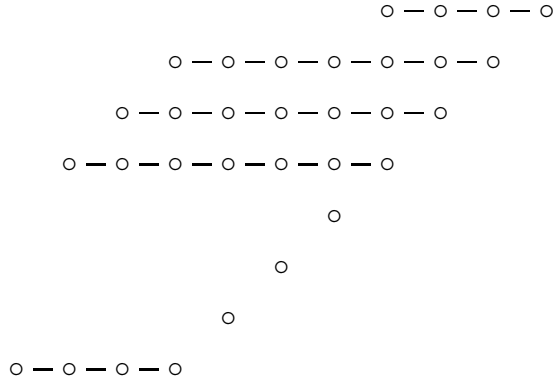
\[
\xymatrix@=0.6em{
&&&&&&&\circ\ar@{-}[r]&\circ\ar@{-}[r]&\circ\ar@{-}[r]&\circ&\\
&&&\circ\ar@{-}[r]&\circ\ar@{-}[r]&\circ\ar@{-}[r]&\circ\ar@{-}[r]&\circ\ar@{-}[r]&\circ\ar@{-}[r]&\circ&\\
&&\circ\ar@{-}[r]&\circ\ar@{-}[r]&\circ\ar@{-}[r]&\circ\ar@{-}[r]&\circ\ar@{-}[r]&\circ\ar@{-}[r]&\circ&\\
&\circ\ar@{-}[r]&\circ\ar@{-}[r]&\circ\ar@{-}[r]&\circ\ar@{-}[r]&\circ\ar@{-}[r]&\circ\ar@{-}[r]&\circ&\\
&&&&&&\circ&\\
&&&&&\circ&\\
&&&&\circ&\\
\circ\ar@{-}[r]&\circ\ar@{-}[r]&\circ\ar@{-}[r]&\circ&}
\]
$\m^\#$ is given by
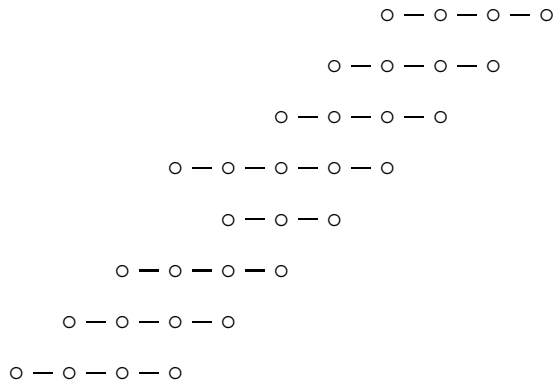
\[
\xymatrix@=0.6em{
&&&&&&&\circ\ar@{-}[r]&\circ\ar@{-}[r]&\circ\ar@{-}[r]&\circ&\\
&&&&&&\circ\ar@{-}[r]&\circ\ar@{-}[r]&\circ\ar@{-}[r]&\circ&\\
&&&&&\circ\ar@{-}[r]&\circ\ar@{-}[r]&\circ\ar@{-}[r]&\circ&\\
&&&\circ\ar@{-}[r]&\circ\ar@{-}[r]&\circ\ar@{-}[r]&\circ\ar@{-}[r]&\circ&\\
&&&&\circ\ar@{-}[r]&\circ\ar@{-}[r]&\circ&\\
&&\circ\ar@{-}[r]&\circ\ar@{-}[r]&\circ\ar@{-}[r]&\circ&\\
&\circ\ar@{-}[r]&\circ\ar@{-}[r]&\circ\ar@{-}[r]&\circ&\\
\circ\ar@{-}[r]&\circ\ar@{-}[r]&\circ\ar@{-}[r]&\circ&}
\]

Finally, we can prove the remaining parts of Theorem \ref{thm: main}, namely that if $\m$ is an unbalanced multisegment then $\zele{\m}$ is not \LM\ and $\m$ is not \GLS.

Indeed, assume on the contrary that $\m$ is an unbalanced multisegment with minimal $\deg\m$ such that $\pi=\zele{\m}$ is \LM.
In view of Lemma \ref{lem: 1stred}, Corollary \ref{cor: derisLM}, Proposition \ref{prop: contract} and Remark \ref{rem: contrblncd},
the minimality of $\deg\m$ implies that $\m$ is absolutely minimal unbalanced and not contractible.
Moreover, by Proposition \ref{prop: ZIred} if $\m^\#$ is regular then it is necessarily minimal unbalanced.
By Corollary \ref{cor: absmin} $\pi$ is therefore one of the representations considered in Propositions
\ref{prop: basictype4231}, \ref{prop: 3412basic} and \ref{prop: lastcase3412} of the last section.
These propositions yield the required contradiction.

By a similar reasoning, using Lemma \ref{lem: lderivm} and Remarks \ref{rem: contractGLS}, \ref{rem: GLS^t} and \ref{rem: GLScomb}
no unbalanced multisegment can be \GLS.

\begin{remark}
The use of Propositions \ref{prop: contract} (whose proof depends on the material of the next section) is not indispensable.
The ideas of \S\ref{sec: basicases} work slightly more generally for all multisegments listed in Lemma \ref{lem: absmincls}.
However, the additional reduction alleviates the bookkeeping. Similarly, the use of the Zelevinsky involution is not essential.
\end{remark}

\begin{remark} \label{rem: lwrbndN}
In principle, one can explicate the argument of this section to weaken the lower bound on $N$ stated in the converse part of Corollary \ref{cor: mainq}.
However, we will not pursue this matter here.
\end{remark}

\section{An identity of Kazhdan--Lusztig polynomials} \label{sec: KLid}
Using the Arakawa--Suzuki equivalence \cite{MR1652134} we may reinterpret Theorem \ref{thm: main} in terms of the
Kazhdan--Lusztig polynomials for the symmetric group $S_{2k}$ (Corollary \ref{cor: KLidnt}).

\subsection{The  Arakawa--Suzuki functor}
We sketch the setup, referring the reader to \cite{MR2320806} and \cite{MR3495794} and the references therein for more details.
Consider the category $\CO$ with respect to $\mathfrak{gl}_k$.
For any $\mu\in\Z^k$ let $M(\mu)$ (resp., $L(\mu)$) be the Verma (resp., simple) module with highest weight $\mu$.
Suppose that $\mu=(\mu_1,\dots,\mu_k)\in\Z^k$ with $\mu_1\ge\dots\ge\mu_k$ and let $S_\mu$ be the stabilizer of $\mu$ in $S_k$,
a parabolic subgroup of $S_k$.
As is well-known, for any $\mu'\in\Z^k$ and $w\in S_k$, $L(\mu')$ occurs in $\JH(M(w\mu))$ if and only if $\mu'$
is of the form $w'\mu$ with $w'\ge w$ in the Bruhat order of $S_k$.
In the latter case, if we take $w'$ to be of maximal length in its coset $w'S_\mu$ then the multiplicity of $L(\mu)$ in $\JH(M(w\mu))$
is $P_{w,w'}(1)$ where $P_{w,w'}(q)$ denotes the Kazhdan--Lusztig polynomial with respect to $S_k$ (\cite{MR560412, MR610137, MR632980, MR1237825, MR1802178}).
In other words, denoting by $\grimg{\cdot}$ the image of an object of a locally finite abelian category in its Grothendieck group, we have
\begin{equation} \label{eq: KLconj}
\grimg{M(w\mu)}=\sum_{w'\in S_k:w'\text{ of maximal length in }w'S_\mu}P_{w,w'}(1)\grimg{L(w'\mu)}.
\end{equation}
Equivalently, for any $w\in S_k$ of maximal length in $wS_\mu$ we have
\[
\grimg{L(w\mu)}=\sum_{w'\in S_k}\sgn ww'\ P_{w'w_0,ww_0}(1)\grimg{M(w'\mu)}
\]
where $w_0$ is the longest element of $S_k$. Recall that $P_{w,w'}\equiv0$ unless $w\le w'$.

Fix $\lambda=(\lambda_1,\dots,\lambda_k)\in\Z^k$ with $\lambda_1\ge\dots\ge\lambda_k$.
For any integer $l\ge0$ let $F_{\lambda,l}$ be the exact functor of Arakawa--Suzuki from category $\CO$ to the category of
finite-dimensional representations of the graded affine Hecke algebra $\GH_l$ of $\GL_l$ (\cite{MR1652134}).
Let $\chi$ be an integral infinitesimal character of the center $\mathfrak{z}$ of the universal enveloping algebra of $\mathfrak{gl}_k$
and let $\CO_\chi$ be the full subcategory of $\CO$ on which $\mathfrak{z}$ acts by $\chi$.
There is at most one $l$ for which $F_{\lambda,l}$ is non-zero on $\CO_{\chi}$.
For this $l$ (if exists) $Z(\GH_l)$ acts by an integral character $\chi'$ (depending on $\chi$) on the image of $F_{\lambda,l}$.
Denote by $J_\chi$ the maximal ideal of $Z(\GH_l)$ corresponding to $\chi'$ (i.e., which annihilates the image of $F_{\lambda,l}$).
Let now $\IH_l$ be the Iwahori--Hecke algebra of $\GL_l(F)$.
The category of finite-dimensional representations of $\IH_l$ is equivalent to the category $\Reps_I(\GL_l(F))$ of finite-length representations of $\GL_l(F)$
which are generated by the vectors which are fixed under the Iwahori subgroup.
To $\chi'$ corresponds a character $\tilde\chi$ of $Z(\IH_l)$.
Let $J_{\tilde\chi}$ be the corresponding maximal ideal of $Z(\IH_l)$.
Then the algebras $\GH_l/J_\chi\GH_l$ and $\IH_l/J_{\tilde\chi}\IH_l$ are isomorphic \cite{MR991016}.
Thus, we may view $F_{\lambda,l}$ as an exact functor from $\CO_{\chi}$ to the full subcategory of $\Reps_I(\GL_l(F))$ on which $Z(\IH_l)$ acts by $\tilde\chi$.
We will omit $\chi$ from the notation since it will be generally clear from the context.

Taking $D=F$ and $\rho$ to be the trivial one-dimensional character of $\GL_1(F)=F^*$, the functor $F_{\lambda,l}$ satisfies
\[
F_{\lambda,l}(M(\mu))=\std(\m_{\mu,\lambda})\text{ and }
F_{\lambda,l}(L(\mu))=\begin{cases}\zele{\m_{\mu,\lambda}}&
\text{if $\mu_i\le\mu_{i+1}$ whenever $\lambda_i=\lambda_{i+1}$,}\\0&\text{otherwise,}\end{cases}
\]
where $\m_{\mu,\lambda}=\sum_{i=1}^k[\mu_i,\lambda_i]$ and $l=\sum_{i=1}^k(\lambda_i-\mu_i+1)$.
(Recall the notational convention \eqref{eq: convention}.)

Let $\mu=(\mu_1,\dots,\mu_k)\in\Z^k$ with $\mu_1\ge\dots\ge\mu_k$.
Note that $\m_{w\mu,\lambda}$ depends only on the double coset $S_\lambda wS_\mu$ of $w$.
Moreover, $\m_{w\mu,\lambda}\obt\m_{w'\mu,\lambda}$
(i.e., $\zele{\m_{w'\mu,\lambda}}$ occurs in $\JH(\std(\m_{w\mu,\lambda}))$) if and only if
$S_\lambda wS_\mu\le S_\lambda w'S_\mu$ with respect to the partial order on the double coset set $S_\lambda\bs S_k/S_\mu$ induced by
the Bruhat order of $S_k$ (by passing to the representatives of minimal length).

Applying $F_{\lambda,l}$ to \eqref{eq: KLconj} we get that for any $\mu=(\mu_1,\dots,\mu_k)\in\Z^k$ with $\mu_1\ge\dots\ge\mu_k$ and $w\in S_k$ we have
\[
\grimg{\std(\m_{w\mu,\lambda})}=\sum_{w'\in S_k:w'\text{ is of maximal length in }S_\lambda w'S_\mu}P_{w,w'}(1)\grimg{\zele{\m_{w'\mu,\lambda}}}.
\]
Equivalently, for $w$ of maximal length in $S_\lambda wS_\mu$ we have
\begin{equation} \label{eq: invgen}
\grimg{\zele{\m_{w\mu,\lambda}}}=\sum_{w'\in S_k}\sgn w'w\ P_{w'w_0,ww_0}(1)\ \grimg{\std(\m_{w'\mu,\lambda})}.
\end{equation}
In other words, the coefficients of $\zele{\m_{w\mu,\lambda}}$ in the basis
$\std(\m_{w'\mu,\lambda})$, $w'\in S_\lambda\bs S_k/S_\mu$ (ignoring $0$ terms) are
\[
\sgn w\sum_{x\in S_\lambda w'S_\mu}\sgn x\ P_{xw_0,ww_0}(1),\ \ w'\in S_\lambda\bs S_k/S_\mu.
\]
By Theorem \ref{thm: indepcspline}, this relation holds for arbitrary $\rho\in\Cusp$ and $D$.

Going back to the setup of \S\ref{sec: smth pairs} and \S\ref{sec: combi} we infer:
\begin{corollary} \label{cor: detsmth}
Let $\tseq=\bitmplt$ be a \biseq, $\sigma_0=\sigma_0(\tseq)$ and let $\sigma\in S_k$ be such that $\sigma(i)<\sigma(i+1)$ whenever $a_i=a_{i+1}$
and $\sigma^{-1}(i)<\sigma^{-1}(i+1)$ whenever $b_i=b_{i+1}$. Then
\begin{equation} \label{eq: geninv}
\grimg{\zele{\m_{\sigma}(\tseq)}}=\sum_{\sigma'\in [\sigma_0,\sigma]}\sgn \sigma'\sigma\ P_{\sigma',\sigma}(1)\ \grimg{\std(\m_{\sigma'}(\tseq))}.
\end{equation}
In particular, if $(\sigma,\sigma_0)$ is a smooth pair then
\[
\grimg{\zele{\m_{\sigma}(\tseq)}}=\sum_{\sigma'\in [\sigma_0,\sigma]}\sgn \sigma'\sigma\ \grimg{\std(\m_{\sigma'}(\tseq))}.
\]
The converse also holds in the case where $\tseq$ is regular.
\end{corollary}

This follows from \eqref{eq: invgen} by taking $\lambda=(b_1,\dots,b_k)$, $\mu=(a_k,\dots,a_1)$, $w=\sigma w_0$ and noting that
$\m_{\sigma}(\tseq)=\m_{\sigma w_0\mu,\lambda}$.

\begin{remark}
The corollary suggests that if $(\sigma,\sigma_0)$ is a smooth pair then the semisimplification of the Jacquet module of $\zele{\m_{\sigma}(\tseq)}$
is relatively easy to compute. (See \cite{MR2996769} for a special case.)
In view of Theorem \ref{thm: main} this is in accordance with Conjectures 1 and 2 of \cite{LecChev}.
\end{remark}

Next, we go back to Proposition \ref{prop: contract}.
Let $f_b,f_e:\Z\rightarrow\Z$ be the strictly monotone maps
\[
f_b(n)=\begin{cases}n+1&\text{if }n>0,\\n&\text{otherwise,}\end{cases}\ \ \
f_e(n)=\begin{cases}n+1&\text{if }n\ge0,\\n&\text{otherwise.}\end{cases}
\]
Note that $f_b(x)=f_e(y)+1$ if and only if $x=y+1$ so that
\begin{equation} \label{eq: aleb+1}
\text{$x\le y+1$ if and only if $f_b(x)\le f_e(y)+1$.}
\end{equation}
Define an injective endofunction $f$ on the set of segments by $f([x,y])=[f_b(x),f_e(y)]$.
We extend $f$ by additivity to an injective endomorphism (also denoted by $f$) of $\Mult_\rho$.
On the other hand, $f$ also defines an injective (non-graded) ring endomorphism $\phi$ of $\Gr_\rho$
determined by $\phi(\grimg{\zele{[a,b]}})=\grimg{\zele{f([a,b])}}$ for any segment $[a,b]$.
Thus $\phi(\grimg{\std(\m)})=\grimg{\std(f(\m))}$ for any $\m\in\Mult_\rho$.

\begin{corollary} \label{cor: fzele}
Under the notation above we have $\phi(\grimg{\zele{\m}})=\grimg{\zele{f(\m)}}$ for any $\m\in\Mult_\rho$, i.e., $\phi$ preserves irreducibles.
In particular, if $\m_1,\m_2\in\Mult_\rho$ then $\zele{\m_1}\times\zele{\m_2}$ is irreducible if and only if $\zele{f(\m_1)}\times\zele{f(\m_2)}$ is irreducible.
\end{corollary}

\begin{proof}
Given $\tseq=\bitmplt$ let $f(\tseq)=\begin{pmatrix}f_b(a_1)&\dots&f_b(a_k)\\f_e(b_1)&\dots&f_e(b_k)\end{pmatrix}$.
By \eqref{eq: aleb+1} $\sigma_0(f(\tseq))=\sigma_0(\tseq)$ and $f(\m_\sigma(\tseq))=\m_\sigma(f(\tseq))$ for any $\sigma\in S_k$.
Therefore, $\phi(\grimg{\std(\m_\sigma(\tseq))})=\grimg{\std(\m_\sigma(f(\tseq)))}$ for any $\sigma\in S_k$.
It follows from \eqref{eq: geninv} that $\phi(\grimg{\zele{\m_\sigma(\tseq)}})=\grimg{\zele{\m_\sigma(f(\tseq))}}=\grimg{\zele{f(\m_\sigma(\tseq))}}$.
The corollary follows.
\end{proof}

\begin{remark}
It would be interesting to have a more functorial proof of Corollary \ref{cor: fzele}.
\end{remark}



\begin{remark} \label{rem: maini}
Now that Proposition \ref{prop: contract} is proved,
Corollary \ref{cor: detsmth}, together with the statement \eqref{eq: redrelsmth} provides the last missing part of
Theorem \ref{thm: maini} of the introduction. (See the discussion following Theorem \ref{thm: main}.)
\end{remark}

\subsection{}
Let $H$ be the parabolic subgroup of $S_{2k}$
\[
H=\{w\in S_{2k}:\{w(2i-1),w(2i)\}=\{2i-1,2i\}\ \forall i\}\simeq S_2^k.
\]
As is well known, the map
\[
w\mapsto M_w=\#(w(\{2i-1,2i\})\cap\{2j-1,2j\})_{i,j=1,\dots,k}
\]
is bi-$H$-invariant and defines a bijection between $H\bs S_{2k}/H$
and the set $\Matk$ of $k\times k$ matrices with entries in $\{0,1,2\}$ such that the sum of the entries in each row and each column is $2$.
In turn, by the Birkhoff--von-Neumann theorem, these are precisely the matrices that can be written as the sum of two
$k\times k$ permutation matrices. (We will write $\permat{\sigma}$ for the permutation matrix corresponding to $\sigma\in S_k$.)
The corresponding permutations in $S_k$ (say $\sigma_1$, $\sigma_2$) are not uniquely determined (even up to interchanging).
However, the conjugacy class of $\sigma_2^{-1}\sigma_1$ in $S_k$, which will be denote by $[w]$, is uniquely determined by the double coset.
More precisely, we have the following.

\begin{lemma} \label{lem: RHS}
For any $M\in\Matk$ let $C_1,\dots,C_s$ be the equivalence classes for the equivalence relation
generated by $i\sim j$ if there exists $l$ such that $M_{i,l}=M_{j,l}=1$.
Then the set
\[
\{(\sigma_1,\sigma_2)\in S_k\times S_k:\permat{\sigma_1}+\permat{\sigma_2}=M\}
\]
has cardinality $2^r$ where $r=\{i:\abs{C_i}>1\}$.
Moreover, for any $(\sigma_1,\sigma_2)\in S_k\times S_k$ such that $\permat{\sigma_1}+\permat{\sigma_2}=M$,
the cycles of $\sigma_2^{-1}\sigma_1$ are the $C_i$'s. In particular,
the conjugacy class of $\sigma_2^{-1}\sigma_1$ in $S_k$ is determined by $M$ only.
\end{lemma}


\begin{proof}
The symmetric $k\times k$ matrix $MM^t-2I_k$ has non-negative integer entries and the sum along each row and column is two.
Therefore, it is the adjacency matrix of an 
undirected $2$-regular graph $G$ (possibly containing loops and double edges), with vertex set $\{1,\dots,k\}$.\footnote{As usual,
a loop counts twice for the degree of a vertex.}
Hence, the connected components of $G$ are cycles (including loops and $2$-cycles) and their underlying vertex sets are the $C_i$'s.
Note that the loops in $G$ correspond to the indices $i$ for which there exists $l$ such that $M_{i,l}=2$, while the $2$-cycles in $G$
(i.e., the double edges) correspond to the pairs of indices $i\ne j$ for which there exist $l\ne m$ such that $M_{i,l}=M_{j,l}=M_{i,m}=M_{j,m}=1$.
If $M=\permat{\sigma_1}+\permat{\sigma_2}$ then the edges of $G$ (counted with multiplicities) are given by $\{\sigma_1(i),\sigma_2(i)\}$, $i=1,\dots,k$.
Moreover, any such presentation gives rise to an orientation of $G$, given by $\sigma_1(i)\rightarrow\sigma_2(i)$ such that $G$ is the union
of directed cycles, i.e., such that the indegree and the outdegree of each vertex is one. Conversely, any such orientation
arises from a presentation $M=\permat{\sigma_1}+\permat{\sigma_2}$ where $\sigma_1$ and $\sigma_2$ are uniquely determined
and the cycles of $\sigma_2^{-1}\sigma_1$ are the $C_i$'s.
Clearly, the number of such orientations is $2^r$ where $r$ is the number of non-trivial connected components of $G$.
\end{proof}

Given a \biseq\ $\tseq=\bitmplt$ we write $\tilde{\tseq}$ for the duplicated \biseq\
$\begin{pmatrix}a_1&a_1&\dots&a_k&a_k\\b_1&b_1&\dots&b_k&b_k\end{pmatrix}$ of length $2k$.
Similarly, for any $\sigma\in S_k$ we write $\tilde\sigma$ for the permutation in $S_{2k}$ given by
$\tilde\sigma(2i-j)=2\sigma(i)-j$, $i=1,\dots,k$, $j=0,1$.
Clearly, $\tilde\sigma$ normalizes the subgroup $H$ of $S_{2k}$.
It easily follows from \eqref{def: sigma0} that
\begin{equation} \label{eq: dbltseq}
\sigma_0(\tilde\tseq)=\widetilde{\sigma_0(\tseq)}.
\end{equation}

Let $\iota:S_k\times S_k\rightarrow S_{2k}$ be the embedding
\[
\iota(\sigma_1,\sigma_2)(2(i-1)+j)=2(\sigma_j(i)-1)+j,\ \ i=1,\dots,k,\ j=1,2.
\]
In particular, $\iota(\sigma,\sigma)=\tilde\sigma$.
Clearly, if $\sigma_1'\le\sigma_1$ and $\sigma_2'\le\sigma_2$ then $\iota(\sigma_1',\sigma_2')\le\iota(\sigma_1,\sigma_2)$.
Also, for any $w\in S_{2k}$ and $\sigma_1,\sigma_2\in S_k$
\begin{equation} \label{eq: iotasum}
\iota(\sigma_1,\sigma_2)\in HwH\text{ if and only if }\permat{\sigma_1}+\permat{\sigma_2}=M_w.
\end{equation}

For any $\sigma\in S_k$ let $\rk_\sigma:\{1,\dots,k\}^2\rightarrow\Z_{\ge0}$ be the rank function
\[
\rk_\sigma(i,j)=\#\{u=1,\dots,i:\sigma(u)\le j\}.
\]

It is well known that for any $\sigma,\tau\in S_k$ we have $\tau\le\sigma$ if and only if $\rk_\sigma\le\rk_\tau$ on $\{1,\dots,k\}^2$.
We will use the following combinatorial result.

\begin{proposition}[\cite{1710.06115}] \label{prop: tight}
Let $(\sigma,\sigma_0)$ be a smooth pair and let $\tau\in S_k$.
Suppose that $\rk_\tau(i,j)=\rk_\sigma(i,j)$ for all $(i,j)\in\{1,\dots,k\}^2$ such that $\rk_{\sigma_0}(i,j)=\rk_\sigma(i,j)$.
Then $\tau\le\sigma$.
\end{proposition}

Note that for $\sigma_0=\id$, i.e., when $\sigma$ itself is smooth, this is a classical result. (See \cite{MR1934291} and the references therein.)
We obtain the following consequence.

\begin{corollary} \label{cor: dblsame}
Suppose that $(\sigma,\sigma_0)$ is a smooth pair. Let $\sigma_1,\sigma_2\in S_k$ be such that
$\sigma_0\le\sigma_1,\sigma_2$ and $H\iota(\sigma_1,\sigma_2)H\le H\tilde\sigma$. Then $\sigma_1,\sigma_2\le\sigma$.
\end{corollary}

\begin{proof}
Indeed, the condition $H\iota(\sigma_1,\sigma_2)H\le H\tilde\sigma$ means that
\[
\rk_{\sigma_1}(i,j)+\rk_{\sigma_2}(i,j)\ge 2\rk_{\sigma}(i,j),\ \ i,j=1,\dots,k.
\]
On the other hand, $\rk_{\sigma_1}(i,j),\rk_{\sigma_2}(i,j)\le\rk_{\sigma_0}(i,j)$ since $\sigma_0\le\sigma_1,\sigma_2$.
Hence, whenever $\rk_{\sigma_0}(i,j)=\rk_\sigma(i,j)$ we also have $\rk_{\sigma_1}(i,j)=\rk_{\sigma_2}(i,j)=\rk_{\sigma}(i,j)$.
By Proposition \ref{prop: tight} we conclude that $\sigma_1,\sigma_2\le\sigma$ as required.
\end{proof}

Let $\clsf$ be the class function on $S_k$ given by $\clsf(\sigma)=\sgn\sigma\ 2^r$ where $r$ is the number of non-trivial cycles of $\sigma$.
We now interpret Theorem \ref{thm: main} in terms of an identity of Kazhdan--Lusztig polynomials.

\begin{corollary} (of Theorem \ref{thm: main}) \label{cor: KLidnt}
Suppose that $(\sigma,\sigma_0)$ is a smooth pair with $\sigma_0$ $213$-avoiding. Then for any $x\in [\widetilde{\sigma_0},\tilde\sigma]$ we have
\begin{equation} \label{eq: allxrltn}
\sum_{w\in HxH}\sgn w\ P_{w,\widetilde{\sigma}}(1)=\clsf([x]).
\end{equation}
In particular,
\begin{equation} \label{eq: parkl}
\sum_{w\in H}\sgn w\ P_{\widetilde{\sigma'}w,\widetilde{\sigma}}(1)=1.
\end{equation}
for any $\sigma'\in[\sigma_0,\sigma]$.
\end{corollary}

\begin{proof}
Let $\tseq$ be a regular \biseq\ such that $\sigma_0=\sigma_0(\tseq)$ (see Lemma \ref{lem: comb12}) and let $\m=\m_\sigma(\tseq)$.
By Theorem \ref{thm: main} $\zele{\m}\times\zele{\m}$ is irreducible, i.e., $\zele{\m}\times\zele{\m}=\zele{\m+\m}$.
We will deduce the corollary from Corollary \ref{cor: detsmth} by computing the coefficient of $\std(\m_x(\tilde\tseq))$ in the expansion of
$\zele{\m}\times\zele{\m}=\zele{\m+\m}$ in terms of standard modules in two different ways.

On the one hand,
\[
\grimg{\zele{\m}}=\sum_{\sigma'\in S_k:\sigma'\le\sigma}\sgn\sigma\sigma'\ \grimg{\std(\m_{\sigma'}(\tseq))}
\]
where of course only the terms $\sigma'\ge\sigma_0$ give a non-zero contribution.
Note that for any $\sigma_1,\sigma_2\in S_k$ we have
\[
\std(\m_{\sigma_1}(\tseq))\times\std(\m_{\sigma_2}(\tseq))=\std(\m_{\sigma_1}(\tseq)+\m_{\sigma_2}(\tseq))=
\std(\m_{\iota(\sigma_1,\sigma_2)}(\tilde\tseq))
\]
and this is non-zero (i.e., by \eqref{eq: dbltseq}, $\iota(\sigma_1,\sigma_2)\ge\widetilde{\sigma_0}$) if and only if $\sigma_0\le\sigma_1,\sigma_2$.
Thus,
\begin{multline*}
\grimg{\zele{\m}\times\zele{\m}}=\sum_{\sigma_1,\sigma_2\in S_k:\sigma_1,\sigma_2\le\sigma}\sgn\sigma_1\sigma_2\ \grimg{\std(\m_{\sigma_1}(\tseq)+\m_{\sigma_2}(\tseq))}\\=
\sum_{\sigma_1,\sigma_2\in S_k:\sigma_1,\sigma_2\le\sigma}\sgn\sigma_1\sigma_2\ \grimg{\std(\m_{\iota(\sigma_1,\sigma_2)}(\tilde\tseq))}.
\end{multline*}
On the other hand, by \eqref{eq: geninv} we have
\[
\grimg{\zele{\m+\m}}=\sum_{w\in S_{2k}}\sgn w\ P_{w,\tilde\sigma}(1)\ \grimg{\std(\m_w(\tilde\tseq))}.
\]
Comparing coefficients, for any $x\in H\bs S_{2k}/H$ such that $x\ge\widetilde{\sigma_0}$ we get
\begin{equation} \label{eq: intprtirred}
\sum_{w\in HxH}\sgn w\ P_{w,\tilde\sigma}(1)=\sum_{\sigma_1,\sigma_2\in S_k:\sigma_1,\sigma_2\le\sigma\text{ and }\iota(\sigma_1,\sigma_2)\in HxH}\sgn\sigma_1\sigma_2.
\end{equation}
Recall that $\iota(\sigma_1,\sigma_2)\ge\widetilde{\sigma_0}$ (or equivalently, $H\iota(\sigma_1,\sigma_2)H\ge\widetilde{\sigma_0}H$)
if and only if $\sigma_0\le\sigma_1,\sigma_2$.
Thus, by Corollary \ref{cor: dblsame}, if $x\in [\widetilde{\sigma_0},\tilde\sigma]$ then the condition $\sigma_1,\sigma_2\le\sigma$ on the
right-hand side of \eqref{eq: intprtirred} is superfluous.
Hence, by Lemma \ref{lem: RHS} and \eqref{eq: iotasum} the right-hand side of \eqref{eq: intprtirred} is $\clsf([x])$, proving our claim.
\end{proof}

\begin{remark}
The (computer-assisted) example $\sigma=(4231)$, $\sigma_0=(1324)$ shows that the condition that $\sigma_0$ is $213$-avoiding is essential
for the relation \eqref{eq: allxrltn}.
On the other hand, in \cite{1705.06517} we conjecture among other things that for any smooth pair $(\sigma,\sigma_0)$ we have
\[
\sum_{w\in H}\sgn w\ P_{\widetilde{\sigma_0}w,\widetilde{\sigma}}(q)=q^{\ell(\sigma)-\ell(\sigma_0)}
\]
(and in particular \eqref{eq: parkl} holds)
and prove it in the case where $\sigma$ is a product of distinct simple reflexions (i.e., a Coxeter element in a parabolic subgroup of $S_n$).
We also remark that for the relation \eqref{eq: parkl} (assuming $\sigma_0$ is $213$-avoiding) we haven't used the result of \cite{1710.06115}
since Corollary \ref{cor: dblsame} is trivial if $H\iota(\sigma_1,\sigma_2)H=H\widetilde{\sigma'}$
(in which case $\sigma_1=\sigma_2=\sigma'$).
\end{remark}



\subsection{}
More generally, let $m>1$ and consider the parabolic subgroup $H\simeq S_m\times\dots\times S_m$ of $S_{mk}$ of type $(m,\dots,m)$
($k$ times) and the subgroup $K\simeq S_k\times\dots\times S_k$ ($m$ times) of $S_{mk}$ given by
\[
K=\{\sigma\in S_{mk}:\sigma(i)\equiv i\pmod m\text{ for all }i\}.
\]
Thus, $H\cap K=1$ and the normalizer of $H$ is $H\rtimes\{\tilde\sigma:\sigma\in S_k\}$ where as before
\[
\tilde\sigma(mi-j)=m\sigma(i)-j,\ i=1,\dots,k,\ j=0,\dots,m-1.
\]
We have
\begin{theorem} \label{thm: higherKL}
Suppose that $(\sigma,\sigma_0)$ is a smooth pair with $\sigma_0$ $213$-avoiding. Then
\begin{equation} \label{eq: gencaseirredcoseq}
\sum_{w\in HxH}\sgn w\ P_{w,\tilde\sigma}(1)=\sum_{\tau\in HxH\cap K}\sgn\tau
\end{equation}
for any $x\in[\widetilde{\sigma_0},\tilde\sigma]$. In particular,
\begin{equation} \label{eq: spclcasesigma'}
\sum_{w\in H}\sgn w\ P_{\widetilde{\sigma'}w,\widetilde{\sigma}}(1)=1
\end{equation}
for any $\sigma'\in[\sigma_0,\sigma]$.
\end{theorem}
Indeed, if $\tseq$ is a regular \biseq\ such that $\sigma_0=\sigma_0(\tseq)$ and $\m=\m_\sigma(\tseq)$ then as in the proof of Corollary \ref{cor: KLidnt}
(using an obvious analog of Corollary \ref{cor: dblsame}), the left-hand side (resp., right-hand side) of
\eqref{eq: gencaseirredcoseq} is $\sgn\tilde\sigma$ times the coefficient of $\std(\m_x(\tilde\tseq))$ in the expansion of
$\zele{m\cdot\m}$ (resp., $\zele{\m}^{\times m}$) in terms of standard modules.
The theorem therefore follows from Theorem \ref{thm: main} and Corollary \ref{cor: pi1pi2LM} which imply that $\zele{\m}^{\times m}=\zele{m\cdot\m}$.
(Once again, for \eqref{eq: spclcasesigma'} we do not need to use \cite{1710.06115}.)

Note that as before, the double cosets $H\bs S_{mk}/H$ correspond to matrices of size $k\times k$ with non-negative integer entries,
whose sums along each row and each column are all equal to $m$. Each such matrix can be written as a sum of $m$ permutation matrices.
Thus, $HxH\cap K\ne\emptyset$ for all $x\in S_{mk}$.
However, for $m>2$ it is no longer true that $\sgn$ is constant on $HxH\cap K$.
In fact, the right-hand side of \eqref{eq: gencaseirredcoseq} is much more mysterious for $m>2$.
For instance, in the case where $m=k$ and the double coset $HxH$ corresponds to the matrix all of whose entries are 1,
the right-hand side of \eqref{eq: gencaseirredcoseq} is $(-1)^{m\choose 2}$ times the difference $\partial_m$ between the number
of even and odd Latin squares of size $m\times m$. Clearly $\partial_m=0$ if $m$ is odd but it is still an open question,
known as the Alon--Tarsi conjecture, whether $\partial_m\ne0$ for all even $m$ \cite{MR1179249}.
This conjecture is related to other problems in linear algebra. Some progress on it was made by Janssen, Drisko, Zappa and others \cite{MR1309160, MR1451417, MR1624999, MR1453404}.
In particular, $\partial_m\ge0$. An upper bound for $\partial_m$ was given by Alpoge \cite{MR3638338}.

\def\cprime{$'$}
\providecommand{\bysame}{\leavevmode\hbox to3em{\hrulefill}\thinspace}
\providecommand{\MR}{\relax\ifhmode\unskip\space\fi MR }

\end{document}